\documentclass[11pt]{amsart}

\newcommand{\A}{\mathbb{A}}

\newcommand{\GG}{\mathbb{G}}
\newcommand{\NN}{\mathbb{N}}

\newcommand{\ZZ}{\mathbb{Z}}

\newcommand{\cA}{\mathcal{A}}
\newcommand{\cB}{\mathcal{B}}
\newcommand{\cC}{\mathcal{C}}
\newcommand{\ccD}{\mathcal{D}}

\newcommand{\cG}{\mathcal{G}}
\newcommand{\cH}{\mathcal{H}}

\newcommand{\cN}{\mathcal{N}}

\newcommand{\cP}{\mathcal{P}}

\newcommand{\cU}{\mathcal{U}}
\newcommand{\cV}{\mathcal{V}}

\newcommand{\fA}{\mathfrak{A}}
\newcommand{\fb}{\mathfrak{b}}
\newcommand{\fE}{\mathfrak{E}}

\newcommand{\fg}{\mathfrak{g}}
\newcommand{\fh}{\mathfrak{h}}

\newcommand{\fsl}{\mathfrak{sl}}

\newcommand{\fu}{\mathfrak{u}}

\newcommand{\fZ}{\mathfrak{Z}}

\newcommand{\dact}{\boldsymbol{.}}
\newcommand{\lra}{\longrightarrow}

\DeclareMathOperator{\ad}{ad}
\DeclareMathOperator{\Ad}{Ad}
\DeclareMathOperator{\Char}{char}
\DeclareMathOperator{\cx}{cx}

\DeclareMathOperator{\End}{End}
\DeclareMathOperator{\Ext}{Ext}

\DeclareMathOperator{\HH}{H}

\DeclareMathOperator{\Hom}{Hom}
\DeclareMathOperator{\im}{im}
\DeclareMathOperator{\id}{id}
\DeclareMathOperator{\Jt}{Jt}

\DeclareMathOperator{\Lie}{Lie}
\DeclareMathOperator{\modd}{mod}
\DeclareMathOperator{\Pt}{{\Pi}t}
\DeclareMathOperator{\ql}{q\ell}
\DeclareMathOperator{\Rad}{Rad}
\DeclareMathOperator{\rk}{rk}
\DeclareMathOperator{\res}{res}
\DeclareMathOperator{\srk}{srk}
\DeclareMathOperator{\SL}{SL}
\DeclareMathOperator{\StJt}{StJt}
\DeclareMathOperator{\Soc}{Soc}
\DeclareMathOperator{\supp}{supp}
\DeclareMathOperator{\Top}{Top}
\DeclareMathOperator{\tr}{tr}

\DeclareMathOperator{\Proj}{Proj}

\usepackage{pictex}
\usepackage{amscd}
\usepackage{latexsym}
\usepackage{amssymb}
\usepackage{eucal}
\usepackage{eufrak}
\usepackage{verbatim}

\numberwithin{equation}{section}

\newtheorem{Theorem}{Theorem}[section]
\newtheorem{Lemma}[Theorem]{Lemma}
\newtheorem{Corollary}[Theorem]{Corollary}
\newtheorem{Proposition}[Theorem]{Proposition}

\theoremstyle{Theorem}
\newtheorem{Thm}{Theorem}[subsection]
\newtheorem{Lem}[Thm]{Lemma}
\newtheorem{Prop}[Thm]{Proposition}
\newtheorem{Cor}[Thm]{Corollary}

\theoremstyle{remark}
\newtheorem*{Remark}{Remark}
\newtheorem*{Remarks}{Remarks}
\newtheorem*{Definition}{Definition}

\newtheorem*{Examples}{Examples}

\numberwithin{equation}{section}

\begin{document}

\title[Jordan types]{Jordan Types for indecomposable Modules of Finite Group Schemes}

\author[R. Farnsteiner]{Rolf Farnsteiner}

\address{Mathematisches Seminar, Christian-Albrechts-Universit\"at zu Kiel, Ludewig-Meyn-Str.4, 24098 Kiel, Germany}
\email{rolf@math.uni-kiel.de}
\thanks{Supported by the D.F.G. priority program SPP1388 `Darstellungstheorie'.}
\date{\today}

\makeatletter
\makeatother

\begin{abstract}
In this article we study the interplay between algebro-geometric notions related to $\pi$-points and structural features of the stable Auslander-Reiten quiver of a finite group scheme. We
show that $\pi$-points give rise to a number of new invariants of the AR-quiver on one hand, and exploit combinatorial properties of AR-components to obtain information on $\pi$-points on the other. Special attention is given to components containing Carlson modules, constantly supported modules, and endo-trivial modules.
\end{abstract}

\maketitle

\setcounter{section}{-1}\label{S:Int}

\section{Introduction}
Let $\cG$ be a finite group scheme over an algebraically closed field $k$ of characteristic $p>0$. In their recent articles \cite{FPe1,FPe2}, the authors have expounded the theories of
$p$-points and $\pi$-points, which generalize and unify various concepts of rank varieties defined earlier. One new feature arising via this vantage point is the notion of Jordan type \cite{FPS},
providing new invariants for $\cG$-modules that are finer than those given by rank varieties and support varieties. The ramifications of the additional information encoded in Jordan types are 
only beginning to be understood, even for the class of modules of constant Jordan type, \cite{CFP,CF}.

Our investigations focus on two aspects that underscore the utility of this approach. Sections \ref{S:SF} through \ref{S:CRP} are concerned with the behavior of Jordan types on the 
components of the stable Auslander-Reiten quiver $\Gamma_s(\cG)$ of the algebra of measures $k\cG$ of $\cG$. The results of the first three sections can be generalized to perfect fields, 
while those of Section \ref{S:CRP} rest on $k$ being algebraically closed. The last four sections employ $\pi$-points to define and study certain classes of modules and to determine their 
position within the AR-quiver. 

Support varieties have played an important r\^ole in the investigation of the quiver $\Gamma_s(\cG)$. It therefore seems expedient to explore the interplay between Jordan types and the 
Auslander-Reiten theory of the group scheme $\cG$.  Indeed, Jordan types turn out to provide new invariants for AR-components, enabling us to discern differences between components that cannot be detected via their support varieties. On the other hand, the Jordan types of all modules belonging to an Auslander-Reiten component may often be computed from the corresponding information of one single vertex.

In dealing with the AR-quiver, our main tools are subadditive functions on stable representation quivers, whose relevant properties are provided in Section \ref{S:SF}. As we show in Section 
\ref{S:AFP}, $\pi$-points of $\cG$ define various additive functions on the so-called locally split components of the stable Auslander-Reiten quiver. By definition, the pull-backs of the almost 
split sequences of such components along any $\pi$-point are split exact. With the exception of some infinite tubes, all infinite AR-components are locally split, so that our methods usually 
apply whenever the representation-theoretic support of the ambient block of $k\cG$ has dimension at least $2$. For these components we establish new invariants, all of which arise via 
$\pi$-points. In particular, we discuss the number of Jordan types of a module, the dominance order on $\pi$-points associated to an AR-component, and varieties of non-maximal supports.

For infinite tubes that are not locally split, the relevant functions are eventually additive, with their departure from additivity being controlled by the Cartan matrix $A_{p-1}$ and the structure 
of certain induced modules modules that are defined by $\pi$-points corresponding to closed points of the support scheme $\Pi(\cG)$. As we show in Section \ref{S:CRP}, the structure of 
these components as well as the associated functions are completely understood in the cases where the group scheme $\cG$ is trigonalizable, or $\cG$ is reduced and the relevant component 
contains a module with a cyclic vertex.

Applications concerning three closely related classes of $\cG$-modules, constantly supported modules, Carlson modules and endo-trivial modules, are the subject of the following three
sections. In terms of the hierarchy given by their sets of Jordan types, constantly supported modules naturally follow the modules of constant Jordan type that were investigated in
\cite{CFP,CF}.  In Section \ref{S:CSM} we show that examples are given by direct summands of the Carlson modules $L_\zeta$, associated to non-nilpotent homogeneous elements of the 
even cohomology ring $\HH^\bullet(\cG,k)$. This motivates the study of the $L_\zeta$ in the following section. Aside from their fundamental importance for theoretical purposes, Carlson 
modules are also of interest because their structure is closely reflected by their support varieties. For Carlson modules belonging to homogeneous non-nilpotent elements, we provide a 
sufficient condition for their indecomposability. In classical contexts, such as the first Frobenius kernels of semi-simple algebraic groups, it follows that these modules are usually quasi-simple. 
In particular, the almost split sequences originating in these $L_\zeta$ have an indecomposable middle term. If $\zeta \in \HH^{2n}(\cG,k)\setminus\{0\}$ is nilpotent, then the module
$L_\zeta$ is indecomposable, of constant Jordan type and usually quasi-simple. For elements of odd degree, the set of Jordan types of $L_\zeta$ may have two elements and we show that 
$L_\zeta$ is indecomposable, provided $\cG$ possesses sufficiently many abelian unipotent subgroups of complexity $\ge 2$. Quillen's dimension theorem implies that this condition is 
superfluous whenever $\cG$ corresponds to a finite group. 

Endo-trivial modules were introduced by Dade \cite{Da1, Da2}, who showed that, for abelian $p$-groups, these modules are stably isomorphic to the syzygies of the trivial module. In our 
context, endo-trivial modules are examples of modules of constant Jordan type that occur in connection with decomposable Carlson modules of non-nilpotent type. In Section 6, we investigate 
AR-components containing endo-trivial modules and discuss Carlson's construction \cite{Ca2} of such modules from the perspective of $\pi$-points.

In the final section, we augment recent results of Carlson-Friedlander \cite{CF} by studying the Jordan types of certain finite algebraic groups of tame representation type. In particular, we 
determine the Jordan types of the indecomposable modules of the restricted enveloping algebra $U_0(\fsl(2))$ and classify its indecomposable endo-trivial modules.

\bigskip

\section{Subadditive Functions on Stable Translation Quivers}\label{S:SF}
Throughout this section, we shall be considering quivers $\Gamma := (\Gamma_0,\Gamma_1)$ without loops or multiple arrows. For such a quiver $\Gamma$, the set of arrows is given by a
subset $\Gamma_1 \subseteq \Gamma_0 \times \Gamma_0$ of the Cartesian product of the set $\Gamma_0$ of vertices. We recall a few basic facts and definitions, the interested reader may
consult \cite{Be1,Ri, HPR1, HPR2} for further details.

Given a vertex $x \in \Gamma_0$, we put
\[ x^+ := \{y \in \Gamma_0 \ ; \ (x,y) \in \Gamma_1\} \ \ \ \ ; \ \ \ \ x^- := \{y \in \Gamma_0 \ ; \ (y,x) \in \Gamma_1\},\]
so that $x^+$ and $x^-$ are the sets of successors and predecessors of the vertex $x$, respectively. The quiver $(\Gamma_0,\Gamma_1)$ is referred to as {\it locally finite} if $x^+ \cup
x^-$ is finite for every $x \in \Gamma_0$. {\it Henceforth all quivers are assumed to be locally finite}.

Let $\Gamma := (\Gamma_0, \Gamma_1)$ be a quiver. A {\it valuation} of $\Gamma$ is a map $ \nu : \Gamma_0\times \Gamma_0 \lra \NN_0 \times \NN_0$ such that
$\Gamma_1 = \nu^{-1}(\NN\times \NN)$. The triple $(\Gamma_0, \Gamma_1, \nu)$ is then referred to as a {\it valued quiver}. Homomorphisms of valued quivers are defined canonically.

An automorphism $\tau : \Gamma \lra \Gamma$ is called a {\it translation} if
\[ a^- = \tau (a)^+ \ \ \ \ \ \ \forall \ a \in \Gamma_0. \]
We then refer to $(\Gamma_0, \Gamma_1,\tau)$ as a {\it stable translation quiver}.

If $\nu$ is a valuation of $(\Gamma_0,\Gamma_1)$ and $\tau$ is a translation of $(\Gamma_0, \Gamma_1)$ such that
\[ \nu(\tau(b),a) = \Delta (\nu(a,b)) \  \ \ \ \forall \  (a,b) \in \Gamma_1, \]
where $\Delta (m,n) = (n,m) \ \ \ \forall \ (m,n) \in \NN \times \NN$, then $(\Gamma_0, \Gamma_1, \nu , \tau)$ is a {\it valued stable translation quiver}.

By work of Riedtmann \cite[Struktursatz]{Ri} (see also \cite[(4.15)]{Be1}), a connected stable translation quiver $\Gamma$ is of the form
\[ \Gamma \cong \ZZ[T_\Gamma]/G,\]
where $T_\Gamma$ is a directed tree and $G \subseteq {\rm Aut}(\ZZ[T_\Gamma])$ is an admissible group. The isomorphism class of $\Gamma$ is determined by the associated undirected tree $\bar{T}_\Gamma$, the so-called {\it tree class} of $\Gamma$. We refer the reader to \cite[(4.15.6)]{Be1} for further details and just recall that a subgroup $G \subseteq {\rm Aut}(\Gamma)$ is {\it admissible} if $|G\cdot x \cap(\{y\} \cup y^+)| \leq 1$ and $|G \cdot x \cap(\{y\} \cup y^-)| \leq 1$ for all elements $x,y \in \Gamma_0$.

The aforementioned result rests on the following construction of the orbit quiver $\Gamma/G$ associated to an admissible subgroup $G
\subseteq {\rm Aut}(\Gamma)$:

\begin{itemize}
\item By defining $(\Gamma/G)_1 := \{([x],[y]) \in \Gamma_0/G \times \Gamma_0/G \ ; \ \exists \ a \in [x],\ b \in [y] \ \text{with} \ (a,b) \in \Gamma_1 \}$, we endow $\Gamma/G$ with the structure of a quiver.

\item The map $\bar{\nu} : \Gamma_0/G \times \Gamma_0/G \lra \NN_0 \times \NN_0$, given by $\bar{\nu}([x],[y]) = \nu(a,b)$ if $a \in [x]$ and $b \in [y]$ are such that $(a,b) \in \Gamma_1$, is a valuation, with the canonical map $\pi : \Gamma \lra \Gamma/G$ being a morphism of valued quivers.

\item The map $\bar{\tau} : \Gamma_0/G \lra \Gamma_0/G \ \ ; \ \ \bar{\tau}([x]) = [\tau(x)]$ is a translation such that $\pi$ is a morphism of valued stable translation quivers.
\end{itemize}

\bigskip
\noindent
The presence of subadditive functions on valued stable translation quivers significantly restricts the possible tree classes. Let ${\rm pr}_1:
\NN_0\times \NN_0 \lra \NN_0$ be the projection onto the first coordinate.

\bigskip

\begin{Definition} Let $(\Gamma_0, \Gamma_1, \nu, \tau)$ be a  valued stable translation quiver. A function $f : \Gamma_0 \lra \NN_0$ is said to be {\it subadditive} if
\[ f(y) + f(\tau(y)) \geq \sum_{x \in y^-} f(x) \, {\rm pr}_1(\nu(x,y)) \ \ \ \ \forall \ y \in \Gamma_0. \]
The function $f$ is referred to as {\it additive} if we have equality for every $y \in \Gamma_0$.\end{Definition}

\bigskip
\noindent
For future reference we record the following basic properties:

\bigskip

\begin{Proposition} \label{SF1} Let  $\Gamma := (\Gamma_0, \Gamma_1, \nu, \tau)$ be a connected valued stable translation quiver.

{\rm (1)} \ If $f : \Gamma_0 \lra \NN_0$ is a subadditive function with $f \circ \tau = f$. Then either $f = 0$, or $f(x) > 0$ for every $x \in \Gamma_0$.

{\rm (2)} \ Let $f_\Gamma : \Gamma_0 \lra \NN$ be an additive function with $f_\Gamma \circ \tau = f_\Gamma$ and such that every other such function is an integral multiple of $f_\Gamma$. Then $f_\Gamma \circ g = f_\Gamma$ for every $g \in {\rm Aut}(\Gamma)$.\end{Proposition}

\begin{proof} (1) We put $X := \{ x \in \Gamma_0 \ ; \ f(x) \neq 0 \} $ and $Y := \{ x \in \Gamma_0 \  ; \ f(x) = 0 \} $. It readily follows that $X$ and $Y$ are $\tau$-invariant. Let $x \in X$, $y \in Y$. If $x \in y^+$, then $\tau(x) \in \tau(y)^+ = y^-$. Thus, the existence of an arrow between $X$ and $Y$ implies that there are $x_1 \in X$ and $y_1 \in Y$ such that $x_1 \in y_1^-$. However, this yields
\[ 0 = 2 \, f(y_1) \geq \sum_{x \in y_1^-} f(x) \, {\rm pr}_1(\nu(x,y_1)), \]
so that $f(x_1) = 0$, a contradiction. Since $\Gamma$ is connected, it follows that $X = \Gamma_0$ or $Y = \Gamma_0$.

(2) Let $\cA$ be the set of additive functions $f: \Gamma_0 \lra \NN$ satisfying $f\circ \tau = f$. By assumption, we have $\cA = \NN f_\Gamma$. The group ${\rm Aut}(\Gamma)$ of automorphisms of the stable valued translation quiver $\Gamma$ acts on $\cA$ via
\[ (g\dact f)(x) := f(g^{-1}(x)) \ \ \ \ \forall \ g \in {\rm Aut}(\Gamma),\ f \in \cA, \ x \in \Gamma_0.\]
Given $g \in {\rm Aut}(\Gamma)$, there exists $n(g) \in \NN$ with $g\dact f_\Gamma = n(g)f_\Gamma$. Consequently,
\[ f_\Gamma = g^{-1}\dact(g\dact f_\Gamma) = n(g^{-1})n(g)f_\Gamma,\]
implying $n(g)=1 \ \ \forall \ g \in G$.\end{proof}

\bigskip
\noindent
We recall that a {\it valued graph} is a pair $(I,d)$, consisting of a set $I$ and a map $d : I \times I \longrightarrow \NN_0$, such that

(1) \ $d(i,i) = 0 \ \ \ \ \forall \ i \in I$,

(2) \ $d(i,j) \neq 0 \ \Leftrightarrow \ d(j,i) \neq 0$,

(3) \ for each $i \in I$, the set $A(i) := \{j \in I \ ; \ d(i,j) \neq 0\}$ is finite.

\bigskip
\noindent
One can equally well consider the {\it Cartan matrix} $C(i,j) := 2\delta_{ij}-d(i,j)$, cf.\ \cite{HPR1}. In the graphical presentation two vertices $i$ and $j$ are linked by a bond if $d(i,j) \ne 0$, and we endow this bond with the valuation
\[ i \stackrel{(d(i,j),d(j,i))}{\hrulefill} j.\]

\bigskip
\noindent
Let $T = (T_0,T_1)$ be a quiver. We define a stable translation quiver $\ZZ[T]$ by letting $\ZZ\times T_0$ be the set of vertices. We have arrows
\[ (n,s) \rightarrow (n,t) \ \ \text{and} \ \ (n,t) \rightarrow (n\!+\!1,s)  \ \ \ \ \forall \ n \in \ZZ\]
for every arrow $s \rightarrow t$ in $T$. The translation is given by
\[ \tau : \ZZ[T] \lra \ZZ[T] \ \ ; \ \ (n,t) \mapsto (n\!-\!1,t).\]
Below is the stable representation quiver $\ZZ[A_\infty]$, where $A_\infty$ has $\NN$ as set of vertices and arrows $n \rightarrow n\!+\!1$ for
every $n \in \NN$. The dotted arrows represent the translation.

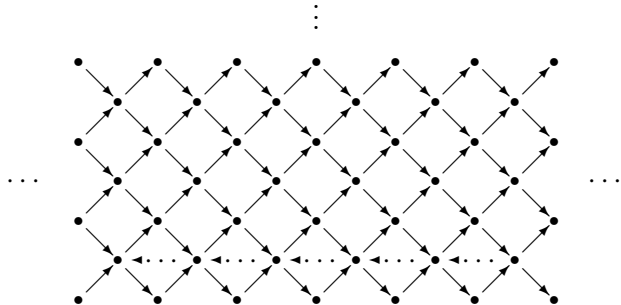
\begin{center}
\begin{picture}(200,130)
\multiput(0,0)(0,30)4{\multiput(0,0)(30,0)7{\circle*{3}}}
\multiput(15,15)(0,30)3{\multiput(0,0)(30,0)6{\circle*{3}}}
\multiput(3,3)(0,30)3{\multiput(0,0)(30,0)6{\vector(1,1){10}}}
\multiput(18,18)(0,30)3{\multiput(0,0)(30,0)6{\vector(1,1){10}}}
\multiput(3,27)(0,30)3{\multiput(0,0)(30,0)6{\vector(1,-1){10}}}
\multiput(18,12)(0,30)3{\multiput(0,0)(30,0)6{\vector(1,-1){10}}}
\put(90,110){\makebox(0,0){$\vdots$}}
\put(200,45){\makebox(0,0){$\cdots$}}
\put(-20,45){\makebox(0,0){$\cdots$}}
\multiput(20,15)(30,0)5{\vector(-1,0){0}}
\multiput(25,12)(30,0)5{$\cdots$}
\end{picture}
\end{center}

\bigskip
\noindent
In the sequel, we let $\tilde{A}_{p,q}$ be the quiver, whose underlying graph is the circle with $p+q$ vertices and with $p$ consecutive clockwise oriented arrows and $q$ consecutive counter-clockwise oriented arrows. Thus, $\tilde{A}_{2,0} \cong \tilde{A}_{0,2}$ is an oriented $2$-cycle.

\bigskip

\begin{Lemma} \label{SF2} Let $T$ be a quiver that does not contain a quiver of type $\tilde{A}_{2,0}$. Then $\langle \tau \rangle$ is an admissible subgroup of the valued stable translation quiver $\ZZ [T]$.  \end{Lemma}

\begin{proof} We put $G := \langle \tau \rangle$ and note that $G \cdot (n,x) = G \cdot (0,x) \ \ \forall \ (n,x) \in \ZZ [T]$. Let $(n_1, z_1), (n_2, z_2)
\in G \cdot (0,x) \cap (\{(m,y)\} \cup (m,y)^+)$. Since $(n_1, z_1), (n_2, z_2)\in G \cdot (0,x)$, we have $z_1 = z_2 = x$. Two cases arise:

\noindent
(a) \ $(n_1, x), (n_2, x) \in (m,y)^+$

Then we either have $n_1 = m$ and $y \rightarrow x$, or $n_1 = m\!+\!1$ and $x \rightarrow y$. Since $T$ does not contain $\tilde{A}_{2,0}$, the former alternative implies $n_2 =m$, while the latter yields $n_2 = m+1$. In either case, we arrive at $(n_1,x) = (n_2,x)$, as desired.

\noindent
(b) \ $(n_1,x) = (m,y)$

Since there is no arrow from $x$ to $x$, this readily implies $(n_2,x) = (m,y)$.

\noindent
The inequality $|G \cdot (0,x) \cap(\{(m,y)\} \cup (m,y)^-)| \leq 1$ follows analogously. \end{proof}

\bigskip

\begin{Definition} Let $(I,d)$ be a valued graph. A function $f : I \lra \NN_0$ is called {\it subadditive} if
\[ 2f(j) \geq \sum_{i \in I} f(i)d(i,j) \ \text{for every} \ j \in I.\]
We say that $f$ is {\it additive} if equality holds for every $j \in I$.\end{Definition}

\bigskip

\begin{Lemma} \label{SF3} Let $\Gamma = (\Gamma_0, \Gamma_1, \nu, \tau)$ be a valued stable translation quiver such that $\langle \tau \rangle \subseteq {\rm Aut}(\Gamma)$ is admissible. Then the following statements hold:

{\rm (1)} \ The function $d : \Gamma_0/\langle \tau \rangle \times \Gamma_0/\langle \tau \rangle \lra \NN_0$, given by
\[ d([x],[y]) := \left\{ \begin{array}{cc} {\rm pr}_1 (\bar{\nu}([x],[y])) &  {\rm if} \ [x]  \rightarrow [y] \\ 0 & {\rm otherwise}, \end{array} \right. \]
endows $\Gamma_0/ \langle \tau \rangle$ with the structure of a valued graph.

{\rm (2)} \ If $\Gamma$ is connected and $f : \Gamma \lra \NN_0$ is a non-zero additive function with $f \circ \tau = f$, then there exists an additive function $\varphi : \Gamma/\langle \tau \rangle \lra \NN$ such that $\varphi([x]) = f(x)$ for every $x \in \Gamma_0$.

{\rm (3)} \ If $\varphi : \Gamma/\langle \tau \rangle \lra \NN$ is an additive function, then
\[ f : \Gamma \lra \NN \ \ ; \ \ x \mapsto \varphi([x])\]
is an additive function on $\Gamma$ such that $f \circ \tau = f$. \end{Lemma}

\begin{proof} (1) Since $\langle \tau \rangle$ is admissible, we have $d([x],[x]) = 0$ for all $[x] \in \Gamma_0/\langle \tau \rangle$. Let $[x] \in \Gamma_0/\langle \tau \rangle$ be
given. Since the quiver $\Gamma/\langle \tau \rangle$ is locally finite, it follows that $d([x],[y]) \neq 0$ for only finitely many $[y] \in \Gamma_0/\langle \tau \rangle$. Finally, we assume
that $d([x],[y]) \neq 0$. Then $[x] \rightarrow [y]$, and by definition of $\Gamma/\langle \tau \rangle$ there exist $a \in [x], \ b \in [y]$ such that $ a \rightarrow b$. Thus,
$a \in b^- = \tau(b)^+$, so that $\tau(b) \rightarrow a$. This implies $[y] \rightarrow [x]$ and $d([y],[x]) \neq 0$.

(2) In view of Proposition \ref{SF1}, we have $f(\Gamma_0) \subseteq \NN$. Hence there is a function $\varphi : \Gamma_0/\langle \tau \rangle \lra \NN$ with $\varphi([x]) =
f(x)$ for every $x \in \Gamma_0$. According to (1) we obtain for every $y \in \Gamma_0$
\[ (\ast) \ \ \ \ \ 2\varphi([y]) = f(y) + f(\tau(y)) = \sum_{x \in y^-} f(x) \, {\rm pr}_1(\nu(x,y)) = \sum_{[x] \in \Gamma_0/\langle\tau \rangle} \varphi([x])\, d([x],[y]),\]
so that $\varphi$ is indeed additive.

(3) This follows directly from ($\ast$). \end{proof}

\bigskip
\noindent
We conclude this section by discussing certain functions of stable translation quivers $\Gamma$ of tree class $A_\infty$ that will make an appearance in the next section.  If $x \in \Gamma_0$
is a vertex, then its distance to the end of $\Gamma$ is referred to as the {\it quasi-length} $\ql(x)$ of $x$. Vertices of quasi-length $1$ are often called {\it quasi-simple}. Given $\ell \ge 1$,
we let $\Gamma_{(\ell)}$ be the full subquiver of $\Gamma$, whose vertices have quasi-length $\ge \ell$. A function $f : \Gamma_0 \lra \NN_0$ is called {\it eventually additive}, if there exists $\ell \ge 1$ such that
\[ f(y) + f(\tau(y)) = \sum_{x \in y^-} f(x) \ \ \ \ \ \ \ \ \forall \ y \in (\Gamma_{(\ell)})_0. \]
The minimal $\ell$ with this property will be denoted $\ell(f)$.

\bigskip

\begin{Examples} (1) The additive functions are precisely those eventually additive functions with $\ell(f) = 1$.

(2) Every constant function $f\ne 0$ is eventually additive with $\ell(f)=2$.

(3) The function $f : \Gamma_0 \lra \NN_0 \ \ ; \ \ x \mapsto \ql(x)\!-\!1$ is eventually additive with $\ell(f) = 2$.\end{Examples}

\bigskip
\noindent
For future reference, we record the following simple observation:

\bigskip

\begin{Lemma} \label{SF4} Let $\Gamma = (\Gamma_0, \Gamma_1, \nu, \tau)$ be a valued stable translation quiver of tree class $\bar{T}_\Gamma = A_\infty$. Let $f : \Gamma_0 \lra \NN_0$ be an eventually additive function such that $f\circ \tau = f$.

{\rm (1)} \ Let $x_{\ell(f)}$ and $x_{\ell(f)-1}$ be two vertices of quasi-lengths $\ell(f)$ and $\ell(f)\!-\!1$, respectively. Then
\[ f(x) = (f(x_{\ell(f)})-f(x_{\ell(f)-1}))(\ql(x)-\ell(f))+f(x_{\ell(f)}) \ \ \ \ \ \forall \ x \in (\Gamma_{(\ell(f))})_0,\]
where we put $f(x_0) = 0$ in case $\ell(f)=1$. If $f$ is bounded, then $f|_{(\Gamma_{(\ell(f)-1)})_0}$ is constant.

{\rm (2)} \ If $f$ is subadditive such that

\indent \indent {\rm (a)} \ $f(x_1) = f(x_2)$ for $x_i \in \Gamma_0$ with $\ql(x_i) = i$, and

\indent \indent {\rm (b)} \ $f(x) \ge f(x_1)$ for all $x,x_1 \in \Gamma_0$ with $\ql(x_1) = 1$ and $\ql(x)<\ell(f)$,

\noindent
then $f$ is constant and $\ell(f)\le 2$. \end{Lemma}

\begin{proof} (1) Since $f$ is eventually additive with $f \circ \tau = f$, the map $f$ gives rise to a function $\varphi : A_\infty \lra \NN_0$ on the orbit graph of $\Gamma$ which satisfies
$\varphi(n) = f(x)$ whenever $\ql(x) =n$. The proof of Lemma \ref{SF3} now yields
\[ 2\,\varphi(n) = \varphi(n\!+\!1) + \varphi(n\!-\!1) \ \ \ \ \ \ \forall \ n \ge \ell(f),\]
so that
\[ \varphi(n) = (\varphi(\ell(f))-\varphi(\ell(f)\!-\!1))(n-\ell(f)) + \varphi(\ell(f)) \ \ \ \ \ \ \forall \ n \ge \ell(f).\]
(Here we define $\varphi(0) = 0$.) This readily yields the asserted formula. If $f$ is bounded, then $\ell(f) \ge 2$ and we obtain $\varphi(\ell(f))-\varphi(\ell(f)-1) = 0$, whence
$\varphi(n) = \varphi(\ell(f)\!-\!1)$ for all $n \ge \ell(f)\!-\!1$. Consequently, $f|_{(\Gamma_{(\ell(f)-1)})_0}$ is constant.

(2) We are going to show $f(x) = f(x_1)$ by induction on $n:=\ql(x)$, the cases $n=1,2$ being trivial. Let $n \ge 3$. Since $f$ is subadditive and $\Gamma$ has tree class $A_\infty$, there exist vertices $y,z \in \Gamma_0$ of quasi-lengths $\ql(y)=n\!-\!1$ and $\ql(z) = n\!-\!2$ such that
\[ f(x) \le f(y)+f(\tau(y))-f(z) = 2f(y)-f(z) = f(x_1),\]
with equality holding for $n \ge \ell(f)$. For $n<\ell(f)$, condition (b) yields the reverse inequality, so that $f$ is constant. In particular,
$\ell(f)\le 2$. \end{proof}

\bigskip

\section{Additive Functions of $\pi$-Points}\label{S:AFP}
Throughout, we let $k$ be an algebraically closed field of characteristic $\Char(k) =p>0$ and consider a finite group scheme $\cG$ over $k$, whose coordinate ring and algebra of measures 
will be denoted $k[\cG]$ and $k\cG$, respectively. We shall identify $\cG$-modules and $k\cG$-modules and let $\modd \cG$ be the category of finite-dimensional $\cG$-modules. If 
$K\!:\!k$ is a field extension, then $\cG_K := {\rm Spec}_K(k[\cG]\otimes_kK)$ denotes the extended group, whose algebra of measures is
\[ K\cG := k\cG_K \cong k\cG\!\otimes_k\!K.\]
We refer the reader to \cite{Ja} and \cite{Wa} for background on group schemes and their representations.

Given $M \in \modd \cG$, we write $M_K := M\!\otimes_k\!K$ for the corresponding $\cG_K$-module. Let $T$ be an indeterminate over $k$ and put
\[ \fA_{p,K} := k[T]/(T^p)\!\otimes_k\!K \cong K[T]/(T^p).\]
Following Friedlander-Pevtsova \cite{FPe2}, we refer to a left flat algebra homomorphism $\alpha_K : \fA_{p,K} \lra K\cG$ as a {\it $\pi$-point} if there exists an abelian, unipotent subgroup
$\cU \subseteq \cG_K$ such that $\im \alpha_K \subseteq K\cU$. Like any flat algebra homomorphism, $\alpha_K$ gives rise to a pull-back functor $\alpha^\ast_K : \modd \cG_K \lra 
\modd \fA_{p,K}$ which is exact and sends projectives to projectives. Two $\pi$-points $\alpha_K$ and $\beta_L$ are {\it equivalent} if
\[ \alpha^\ast_K(M_K) \ \text{projective} \ \Leftrightarrow \beta^\ast_L(M_L) \ \text{projective}\]
for every $M \in \modd \cG$. The set of equivalence classes will be denoted $\Pi(\cG)$. Given $M \in \modd \cG$, we let
\[ \Pi(\cG)_M := \{[\alpha_K] \in \Pi(\cG) \ ; \ \alpha^\ast_K(M_K) \ \text{is not projective}\}\]
be the {\it $\Pi$-support} of $M$. In view of \cite[(3.4),(3.6)]{FPe2} the sets $\Pi(\cG)_M$ form the closed sets of a noetherian topology on $\Pi(\cG)$ such that
\[ \dim \Pi(\cG)_M = \dim \cV_\cG(M)-1,\]
where $\cV_\cG(M) \subseteq {\rm MaxSpec}(\HH^{{\rm ev}}(\cG,k))$ is the {\it cohomological support variety} of the $\cG$-module $M$ (see \cite[(5.7)]{Be2} for the definition that 
also applies in our context).

Let $\alpha_K : \fA_{p,K} \lra K\cG$ be a $\pi$-point. If $M \in \modd \cG$, then $\alpha^\ast_K(M_K) \in \modd \fA_{p,K}$ uniquely decomposes as
\[ \alpha^\ast_K(M_K) \cong \bigoplus_{i=1}^p \alpha_{K,i}(M)[i],\]
where $[i]$ represents the (up to isomorphism) unique indecomposable $\fA_{p,K}$-module of dimension $i$. We will interpret the right-hand side as a base change of an
$\fA_{p,k}$-module $N$, that is,
\[ \alpha^\ast_K(M_K) \cong N_K,\]
with $N = \bigoplus_{i=1}^p\alpha_{K,i}(M)[i] \in \modd \fA_{p,k}$. The isomorphism class of $N$ is the {\it Jordan type of $M$ with respect to $\alpha_K$}, denoted
$\Jt(M,\alpha_K)$.

Given a short exact sequence
\[\fE \ \ : \ \ (0) \lra M' \lra M \lra M'' \lra (0)\]
of $\cG_K$-modules, we write
\[ \alpha^\ast_K(\fE) \ \ : \ \ (0) \lra \alpha^\ast_K(M'_K) \lra \alpha^\ast_K(M_K) \lra \alpha^\ast_K(M''_K) \lra (0).\]
If $M$ is a non-projective indecomposable $\cG_K$-module, then
\[ \fE_M \ \ : \ \ (0) \lra N \lra E \lra M \lra (0)\]
denotes the {\it almost split sequence} terminating in $M$. The reader is referred to \cite[Chap.V]{ARS} for the definition and basic properties of almost split sequences. We shall write
$X\!\mid\!Y$ to indicate that $X$ is isomorphic to a direct summand of $Y$.

For future reference we record a few basic properties of the functor $M \mapsto M_K$. In the sequel, $\Omega_\Lambda$ denotes the Heller operator of the finite-dimensional $k$-algebra 
$\Lambda$.

\bigskip

\begin{Lemma} \label{AFP1} Let $\Lambda$ be a finite-dimensional $k$-algebra, $K\!:\!k$ be a field extension. Given a finite-dimensional $\Lambda$-module $M$, the following statements 
hold:

{\rm (1)} \ If $M$ is indecomposable, then the $\Lambda_K$-module $M_K$ is indecomposable.

{\rm (2)} \ $M$ is projective if and only if $M_K$ is projective.

{\rm (3)} \ We have $\Omega_{\Lambda_K}(M_K) \cong \Omega_\Lambda(M)_K$.\end{Lemma}

\begin{proof} (1) Let $R := \End_\Lambda(M)$ and denote by $J$ the Jacobson radical of $R$. Then $R$ is local and since $k$ is algebraically closed, we have $R/J \cong k$, whence
$R_K/J_K \cong K$. As $R_K \cong \End_{\Lambda_K}(M_K)$, it follows that $M_K$ is indecomposable.

(2) One implication is clear. Part (1) implies that $P \mapsto P_K$ provides a bijection between the isomorphism classes of principal indecomposable modules for $\Lambda$
and $\Lambda_K$, respectively. Thus, if $M_K$ is projective, so is $M$.

(3) This is a direct consequence of \cite[(3.5)]{Ka}. \end{proof}

\bigskip

\begin{Definition} Let $M$ be a $\cG$-module, $\alpha_K : \fA_{p,K} \lra K\cG$ be a $\pi$-point. We say that $M$ is {\it relatively $\alpha_K$-projective} if
$M_K\!\mid\!(K\cG\!\otimes_{\fA_{p,K}}\!\alpha_K^\ast(M_K))$. The $\cG$-module $M$ is {\it relatively $\alpha_K$-injective}, if $M_K\!\mid \! 
\Hom_{\fA_{p,K}}(K\cG,\alpha^\ast_K(M_K))$. \end{Definition}

\bigskip

\begin{Remark} Suppose that $M$ is a $\cG$-module such that $M_K\!\mid\!(K\cG\!\otimes_{\fA_{p,K}}\!N)$ for some $\fA_{p,K}$-module $N$. Then there exists a split surjective
homomorphism $\varphi : K\cG\otimes_{\fA_{p,K}}N \lra M_K$ of $K\cG$-modules. Direct computation shows that the composite $\mu \circ \psi$ of the homomorphism
\[ \psi : K\cG\!\otimes_{\fA_{p,K}}\!N \lra K\cG\!\otimes_{\fA_{p,K}}\!\alpha_K^\ast(M_K) \ \ ; \ \ a\otimes n \mapsto a\otimes \varphi(1\otimes n)\]
with the canonical map
\[ \mu : K\cG\!\otimes_{\fA_{p,K}}\!\alpha_K^\ast(M_K) \lra M_K \ \ ; \ \ a\otimes m \mapsto a.m\]
coincides with $\varphi$. Hence $\mu$ is also split surjective. As a result, the following statements are equivalent:
\begin{itemize}
\item $M$ is relatively $\alpha_K$-projective.
\item There exists an $\fA_{p,K}$-module $N$ such that $M_K\!\mid\!(K\cG\!\otimes_{\fA_{p,K}}\!N)$.
\item The map $\mu : K\cG\!\otimes_{\fA_{p,K}}\!\alpha_K^\ast(M_K) \lra M_K \ \ ; \ \ a\otimes m \mapsto a.m$ is split surjective.
\end{itemize}
\end{Remark}

\bigskip
\noindent
The following subsidiary result, which elaborates on \cite[(8.5)]{CFP}, is inspired by the modular representation theory of finite groups (see \cite[(4.12.10)]{Be1}).

\bigskip

\begin{Lemma} \label{AFP2} Let $M$ be a non-projective indecomposable $\cG$-module, $\alpha_K : \fA_{p,K} \lra K\cG$ be a $\pi$-point. Then the following statements hold:

{\rm (1)} \ $M$ is relatively $\alpha_K$-projective if and only if $M_K\!\mid\!(K\cG\!\otimes_{\fA_{p,K}}\![i])$ for some $i \in \{1,\ldots,p\!-\!1\}$ with $\alpha_{K,i}(M) \ne 0$.

{\rm (2)} \ If $M$ is relatively $\alpha_K$-projective, then $M_K$ is cyclic, and $\Pi(\cG)_M = \{[\alpha_K]\}$.

{\rm (3)} \ $M$ is relatively $\alpha_K$-projective if and only if the sequence $\alpha^\ast_K(\fE_M\!\otimes_k\!K)$ does not split. \end{Lemma}

\begin{proof} (1) In view of Lemma \ref{AFP1} and the above remark, this is a direct consequence of the Theorem of Krull-Remak-Schmidt.

(2) By (1), there exists $i \in \{1,\ldots, p\!-\!1\}$ such that $M_K\!\mid\!(K\cG\!\otimes_{\fA_{p,K}}\![i])$. Consequently, $M_K$ is cyclic, and a consecutive application of 
\cite[(5.5)]{FPe2} and \cite[(8.4)]{CFP} implies $\emptyset \ne \Pi(\cG_K)_{M_K} \subseteq \Pi(\cG_K)_{K\cG\otimes_{\fA_{p,K}}[i]} = \{[\alpha_K]\}$, so that $\Pi(\cG_K)_{M_K} = 
\{[\alpha_K]\}$. It readily follows that $[\alpha_K] \in \Pi(\cG)_M$. Conversely, assume $\beta_L$ to be a $\pi$-point of $\cG$ such that $[\beta_L] \in 
\Pi(\cG)_M$. Consider a common extension field $Q$ of $L$ and $K$. Then $\beta_Q$ is a $\pi$-point of $\cG_K$ such that $\beta^\ast_Q((M_K)_Q) \cong \beta^\ast_Q(M_Q) \cong 
\beta^\ast_L(M_L)_Q$ is not projective. Consequently, $[\beta_Q] \in \Pi(\cG_K)_{M_K}$, so that the $\pi$-points $\alpha_K$ and $\beta_Q$ of $\cG_K$ are equivalent. We conclude 
$\alpha_K$ and $\beta_L$ are also equivalent, when considered as $\pi$-points of $\cG$, whence $\Pi(\cG)_M = \{[\alpha_K]\}$.

(3) Let
\[ \fE_M \ \ : \ \ (0) \lra N \lra E \lra M \lra (0)\]
be the almost split sequence terminating in $M$. According to Lemma \ref{AFP1}, the $\cG_K$-module $M_K$ is indecomposable and non-projective. Thanks to \cite[(3.8)]{Ka}, the sequence
\[ \fE_M\!\otimes_k\!K \ \ : \ \ (0) \lra N_K \lra E_K \stackrel{\pi}{\lra} M_K \lra (0)\]
is the almost split sequence terminating in $M_K$. We consider the canonical surjection $\mu : K\cG\otimes_{\fA_{p,K}} \alpha^\ast_K(M_K) \lra M_K \ ; \ a\otimes m \mapsto am$. If this
map does not split, then $\fE_M\!\otimes_k\!K$ being almost split implies that there exists $\omega : K\!\cG\otimes_{\fA_{p,K}}\!\alpha^\ast_K(M_K) \lra E_K$ such that $\pi \circ
\omega = \mu$. The adjoint isomorphism $\Hom_{K\cG}(K\cG\otimes_{\fA_{p,K}}\alpha^\ast_K(M_K),-) \cong \Hom_{\fA_{p,K}}(\alpha^\ast_K(M_K),-)\circ \alpha^\ast_K$ thus 
provides a map $\tilde{\omega} : \alpha^\ast_K(M_K) \lra \alpha^\ast_K(E_K)$ with
\[ \alpha^\ast_K(\pi) \circ \tilde{\omega} = \id_{\alpha^\ast_K(M_K)}.\]
Thus, if $\alpha^\ast_K(\fE_M\!\otimes_k\!K)$ does not split, then $M_K\!\mid\!(K\cG\!\otimes_{\fA_{p,K}}\!\alpha^\ast_K(M_K))$, and $M$ is relatively $\alpha_K$-projective.

Let $M$ be relatively $\alpha_K$-projective. According to \cite[(VI.3.6)]{ARS}, $\alpha_K^\ast(\fE_M\!\otimes_k\!K)$ being split entails that $\fE_M\!\otimes_k\!K$ is split. Since
$\fE_M\!\otimes_k\!K$ is almost split, this cannot happen, so that $\alpha^\ast_K(\fE_M\!\otimes_k\!K)$ is non-split in that case. \end{proof}

\bigskip

\begin{Remarks} (1) Part (2) of the above Lemma implies, that the $\pi$-points $\alpha_K: \fA_{p,K} \lra K\cG$ for which a given $\cG$-module $M$ is relatively $\alpha_K$-projective, give 
rise to closed points $[\alpha_K] \in \Pi(\cG)$. In view of \cite[(4.7)]{FPe2}, the class $[\alpha_K]$ is therefore represented by some $p$-point $\beta_k : \fA_{p,k} \lra k\cG$.

(2) Let $\alpha_K \in \Pt(\cG)$ be a $\pi$-point such that $[\alpha_K]$ is not closed. Then Lemma \ref{AFP2} implies that no non-projective direct summand $N$ of the induced module 
$K\cG\!\otimes_{\fA_{p,K}}\![i]$ is defined over $k$, that is, there does not exist a $\cG$-module $M$ with $M_K\cong N$.

(3) Lemma \ref{AFP2} could equally well be stated in terms of relatively injective modules. In particular, the sequence $\alpha_K^\ast(\fE_M\!\otimes_k\!K)$ does not split if and only if 
$M$ is relatively $\alpha_K$-injective. \end{Remarks}

\bigskip
\noindent
We let $\Gamma_s(\cG)$ denote the {\it stable Auslander-Reiten quiver} of the Frobenius algebra $k\cG$. By definition, $\Gamma_s(\cG)$ is a valued stable translation quiver, whose
vertices are the isoclasses of the non-projective indecomposable $\cG$-modules. There is an arrow $X \rightarrow Y$ if $X$ is a direct summand of the middle term $E$ of the almost split
sequence $\fE_Y$. This arrow carries the valuation $(m,m)$ if $X$ occurs in $E$ with multiplicity $m$, see \cite[(V.1.3)]{ARS}. The {\it Auslander-Reiten translation} $\tau_\cG$ is given by
\[ \tau_\cG = \Omega^2_\cG \circ \nu_\cG = \nu_\cG \circ \Omega_\cG^2,\]
where $\nu_\cG := \Hom_\cG(-,k\cG)^\ast$ denotes the {\it Nakayama functor} of $\modd \cG$. By virtue of \cite[(V.1.14)]{ARS}, $\tau_\cG(M)$ is the initial term of the almost split 
sequence terminating in $M$.

Let $\Theta \subseteq \Gamma_s(\cG)$ be a connected component. In view of \cite[(3.1)]{Fa3}, whose proof can be easily modified to cover our context, we have
\[ \Pi(\cG)_M = \Pi(\cG)_N \ \ \ \ \ \ \forall \ M,N \in \Theta.\]
Accordingly, we shall speak of the {\it $\Pi$-support $\Pi(\cG)_\Theta$ of $\Theta$}.

\bigskip

\begin{Definition} Let $\Theta \subseteq \Gamma_s(\cG)$ be a component. Given a $\pi$-point $\alpha_K : \fA_{p,K} \lra K\cG$, we say that $\Theta$ is {\it $\alpha_K$-split} if the exact 
sequence $\alpha_K^\ast(\fE_M\!\otimes_k\!K)$ splits for every $M \in \Theta$. \end{Definition}

\bigskip
\noindent
By abuse of notation, we will write $f : \Theta \lra \NN_0$ for a function $f$ that is defined on the set of vertices of $\Theta$. If $N$ is an $\fA_{p,K}$-module, then $N_{\rm pf}$ denotes 
its projective-free part, that is, the direct sum of all non-projective indecomposable summands of $N$.

\bigskip

\begin{Proposition} \label{AFP3} Let $\Theta \subseteq \Gamma_s(\cG)$ be a component, $\alpha_K : \fA_{p,K}\lra K\cG$ be a $\pi$-point.

{\rm (1)} \ We have $\alpha_{K,i}\circ \tau_\cG = \alpha_{K,i}$ for $i \in \{1,\ldots,p\!-\!1\}$.

{\rm (2)} \ If $\Theta$ is not $\alpha_K$-split, then $\Pi(\cG)_\Theta = \{[\alpha_K]\}$ and $\Theta$ is either finite or isomorphic to $\ZZ[A_\infty]/\langle \tau^n\rangle$ for some
$n\in \NN$.

{\rm (3)} \ Given $m \in \{1,\ldots ,p\!-\!1\}$, the function
\[ \psi_m : \Theta \lra \NN_0 \ \ ; \ \ M \mapsto \sum_{i=1}^{m-1} i\alpha_{K,i}(M)+ m(\sum_{i=m}^{p-1}\alpha_{K,i}(M))\]
is subadditive with $\psi_m \circ \tau_\cG = \psi_m$.

{\rm (4)} \ Suppose that $\alpha_{K,i}: \Theta \lra \NN_0$ is additive for $1\le i \le p\!-\!1$. Then $\Theta$ is $\alpha_K$-split.\end{Proposition}

\begin{proof} (1) Let $\zeta : k\cG \lra k$ be the modular function of the Hopf algebra $k\cG$. According to \cite[(1.5)]{FMS}, the convolution
\[ \mu := \zeta \ast \id_{k\cG}\]
is a Nakayama automorphism of the Frobenius algebra $k\cG$, so that $\nu_\cG$ coincides with the pull-back functor $\mu^\ast : \modd \cG \lra \modd \cG$, defined by $\mu$.

Since $(\zeta\otimes \id_K)|_{K\cU} = \varepsilon_{K\cU}$ for every unipotent subgroup $\cU \subseteq \cG_K$ and  $\alpha_K$ factors through some unipotent subgroup $\cU \subseteq 
\cG_K$, we have $(\mu \otimes \id_K) \circ \alpha_K = \alpha_K$, whence $\alpha^\ast_K \circ (\mu\otimes \id_K)^\ast = \alpha^\ast_K$. As $\alpha^\ast_K$ is exact and sends 
projectives to projectives, Lemma \ref{AFP1} implies
\[ \alpha^\ast_K(\tau_\cG(M)_K) \cong \alpha^\ast_K((\mu\otimes \id_K)^\ast(\Omega^2_\cG(M)_K)) \cong \alpha^\ast_K(\Omega^2_\cG(M_K)) \cong 
\Omega^2_{\fA_{p,K}}(\alpha^\ast_K(M_K)) \oplus \text{(proj.)}.\]
Observing $\Omega^2_{\fA_{p,K}}([j]) \cong [j]$ for $1 \le j \le p\!-\!1$, we conclude that
\[ \Omega^2_{\fA_{p,K}}(\alpha^\ast_K(M_K)) \oplus \text{(proj.)} \cong \alpha^\ast_K(M_K).\]
Consequently,
\[ \alpha^\ast_K(\tau_\cG(M)_K) \oplus \text{(proj.)} \cong \alpha^\ast_K(M_K)\oplus \text{(proj.)},\]
so that $\alpha_{K,i}(\tau_\cG(M)) = \alpha_{K,i}(M)$ for $i \in \{1,\ldots,p\!-\!1\}$.

(2) By assumption, there exists $M \in \Theta$ such that the exact sequence $\alpha_K^\ast(\fE_M\!\otimes_k\!K)$ does not split. In view of (2) and (3) of Lemma \ref{AFP2}, we have 
$\Pi(\cG)_\Theta = \Pi(\cG)_M = \{[\alpha_K]\}$. We may now apply \cite[(3.3)]{Fa3} to see that $\Theta$ is either finite or an infinite tube $\ZZ[A_\infty]/\langle \tau^n \rangle$.

(3) We write $\fA_{p,K} = K[t]$ with $t^{p-1} \ne 0 = t^p$. Let $m \in \{1,\ldots,p\!-\!1\}$. Given an $\fA_{p,K}$-module $N$, we let $t^m_N$ be the left multiplication on $N$ 
effected by $t^m$. Direct computation yields
\[ \ker t^m_{\alpha^\ast_K(M_K)} = \bigoplus_{i=1}^{m-1} \alpha_{K,i}(M)[i]\oplus (\sum_{i=m}^p\alpha_{K,i}(M))[m],\]
so that
\[ \dim_K \ker t^m_{\alpha^\ast_K(M_K)} = \psi_m(M)+m\,\alpha_{K,p}(M) \ \ \ \ \ \ \forall \ M \in \modd \cG.\]
We denote the right-hand side by $\varphi_m$ and consider a short-exact sequence $(0) \lra M' \lra M \lra M'' \lra (0)$ of $\cG$-modules. Observing the left-exactness of kernels, we obtain
\[ \varphi_m(M) \le \varphi_m(M')+\varphi_m(M'').\]
For $X \in \{M,M',M''\}$ we write $\alpha_K^\ast(X_K) = \alpha^\ast_K(X_K)_{\rm pf} \oplus P_X$, with $P_X$ projective. Then $(P_{M'}\oplus P_{M''})\!\mid\! P_M$, so that
\[ p\,\alpha_{K,p}(M) \ge p\,\alpha_{K,p}(M')+p\,\alpha_{K,p}(M'').\]
We therefore obtain
\begin{eqnarray*}
\psi_m(M') + \psi_m(M'') & = & \varphi_m(M') +\varphi_m(M'') - m\, \alpha_{K,p}(M') - m\, \alpha_{K,p}(M'')\\
 &\ge & \varphi_m(M) - m\, \alpha_{K,p}(M) = \psi_m(M).
\end{eqnarray*}
Consequently, $\psi_m$ is subadditive on $\Theta$ and (1) implies $\psi_m\circ \tau_\cG = \psi_m$.

(4) Let $M$ be an element of $\Theta$, and consider the almost split sequence
\[ (0) \lra N \lra E \lra M \lra (0)\]
terminating in $M$. Since each $\alpha_{K,i}$ is additive for $1 \le i \le p\!-\!1$, we obtain
\[ \alpha_K^\ast(M_K)_{\rm pf} \oplus \alpha_K^\ast(N_K)_{\rm pf} \cong \alpha_K^\ast(E_K)_{\rm pf}.\]
For $X \in \{M,N,E\}$ we write $\alpha_K^\ast(X_K) = \alpha^\ast_K(X_K)_{\rm pf} \oplus P_X$, with $P_X$ projective. Then $(P_N\oplus P_M)\!\mid\! P_E$ and a comparison of
dimensions yields $P_E \cong P_M \oplus P_N$, so that
\[ \alpha_K^\ast(M_K) \oplus \alpha_K^\ast(N_K)\cong \alpha_K^\ast(E_K).\]
Consequently, the exact sequence
\[ (0) \lra \alpha_K^\ast(N_K) \lra \alpha_K^\ast(E_K) \lra \alpha_K^\ast(M_K) \lra (0)\]
splits. As a result, the component $\Theta$ is $\alpha_K$-split. \end{proof}

\bigskip
\noindent
We say that an AR-component $\Theta \subseteq \Gamma_s(\cG)$ is {\it regular} if the middle terms of the almost split sequences terminating in the vertices of $\Theta$ have no non-zero
projective summands. By general theory (cf.\ \cite[(V.5.5)]{ARS}), an almost split sequence, whose middle term has the principal indecomposable module $P$ as a direct summand, is
isomorphic to the sequence
\[ (0) \lra \Rad(P) \lra P\oplus (\Rad(P)/\Soc(P)) \lra P/\Soc(P) \lra (0).\]
Hence all but finitely many AR-components of $k\cG$ are regular.

By providing the converse to Proposition \ref{AFP3}(4), the following basic result indicates the utility of those $\pi$-points $\alpha_K$ for which a component $\Theta$ is $\alpha_K$-split.

\bigskip

\begin{Theorem} \label{AFP4} Let $\Theta \subseteq \Gamma_s(\cG)$ be a component, $\alpha_K  : \fA_{p,K} \lra K\cG$ be a $\pi$-point such that $\Theta$ is $\alpha_K$-split. Then the 
following statements hold:

{\rm (1)} \ For every $i \in \{1,\ldots,p\!-\!1\}$, the map
\[ \alpha_{K,i} : \Theta \lra \NN_0 \ \ \ ; \ \ \ M \mapsto \alpha_{K,i}(M)\]
is an additive function on $\Theta$ such that $\alpha_{K,i}\circ \tau_\cG = \alpha_{K,i}$.

{\rm (2)} \ The map
\[\alpha_{K,<p} : \Theta \lra \NN_0 \ \ \ ; \ \ \ M \mapsto \dim_kM-p\, \alpha_{K,p}(M)\]
is an additive function on $\Theta$ such that $\alpha_{K,<p}\circ \tau_\cG = \alpha_{K,<p}$.

{\rm (3)} \ The map
\[ \alpha_{K,p} : \Theta \lra \NN_0 \ \ \ ; \ \ \ M \mapsto \alpha_{K,p}(M)\]
is a subadditive function on $\Theta$, which is additive if and only if $\Theta$ is regular.\end{Theorem}

\begin{proof} Let $M$ be a non-projective indecomposable $\cG$-module, and consider the almost split sequence
\[ (0) \lra \tau_\cG(M) \lra E \lra M \lra (0)\]
terminating in $M$. Since $\Theta$ is $\alpha_K$-split, the sequence
\[ (0) \lra \alpha^\ast_K(\tau_\cG(M)_K) \lra \alpha^\ast_K(E_K) \lra \alpha^\ast_K(M_K) \lra (0)\]
is split exact, so that the Theorem of Krull-Remak-Schmidt implies
\[ \alpha_{K,i}(E) = \alpha_{K,i}(\tau_\cG(M))+\alpha_{K,i}(M)\]
for $1\le i\le p$. Upon decomposing $E$ into its indecomposable constituents $E = n_1E_1 \oplus \cdots \oplus n_rE_r \oplus (\text{proj.})$ with $\ E_i\not \cong E_j$ for $i\ne j$, we 
obtain
\[ \alpha_{K,i}(E) = \sum_{j=1}^r n_j \alpha_{K,i}(E_j)\]
for $i \le p\!-\!1$, while
\[ \alpha_{K,p}(E) = \sum_{j=1}^r n_j \alpha_{K,p}(E_j) + \alpha_{K,p}(\text{(proj.)}).\]
Consequently, $\alpha_{K,i}$ is an additive function on $\Theta$ for $i \le p\!-\!1$, while $\alpha_{K,p}$ is a subadditive function, which is additive if and only if there is no middle term $E$
such that $E_K$ has a non-zero projective summand. Owing to Lemma \ref{AFP1}, the functor $M \mapsto M_K$ sends indecomposables to indecomposables and preserves and reflects
projectives. Thus, $E_K$ has a non-zero projective summand if and only if $E$ does. As a result, $\alpha_{p,K}$ is additive if and only if $\Theta$ is regular.

Being a sum of additive functions, $\alpha_{K,<p} = \sum_{i=1}^{p-1}i \alpha_{K,i}$ is also additive. According to (\ref{AFP3}(1)), the additive functions of (1) and (2) are
$\tau_\cG$-invariant. \end{proof}

\bigskip
\noindent
Setting $\max \emptyset := 0$, we record the following consequence of the proof of Theorem \ref{AFP4}:

\bigskip

\begin{Corollary} \label{AFP5} Let $\Theta \cong \ZZ[A_\infty]/\langle \tau^n \rangle$ be an infinite tube, $\alpha_K : \fA_{p,K} \lra K\cG$ be a $\pi$-point with $\Pi(\cG)_\Theta = 
\{[\alpha_K]\}$. The functions $\alpha_{K,i}, \psi_m : \Theta \lra \NN_0$ are eventually additive for $i,m \in \{1,\ldots, p-1\}$, with $\ell(\alpha_{K,i}), \ell(\psi_m) \le 1+\max \{\ql(M) 
\ ; \ M \ \text{is relatively $\alpha_K$-projective}\}$.\end{Corollary}

\begin{proof} Given $i \in \{1,\ldots,p\!-\!1\}$, we put $f := \alpha_{K,i}$. Owing to (1) of Lemma \ref{AFP2}, at most finitely many $M \in \Theta$ are relatively $\alpha_K$-projective. 
Hence setting
\[ \ell := 1+\max \{\ql(M) \ ; \ M \ \text{is relatively $\alpha_K$-projective}\},\]
we see that no $M \in \Theta_{(\ell)}$ is relatively $\alpha_K$-projective. The proof of Theorem \ref{AFP4} in conjunction with Lemma \ref{AFP2}(3) now ensures that $f$ is eventually 
additive with $\ell(f) \le \ell$. By the same token, the subadditive functions $\psi_m$ are eventually additive with $\ell(\psi_m) \le \ell$. \end{proof}

\bigskip
\noindent
Let $\Theta \subseteq \Gamma_s(\cG)$ be a component with $\Pi(\cG)_\Theta = \{[\alpha_K]\}$. If $\Theta \subseteq \Gamma_s(\cG)$ is finite, then \cite[(3.3)]{Fa3} ensures that the 
tree class $\bar{T}_\Theta$ is a finite Dynkin diagram. Since the corresponding Cartan matrix is invertible, $f=0$ is the only $\tau_\cG$-invariant additive function on $\Theta$.

More specifically, consider the algebra $k\GG_{a(1)}$, associated to the first Frobenius kernel $\GG_{a(1)}$ of the additive group $\GG_a$. Then $k\GG_{a(1)}$ is isomorphic to $\fA_{p,k}$
and $\Gamma_s(\GG_{a(1)})$ is connected. For every $j \in \{1,\ldots,p\}$ there is an indecomposable module $M_j$ of dimension $j$, and the isomorphism $\alpha_k : \fA_{p,k} \lra
k\GG_{a(1)}$ gives $\alpha_{k,i}(M_j) = \delta_{i,j}$. This shows that, for $p\ge 3$, none of the functions $\alpha_{k,i}$ ($1\le i \le p\!-\!1$) is subadditive (see Proposition \ref{SF1}). 
For $p=2$, the function $\alpha_{k,1} = \psi_1$ is subadditive.

\bigskip

\begin{Corollary} \label{AFP6}  Suppose that $\dim \Pi(\cG) \ge 1$. Let $\Theta \cong \ZZ[A_\infty]/\langle \tau^n \rangle$ be a periodic component of $\Gamma_s(\cG)$, $\alpha_K$
be a $\pi$-point. If there exist $\cG$-modules $M_1, M_2 \in \Theta$ with $\ql(M_j) =j$ and $\dim_K\alpha^\ast_K((M_j)_K)_{\rm pf} = p$ for $1 \le j \le 2$, then $\alpha_{K,i} : 
\Theta_{(\ell(\alpha_{K,i})-1)} \lra \NN_0$ is constant for each $i \in \{1, \ldots, p-1\}$ and $\Theta$ contains a relatively $\alpha_K$-projective module. \end{Corollary}

\begin{proof} Since $\Theta$ is a periodic component, we have $\dim \Pi(\cG)_\Theta = 0$. Thus, $\Pi(\cG) \ne \Pi(\cG)_\Theta$, and there exists a $\pi$-point $\beta_L$ such that 
$\beta^\ast_L(M_L)$ is projective for every $M \in \Theta$. Consequently, the dimension of every module $M \in \Theta$ is divisible by $p$, so that
\[ p\!\mid\!\dim_K\alpha^\ast_K(M_K)_{\rm pf} = \psi_{p-1}(M)\]
for all $M \in \Theta$. Thanks to Proposition \ref{AFP3}, the function $\psi_{p-1}$ is subadditive with $\psi_{p-1}\circ \tau_\cG = \psi_{p-1}$. Lemma \ref{SF4}(2) now yields
\[ \psi_{p-1}(M) = p \ \ \ \ \ \forall \ M \in \Theta.\]
In particular, each of the functions $\alpha_{K,i} : \Theta \lra \NN_0,\,  (1\le i\le p\!-\!1)$ is bounded and a consecutive application of Corollary \ref{AFP5} and Lemma \ref{SF4}(1) now
ensures that each $\alpha_{K,i}|_{\Theta_{(\ell(\alpha_{K,i})-1)}}$ is constant.

Since not all $\alpha_{K,i}$ are identically zero, we conclude that at least one $\alpha_{K,i}$ is not additive. A consecutive application of Theorem \ref{AFP4} and Lemma \ref{AFP2} now 
yields the existence of a vertex $M \in \Theta$ which is relatively $\alpha_K$-projective. \end{proof}

\bigskip

\section{Invariants of Auslander-Reiten Components}\label{S:IAR}
In this section we show that various notions introduced in \cite{CFP,FPS} give rise to invariants of stable Auslander-Reiten components of $k\cG$. In view of Theorem \ref{AFP4}, we shall
focus on those components $\Theta \subseteq \Gamma_s(\cG)$ that are $\alpha_K$-split for every $\pi$-point $\alpha_K$. Thanks to Lemma \ref{AFP2} and \cite[(3.3)]{Fa3}, the remaining
components are either of the form $\ZZ[A_\infty]/\langle \tau^n \rangle$ or are of finite Dynkin type, in which case their vertices form the non-projective modules of a block $\cB \subseteq 
k\cG$ of finite representation type, see \cite[(VI.1.4)]{ARS}. In the following, we let $\Pt(\cG)$ be the set of $\pi$-points of $\cG$.

\bigskip

\begin{Definition} A component $\Theta \subseteq \Gamma_s(\cG)$ is called {\it locally split} if $\Theta$ is $\alpha_K$-split for every $\alpha_K \in \Pt(\cG)$.\end{Definition}

\bigskip

\subsection{The Fundamental Invariants}
Given a component $\Theta \subseteq \Gamma_s(\cG)$, we define
\[ \Pt(\cG,\Theta) := \{ \alpha_K \in \Pt(\cG) \ ; \ \Theta \text{\ is $\alpha_K$-split}\}.\]
Our first result provides for every infinite component $\Theta \subseteq \Gamma_s(\cG)$ a function $d^\Theta : \Pt(\cG,\Theta) \lra \NN^p_0$ on which our new invariants will depend.

\bigskip

\begin{Thm}\label{FI1} Let $\Theta \subseteq \Gamma_s(\cG)$ be an infinite component. Then there exist an additive function $f_\Theta : \Theta \lra \NN$ and a function $d^\Theta : 
\Pt(\cG,\Theta) \lra \NN_0^p$ with
\[ \alpha_{K,i}(M_K) = d_i^\Theta(\alpha_K)f_\Theta(M) \ \ 1 \le i \le p\!-\!1 \ \ \text{and} \ \ \alpha_{K,<p}(M_K) = d_p^\Theta(\alpha_K)f_\Theta(M)\]
for every $M \in \Theta$ and every $\alpha_K \in \Pt(\cG,\Theta)$. \end{Thm}

\begin{proof} According to \cite[(3.2)]{Fa3}, the tree class $\bar{T}_\Theta$ is either a simply laced finite or infinite Dynkin diagram, a simply laced Euclidean diagram, or $\tilde{A}_{12}$. 
If $\bar{T}_\Theta$ is a finite Dynkin diagram, then \cite[(3.3)]{Fa3} implies that each vertex of $\Theta$ is periodic, so that $\Theta$ is finite. Consequently, the tree class $\bar{T}_\Theta$
belongs to $\{A_\infty, \, A_\infty^\infty,\, D_\infty, \, \tilde{A}_{12},\, \tilde{D}_n,\, \tilde{E}_{6,7,8}\}$, with each tree carrying the standard valuation.

For each tree class $\bar{T}_\Theta$ there exists an additive function $\varphi_{\bar{T}_\Theta} : \bar{T}_\Theta \lra \NN$ such that every other additive function on $\bar{T}_\Theta$ is 
an integral multiple of $\varphi_{\bar{T}_\Theta}$ (cf.\ \cite[p.243ff]{ARS} or \cite[Thm.2, Rem. p.328]{HPR2}). We define
\[ f_{ \ZZ[T_\Theta]} : \ZZ[T_\Theta] \lra \NN \ \ ; \ \ (n,x) \mapsto \varphi_{\bar{T}_\Theta}(x).\]
By (\ref{SF2}) and (\ref{SF3}), $\bar{T}_\Theta$ is the orbit graph of $\ZZ[T_\Theta]$ and Lemma \ref{SF3}(3) ensures that $f_{\ZZ[T_\Theta]}$ is an additive function on 
$\ZZ[T_\Theta]$ such that
\[ f_{\ZZ[T_\Theta]}\circ \tau = f_{\ZZ[T_\Theta]}.\]
Moreover, any other additive function $f : \ZZ[T_\Theta] \lra \NN$ with $f\circ \tau = f$ is an integral multiple of $f_{\ZZ[T_\Theta]}$.

The Riedtmann structure theorem \cite[(4.15)]{Be1} provides an admissible subgroup $G \subseteq {\rm Aut}(\ZZ[T_\Theta])$ such that
\[ \Theta \cong \ZZ[T_\Theta]/G.\]
Since $G$ is admissible, the canonical projection map $\pi : \ZZ[T_\Theta] \lra \Theta$ is a covering of valued stable translation quivers, that is, $\pi$ induces bijections $x^+ 
\stackrel{\sim}{\lra} \pi(x)^+$ and $x^- \stackrel{\sim}{\lra} \pi(x)^-$ for every $x \in \ZZ[T_\Theta]_0$. Thanks to Proposition \ref{SF1}(2), we have $f_{\ZZ[T_\Theta]}\circ g = 
f_{\ZZ[T_\Theta]}$ for every $g \in G$. Consequently, there exists a function $f_\Theta : \Theta \lra \NN$ with
\[ f_\Theta \circ \pi = f_{\ZZ[T_\Theta]}.\]
As $\pi$ a covering, $f_\Theta$ is additive and every other additive function $f : \Theta \lra \NN$ is an integral multiple of $f_\Theta$.

Let $\alpha_K \in \Pt(\cG,\Theta)$ be a $\pi$-point. Since $\Theta$ is $\alpha_K$-split, Theorem \ref{AFP4} ensures that each $\alpha_{K,i}$ $(1\le i \le p\!-\!1)$ is an additive function, 
with $\alpha_{K,i} \circ \tau_\cG = \alpha_{K,i}$. In view of Proposition \ref{SF1}, we either have $\alpha_{K,i} = 0$ or $\alpha_{K,i}(M) \in \NN$ for every $M \in \Theta$. We thus set 
$d_i^\Theta(\alpha_K)=0$ whenever $\alpha_{K,i} = 0$.

Suppose that $\alpha_{K,i} \neq 0$. Since $\alpha_{K,i} \circ \tau_\cG = \alpha_{K,i}$, there exists $d_i^\Theta(\alpha_K) \in \NN$ with
\[ \alpha_{K,i} = d_i^\Theta(\alpha_K)f_\Theta.\]
The same reasoning yields the existence of $d_p^\Theta(\alpha_K)$.  \end{proof}

\bigskip
\noindent
For future reference, we record a few special instances:

\bigskip

\begin{Cor} \label{FI2} Let $\Theta \subseteq \Gamma_s(\cG)$ be a component.

{\rm (1)} \  If $\Theta$ has tree class $\bar{T}_\Theta \cong A_\infty$, then we have
\[ \alpha_{K,i}(M_K) = d_i^\Theta(\alpha_K)\ql(M) \ \ \text{and} \ \ \alpha_{K,<p}(M_K) = d_p^\Theta(\alpha_K)\ql(M)\]
for every $i \in \{1,\ldots,p\!-\!1\}$, every $M \in \Theta$ and every $\alpha_K \in \Pt(\cG,\Theta)$.

{\rm (2)} \ If $\bar{T}_\Theta \cong \tilde{A}_{12},\, A_\infty^\infty$, then we have
\[ \alpha_{K,i}(M_K) = d_i^\Theta(\alpha_K) \ \ \text{and} \ \ \alpha_{K,<p}(M_K) = d_p^\Theta(\alpha_K)\]
for every $i \in \{1,\ldots,p\!-\!1\}$, every $M \in \Theta$ and every $\alpha_K \in \Pt(\cG)$.\end{Cor}

\begin{proof} (1) If $\Theta$ has tree class $A_\infty$, then $f_{\Theta}(M) = \ql(M)$ for every $M \in \Theta$. Since $\Theta$ is infinite, our assertion follows from Theorem \ref{FI1}.

(2) According to \cite[(3.3)]{Fa3}, we have $\dim \Pi(\cG)_\Theta \ge 1$, so that $\Theta$ is locally split and $\Pt(\cG,\Theta) = \Pt(\cG)$. Since $f_\Theta \equiv 1$ for $\bar{T}_\Theta 
\cong \tilde{A}_{12},\, A_\infty^\infty$ (cf.\ \cite[Thm.2, Rem.~p.~328]{HPR2}), our assertion follows directly from Theorem \ref{FI1}. \end{proof}

\bigskip

\subsection{Sets of Jordan types}
Let $M$ be a $\cG$-module. Then
\[\Jt(M) := \{\Jt(M,\alpha_K) \ ; \ \alpha_K \in \Pt(\cG)\}\]
is a finite set, the {\it set of Jordan types of $M$}. Following \cite{CFP}, we say that a $\cG$-module $M$ has {\it constant Jordan type}, provided $\Jt(M)$ is a singleton. In that case, we
call $\Jt(M)$ the Jordan type of $M$.

Our first application, which generalizes \cite[(8.7)]{CFP}, provides a new invariant for locally split Auslander-Reiten components.

\bigskip

\begin{Cor} \label{SJ1} Let $\Theta \subseteq \Gamma_s(\cG)$ be a locally split component. Then
\[ |\Jt(M)| = |\im d^\Theta|\]
for every $M \in \Theta$.\end{Cor}

\begin{proof} Let $M$ be an element of $\Theta$, $\alpha_K, \beta_L \in \Pt(\cG)$ be $\pi$-points. Thanks to Theorem \ref{FI1}, we have $\Jt(M,\alpha_K) = \Jt(M,\beta_L)$ if and only 
if $d^\Theta(\alpha_K) = d^\Theta(\beta_L)$. Consequently, $|\Jt(M)| = |\im d^\Theta|$, as desired. \end{proof}

\bigskip
\noindent
In view of Corollary \ref{SJ1}, the presence of a module $M$ of constant Jordan type in a locally split component $\Theta$ implies that all modules of $\Theta$ have constant Jordan type. 
This is the content of \cite[(8.7)]{CFP}, provided that $\dim \Pi(\cG) \ge 1$. In the remaining case, $\Pi(\cG)$ is a singleton (cf.\ \cite[(3.4)]{CFP}), so that every $\cG$-module has 
constant Jordan type (cf.\ \cite[(4.10)]{FPS}).

\bigskip
\noindent
Let $M$ be a $\cG$-module, $\alpha_K : \fA_{p,K} \lra K\cG$ be a $\pi$-point. Then the isomorphism class of
\[ \bigoplus_{i=1}^{p-1} \alpha_{K,i}(M)[i] \in \modd \fA_{p,K}\]
is referred to as the {\it stable Jordan type $\StJt(M,\alpha_K)$ of $M$ with respect to $\alpha_K$}. We let $\StJt(M)$ be the set of stable Jordan types of $M$. Obviously, a $\cG$-module 
has constant Jordan type if and only if it has constant stable Jordan type.

\bigskip

\begin{Cor} \label{SJ2} Suppose that $\Theta \subseteq \Gamma_s(\cG)$ is a component with $\dim \Pi(\cG)_\Theta \ge 1$ and $\bar{T}_\Theta \not \cong A_\infty$. If $\alpha_K$ is a 
$\pi$-point, then $|\{ \StJt(M,\alpha_K) \ ; \ M \in \Theta\}| = |\im f_\Theta| \le 6$.\end{Cor}

\begin{proof} According to \cite[(3.2),(3.3)]{Fa3}, the tree class of $\Theta$ belongs to $\{\tilde{A}_{12}, \tilde{D}_n, \tilde{E}_{6,7,8}, A_\infty^\infty, D_\infty\}$. It follows from 
Theorem \ref{FI1} that $|\{ \StJt(M,\alpha_K) \ ; \ M \in \Theta\}|$ is equal to the cardinality of $\im f_\Theta$. According to  \cite[Thm.2, Rem. p.328]{HPR2} this number is $1$ for 
$\bar{T}_\Theta = \tilde{A}_{12}, A_\infty^\infty$; $2$ for $\bar{T}_\Theta = \tilde{D}_n, D_\infty$; $3$ for $\bar{T}_\Theta = \tilde{E}_6$; $4$ for $\bar{T}_\Theta =\tilde{E}_7$, 
and $6$ for $\bar{T}_\Theta = \tilde{E}_8$. \end{proof}

\bigskip

\begin{Remark} If $\cG = G_r$ is a Frobenius kernel of a reductive group $G$, then the support of every component $\Theta \subseteq \Gamma_s(G_r)$ of tree class $\bar{T}_\Theta \ne 
A_\infty$ has dimension $\dim \Pi(G_r)_\Theta \ge 1$ and the trees $\tilde{E}_{6,7,8}$ do not occur (cf.\ \cite[(4.1)]{Fa1}). We therefore have  $|\{ \StJt(M,\alpha_K) \ ; \ M \in 
\Theta\}| \le 2$ for every component with $\bar{T}_\Theta \not \cong A_\infty$.\end{Remark}

\bigskip
\noindent
Let $\alpha_K \in \Pt(\cG)$ be a $\pi$-point, $M \in \modd \cG$ be a $\cG$-module. Then
\[ \supp_{\alpha_K}(M) := \{ i \in \{1,\ldots,p\!-\!1\} \ ; \ \alpha_{K,i}(M) \ne 0\}\]
is the {\it $\alpha_K$-support} of $M$. The following immediate consequence of Theorem \ref{FI1} provides useful invariants for locally split components, enabling us to distinguish 
components having the same $\pi$-supports (cf.\ Corollary \ref{Ex2} below).

\bigskip

\begin{Cor} \label{SJ3} Let $\Theta \subseteq \Gamma_s(\cG)$ be a component. If $\alpha_K \in \Pt(\cG,\Theta)$, then
\[ \supp_{\alpha_K}(M) = \{i \in \{1,\ldots,p\!-\!1\} \ ; \ d^\Theta_i(\alpha_K) \ne 0\}\]
for every $M \in \Theta$. \hfill $\square$ \end{Cor}

\bigskip

\subsection{Dominance Ordering and Non-maximal Supports}
Let $\alpha_K$ and $\beta_L$ be $\pi$-points, $M \in \modd \cG$ be a $\cG$-module. In \cite{FPS}, the authors introduce a relation by setting
\[ \alpha_K \unlhd_M \beta_L \ : \Leftrightarrow \dim_K \im t^m_{\alpha_K^\ast(M_K)} \le \dim_K \im t^m_{\beta_L^\ast(M_K)} \ \ \ \ \forall \ m \in\{1,\ldots,p\}.\]
The proof of Proposition \ref{AFP3} readily yields
\[ \alpha_K \unlhd_M \beta_L \ : \Leftrightarrow \ \sum_{i=j}^p (i\!-\!j)\alpha_{K,i}(M_K) \le \sum_{i=j}^p (i\!-\!j)\beta_{L,i}(M_L) \ \ \ \ 1 \le j \le p.\]
In view of \cite[(6.2.2)]{CG},  this relation corresponds the usual dominance ordering on the partitions of $\dim_kM$, associated to $\alpha^\ast_K(M_K)$ and $\beta^\ast_L(M_L)$, 
respectively. More precisely,
\[\alpha_K\sim_M \beta_L : \Leftrightarrow \Jt(M,\alpha_K) = \Jt(M,\beta_L)\]
defines an equivalence relation of $\Pt(\cG)$ and $\unlhd_M$ is a partial ordering on the set of equivalence classes.\footnote{Note that our formula above differs from the one given in 
\cite[\S1]{CFP}, where the authors set  $\alpha_K \unlhd_M \beta_L \ : \Leftrightarrow \ \sum_{i=j}^p i\alpha_{K,i}(M_K) \le \sum_{i=j}^p i\beta_{L,i}(M_L) \ \ \ \ 1 \le j \le p$. The 
choice $p=3$ and $(\alpha_3,\alpha_2,\alpha_1) = (2,0,1)\ ; \ (\beta_3,\beta_2,\beta_1) = (1,2,0)$ defines two partitions $\underline{\alpha}$ and $\underline{\beta}$ of $7$, so that
$\underline{\alpha} \ge \underline{\beta}$ with respect to our definition, while the one of \cite{CFP} leads to $\underline{\alpha} \not \ge \underline{\beta}$.}

Our next result shows that the relation defined by an indecomposable module is an invariant of its stable AR-component.

\bigskip

\begin{Prop} \label{DOMS1} Let $\Theta \subseteq \Gamma_s(\cG)$ be a locally split component. If $\alpha_K$ and $\beta_L$ are $\pi$-points such that $\alpha_K \unlhd_{M_0} \beta_L$ 
for some $M_0 \in \Theta$, then $\alpha_K \unlhd_M \beta_L$ for all $M \in \Theta$.\end{Prop}

\begin{proof} According to Theorem \ref{FI1}, we can find $d_i^\Theta(\alpha_K), d_i^\Theta(\beta_L) \in \NN_0$ such that
\[ \gamma_i(M) = d_i^\Theta(\gamma)f_\Theta(M) \ \ \ \ 1 \le i \le p\!-\!1  \ \ \text{and} \ \ \gamma_{<p}(M) = d_p^\Theta(\gamma)f_\Theta(M) \ \ \ \ \gamma \in 
\{\alpha_K,\beta_L\},\]
for every $M \in \Theta$. Given any $\cG$-module $M \in \Theta$, we have $\gamma_{p}(M) = \frac{1}{p}(\dim_kM-d_p^\Theta(\gamma)f_\Theta(M))$ and thus obtain
\[ \alpha_K \unlhd_M \beta_L \ \Leftrightarrow \ \sum_{i=j}^{p-1} (i\!-\!j) d_i^\Theta(\alpha_K)-(\frac{p\!-\!j}{p})d_p^\Theta(\alpha_K) \le \sum_{i=j}^{p-1} (i\!-\!j) 
d_i^\Theta(\beta_L)-(\frac{p\!-\!j}{p})d_p^\Theta(\beta_L) \]
for $1 \le j \le p$. Since the right-hand side does not depend on $M$, our assertion follows. \end{proof}

\bigskip
\noindent
Let $M$ be a $\cG$-module. We say that $\alpha_K \in \Pt(\cG)$ {\it is maximal for} $\unlhd_M$ if $\beta_L \sim_M \alpha_K$ for every $\beta_L \in \Pt(\cG)$ with $\alpha_K \unlhd_M 
\beta_L$. In \cite{FPS}, the authors introduce the variety of non-maximal supports. Given a $\cG$-module $M$, we write
\[ \widetilde{\Pi}(\cG)_M := \{x \in \Pi(\cG) \ ; \ \exists \ \alpha_K \in x, \ \alpha_K \ \text{is not maximal for} \ \unlhd_M\}.\]
Thanks to \cite[(5.2)]{FPS}, this set is a closed subspace of $\Pi(\cG)_M$, which coincides with $\Pi(\cG)_M$ if and only if $\Pi(\cG)_M \ne \Pi(\cG)$.

\bigskip

\begin{Thm} \label{DOMS2} Let $\Theta \subseteq \Gamma_s(\cG)$ be a component such that $\Pi(\cG)_\Theta = \Pi(\cG)$. Then we have $\widetilde{\Pi}(\cG)_M = 
\widetilde{\Pi}(\cG)_N$ for $M,N \in \Theta$.\end{Thm}

\begin{proof} We first assume that $\dim \Pi(\cG) \ge 1$, so that $\Theta$ is locally split. Suppose that $x \not \in \widetilde{\Pi}(\cG)_M$ and let $\alpha_K \in x$ be a $\pi$-point. Then
$\alpha_K$ is maximal for $\unlhd_M$. If $\beta_L \unrhd_N \alpha_K$ is another $\pi$-point, then Proposition \ref{DOMS1} implies $\beta_L \unrhd_M \alpha_K$, so that 
$\Jt(M,\beta_L) = \Jt(M,\alpha_K)$. As a result, $d^\Theta(\beta_L) = d^\Theta(\alpha_K)$, giving $\Jt(N,\beta_L) = \Jt(N,\alpha_K)$. We conclude that $\alpha_K$ is  maximal for 
$\unlhd_N$, whence $x \not \in \widetilde{\Pi}(\cG)_N$. Consequently, $\widetilde{\Pi}(\cG)_N \subseteq \widetilde{\Pi}(\cG)_M$, implying our assertion.

 If $\dim \Pi(\cG) = 0$, then \cite[(3.4)]{CFP} ensures that $\Pi(\cG)$ is a singleton. We may now apply \cite[(4.10)]{FPS} to see that every $\cG$-module has constant Jordan type.
 Consequently, we have $\widetilde{\Pi}(\cG)_M = \emptyset = \widetilde{\Pi}(\cG)_N$ in that case. \end{proof}

\bigskip

\section{Components containing relatively projective modules} \label{S:CRP}
Throughout this section, we are working over an algebraically closed ground field $k$. We shall study stable AR-components $\Theta \subseteq \Gamma_s(\cG)$ which are not locally split. According to Lemma \ref{AFP2} these are precisely those components that contain a relatively $\alpha_K$-projective $\cG$-module for some $\alpha_K \in \Pt(\cG)$.

\subsection{Quasi-simple relatively projective modules} In order to record detailed information on the functions $\alpha_{K,i}:\Theta \lra \NN_0$ for infinite components $\Theta \subseteq \Gamma_s(\cG)$ that are not $\alpha_K$-split, we introduce the  $(p\times p)$-matrix
\[ A := \left( \begin{matrix} 2 & -1 & 0 & \dots & \dots & \dots & \dots & 0\\
                                          -1 & 2 & -1 & 0       & \dots & \dots & \dots & 0 \\
					       0 & -1 & 2 & -1      &  0      & \dots & \dots & 0 \\
					       0 & 0  & -1 & 2      & -1      & 0        & \dots & 0\\
					\vdots & \vdots & \vdots & \vdots & \ddots & \vdots & \vdots & \vdots\\
					\vdots & \vdots & \vdots & \vdots & \vdots & \ddots & \vdots & \vdots\\
					      0 & \dots  & \dots & \dots      & \dots      &       -1& 2 & -1\\
					      0 & \dots  & \dots & \dots      & \dots      &       0& -1 & 1 \end{matrix} \right)\]
whose coefficients are denoted $a_{ij}$. Note that the $((p\!-\!1)\times (p\!-\!1))$-principal minor of this matrix is just the Cartan matrix of the Dynkin diagram $A_{p-1}$.

Thanks to Proposition \ref{AFP3}, every infinite component $\Theta \subseteq \Gamma_s(\cG)$ which is not locally split is of the form $\Theta \cong \ZZ[A_\infty]/\langle \tau^n \rangle$,
so that we can speak of the quasi-length $\ql(M)$ of a module belonging to $\Theta$. 

\bigskip

\begin{Prop} \label{QRP1} Let $\alpha_K : \fA_{p,K} \lra K\cG$ be a $\pi$-point, $\Theta \subseteq \Gamma_s(\cG)$ be an infinite component such that

{\rm (a)} \ $\Theta$ is not $\alpha_K$-split, and

{\rm (b)} \ if $N \in \Theta$ is relatively $\alpha_K$-projective, then $\ql(N) = 1$.

\noindent
Then there exist a vector $(n_1,\ldots, n_{p-1}) \in \NN_0^{p-1}\setminus\{0\}$ and a quasi-simple module $M \in \Theta$ such that
\[ \alpha_{K,i}(X) = (\alpha_{K,i}(M)\!-\!\sum_{j=1}^{p-1}a_{ij}n_j)\ql(X)+\sum_{j=1}^{p-1}a_{ij}n_j \ \ \ \ \ \ \ \ 1 \le i \le p\!-\!1\]
for every $X \in \Theta$. \end{Prop}

\begin{proof} Given $i \in \{1,\ldots,p\}$, we put $N_i := K\cG\!\otimes_{\fA_{p,K}}\![i]$. Conditions (a) and (b) in conjunction with Lemma \ref{AFP2} ensure the existence of a 
quasi-simple module $M \in \Theta$ that is relatively $\alpha_K$-projective. We define $n_i \in \NN_0$ to be the multiplicity of $M_K$ in $N_i$, that is,
\[ N_i \cong n_iM_K \oplus N'_i \ \ \ \text{with} \ \ \ M_K\nmid N'_i.\]
Thanks to Lemma \ref{AFP2}(1), at least one of the $n_i$ is not zero. Since $M_K$ is not projective, we have $n_p = 0$. 

We consider the almost split sequence
\[ \fE_M: \ \ \ \ \ \ (0) \lra \tau_\cG(M) \lra E \lra M \lra (0)\]
terminating in $M$. Since $k$ is algebraically closed, \cite[(3.8)]{Ka} ensures that
\[ \fE_M\!\otimes_k\!K : \ \ \ \ \ \ (0) \lra \tau_\cG(M)_K \lra E_K \lra M_K \lra (0)\]
is the almost split sequence terminating in $M_K$. Let $\varphi : N_i \lra M_K$ be a homomorphism and write $\varphi = (\varphi_1,\ldots,\varphi_{n_i},\varphi')$ with $\varphi_t
\in \End_{\cG_K}(M_K)$ for $1 \le t \le n_i$ and $\varphi' \in \Hom_{\cG_K}(N'_i,M_K)$. Then $\varphi$ is split surjective if and only if $\varphi_t \not \in \Rad(\End_{\cG_K}(M_K))$
for some $t \in \{1,\ldots,n_i\}$. Since $\fE_M\!\otimes_k\!K$ is almost split, there results an exact sequence
\begin{eqnarray*}
(0) \lra \Hom_{\cG_K}(N_i,\tau_\cG(M)_K) &\lra & \Hom_{\cG_K}(N_i,E_K) \lra n_i\Rad(\End_{\cG_K}(M_K))\oplus \Hom_{\cG_K}(N'_i,M_K)\\
                                                                       &\lra &(0)
\end{eqnarray*}
for every $i \in\{1,\ldots, p\}$.  As $k$ is algebraically closed, the right-hand term above has dimension $\dim_K\Hom_{\cG_K}(N_i,M_K)-n_i$ and Frobenius reciprocity implies
\begin{eqnarray*} \dim_K \Hom_{\fA_{p,K}}([i],\alpha^\ast_K(E_K)) &=&\dim_K \Hom_{\fA_{p,K}}([i],\alpha^\ast_K(M_K)) + \dim_K 
\Hom_{\fA_{p,K}}([i],\alpha^\ast_K(\tau_\cG(M)_K))\\
&& -n_i.
\end{eqnarray*}
From the formula
\[ \dim_K\Hom_{\fA_{p,K}}([s],[t]) = \min\{s,t\} \ \ \ \ \ \ \ \ \forall \ s,t \in \{1,\ldots,p\}\]
we get
\[ \sum_{\ell=1}^p \min\{i,\ell\} \alpha_{K,\ell}(E) = \sum_{\ell=1}^p \min\{i,\ell\}(\alpha_{K,\ell}(M) +\alpha_{K,\ell}(\tau_\cG(M)))-n_i \ \ \ \ \ 1 \le i \le p.\]
Let $B := ( \min\{i,\ell\})_{1\le i,\ell\le p} \in {\rm Mat}_p(\ZZ)$ and $n = (n_1,\ldots,n_p)^{\rm tr} \in \ZZ^p$. The above identities then amount to
\[Bx = By-n,\]
where
\[x := (\alpha_{K,1}(E),\ldots,\alpha_{K,p}(E))^{\tr} \ \text{and} \ y := (\alpha_{K,1}(M)+\alpha_{K,1}(\tau_\cG(M)),\ldots,\alpha_{K,p}(M)+\alpha_{K,p}(\tau_\cG(M)))^{\tr}.\] 
Direct computation reveals that $B$ is invertible with inverse $B^{-1} = A$. Consequently, $x = y -An$, whence
\[(\ast) \ \ \ \ \ \ \ \ \ \alpha_{K,i}(E) = \alpha_{K,i}(M)+\alpha_{K,i}(\tau_\cG(M))-\sum_{j=1}^{p-1}a_{ij}n_j \ \ \ \ \ \ \ \text{for} \ \ 1 \le i \le p.\]
Since $\ql(M)=1$, we have
\[ E \cong M_2\oplus P,\]
where $M_2 \in \Theta$ has quasi-length $2$ and $P$ is indecomposable projective or $(0)$. Hence ($\ast$) implies
\[ (\ast\ast) \ \ \ \ \ \ \ \ \ \alpha_{K,i}(M_2) = \alpha_{K,i}(M)+\alpha_{K,i}(\tau_\cG(M))-\sum_{j=1}^{p-1}a_{ij}n_j\]
for $1 \le i \le p\!-\!1$.

Now suppose that  $i \le p\!-\!1$. Proposition \ref{AFP3}(1)  yields $\alpha_{K,i} \circ \tau_\cG = \alpha_{K,i}$, while Corollary \ref{AFP5} and condition (b) ensure that
$\ell(\alpha_{K,i}) \le 2$. Thus, if $\sum_{j=1}^{p-1}a_{ij}n_j = 0$, then $\alpha_{K,i}$ is additive, and our assertion follows from Lemma \ref{SF4}. Alternatively, 
$\ell(\alpha_{K,i})=2$, and Lemma \ref{SF4} implies
\begin{eqnarray*}
\alpha_{K,i}(X) & = & (\alpha_{K,i}(M_2)-\alpha_{K,i}(M))(\ql(X)\!-\!2)+\alpha_{K,i}(M_2)\\
                          & = &(\alpha_{K,i}(M)\!-\!\sum_{j=1}^{p-1}a_{ij}n_j)\ql(X) +\sum_{j=1}^{p-1}a_{ij}n_j
\end{eqnarray*}
for every $X \in \Theta$. \end{proof}

\bigskip

\begin{Remarks} (1) Suppose that $\Theta \cong \ZZ[A_\infty]/\langle \tau \rangle$ is a homogeneous tube. If $E = M_2\oplus P$, with $P \ne (0)$, then $M \cong P/\Soc(P)$ and
$M \cong \tau_\cG(M) \cong \Rad(P)$, see \cite[(V.5.5)]{ARS}. Thus, $\Rad(P) \cong P/\Soc(P)$ and the block $\cB \subseteq k\cG$ containing $P$ is a Nakayama algebra. Hence
$\Theta$ is finite, a contradiction. It now follows from ($\ast$) that the formula of Proposition \ref{QRP2}  also holds for $i=p$.

(2) The foregoing result can be stated more formally by letting $n_j(M,\alpha_K)$ be the multiplicity of $M_K$ as a summand of $K\cG\!\otimes_{\fA_{p,K}}\![j]$. As noted in 
\cite[(3.2)]{Fa3}, there is an isomorphism
\[\tau_{\cG_K}(K\cG\!\otimes_{\fA_{p,K}}\![j]) \oplus ({\rm proj.}) \cong K\cG\!\otimes_{\fA_{p,K}}\![j],\]
whence $n_j(M,\alpha_K) = n_j(\tau_\cG(M),\alpha_K)$ for $1\le j \le p\!-\!1$. Thus, Proposition \ref{QRP1} holds for {\it every} quasi-simple module $M \in \Theta$. \end{Remarks}

\bigskip
\noindent
Let $\Theta \subseteq \Gamma_s(\cG)$ be an infinite component that is not $\alpha_K$-split for some $\pi$-point $\alpha_K \in \Pt(\cG)$. The utility of the above result crucially depends 
on

(a) \ all relatively $\alpha_K$-projective modules $M \in \Theta$ being quasi-simple, and

(b) \ the knowledge of the multiplicty $n_j$ of $M_K$ as a summand of $K\cG\!\otimes_{\fA_{p,K}}\![j]$ for $1\le j \le p\!-\!1$.

\noindent
Proposition \ref{QRP3} below addresses these issues for relatively $\alpha_K$-projective modules whose tops contain one-dimensional modules. We begin with the following easy observation:

\bigskip

\begin{Lem} \label{QRP2} Let $\alpha_K \in \Pt(\cG)$ be a $\pi$-point, $j \in \{1,\ldots,p\}$. If $S$ is a simple $\cG$-module, then $S_K$ is a simple $\cG_K$-module that
occurs in $\Top_{\cG_K}(K\cG\!\otimes_{\fA_{p,K}}\![j])$ with multiplicity $\dim_K\ker t^j_{\alpha_K^\ast(S_K)}$. \end{Lem}

\begin{proof} Since $k$ is algebraically closed, we have $\End_\cG(S) \cong k$, whence $\End_{\cG_K}(S_K) \cong K\!\otimes_k\!k \cong K$. Hence $S_K$ is simple, and Frobenius
reciprocity implies
\[ \Hom_{\cG_K}(K\cG\!\otimes_{\fA_{p,K}}\![j],S_K) \cong \Hom_{\fA_{p,K}}([j],\alpha_K^\ast(S_K)).\]
Since $\dim_K\End_{\cG_K}(S_K) = 1$, the dimension of the former space counts the multiplicity of $S_K$ in $\Top_{\cG_K}(K\cG\!\otimes_{\fA_{p,K}}\![j])$. Writing 
$\alpha_K^\ast(S_K) \cong \bigoplus_{i=1}^p\alpha_{K,i}(S)[i]$, we obtain $\dim_K\Hom_{\fA_{p,K}}([j],\alpha^\ast_K(S_K)) = \sum_{i=1}^p \min\{i,j\}\alpha_{K,i}(S)$. As 
observed in the proof of Proposition\ref{AFP3}, the latter number coincides with $\dim_K\ker t^j_{\alpha_K^\ast(S_K)}$. \end{proof}

\bigskip

\begin{Prop} \label{QRP3} Suppose that $p\ge 3$, and let $\cG$ be a finite group scheme of infinite representation type. If $M$ is a non-projective, indecomposable, relatively
$\alpha_K$-projective $\cG$-module such that $\Top_\cG(M)\oplus \Soc_\cG(M)$ contains a one-dimensional submodule, then the following statements hold:

{\rm (1)} \ The component $\Theta \subseteq \Gamma_s(\cG)$ containing $M$ is isomorphic to $\ZZ[A_\infty]/\langle \tau^m \rangle$ for some $m \ge 1$.

{\rm (2)} \ The module $M$ is quasi-simple.

{\rm (3)} \ There exists a non-zero vector $(n_1,\ldots ,n_{p-1}) \in \{0,1\}^{p-1}$ such that
\[ \alpha_{K,i}(X) = (\alpha_{K,i}(M)\!-\!\sum_{j=1}^{p-1}a_{ij}n_j)\ql(X)+\sum_{j=1}^{p-1}a_{ij}n_j \ \ \ \ \ \ \ \ 1 \le i \le p\!-\!1\]
for every $X \in \Theta$.\end{Prop}

\begin{proof} (1) By assumption, there exists an algebra homomorphism $\lambda : k\cG \lra k$ such that the one-dimensional $\cG$-module $k_\lambda$ associated to $\lambda$ occurs
as a summand of $\Top_\cG(M)\oplus \Soc_\cG(M)$.

In view of Proposition \ref{AFP3}, the component $\Theta$ is either finite or an infinite tube. In the former case, Auslander's Theorem \cite[(VII.2.1)]{ARS} implies that $\Theta$ contains all
non-projective indecomposable modules of the block $\cB_M$ containing $M$. Thus, $\cB_M$ has finite representation type and $k_\lambda$ belongs to $\cB_M$. The convolution
$\id_{k\cG}\ast\lambda$ is an automorphism of $k\cG$ that sends the block $\cB_M$ onto the principal block $\cB_0(\cG)$ of $k\cG$. Consequently, $\cB_0(\cG)$ is representation-finite,
and \cite[(3.1)]{FV1} implies that $k\cG$ enjoys the same property, a contradiction.

(2) We first assume that $\Top_\cG(M)$ contains a one-dimensional module. Since $M$ is relatively $\alpha_K$-projective, Lemma \ref{AFP2} provides an element $j \in \{1,\ldots, 
p\!-\!1\}$ such that $M_K$ is a direct summand of $K\cG\!\otimes_{\fA_{p,K}}\![j]$. Then $K_\lambda \cong k_\lambda\!\otimes_K\!K$ is a direct summand of $\Top_{\cG_K}(M_K)$ 
(cf.\ \cite[(3.5)]{Ka}), and Lemma \ref{QRP2} shows that $K_\lambda$ occurs in $\Top(K\cG\!\otimes_{\fA_{p,K}}\![j])$ with multiplicity $1$. Accordingly, $M_K$ is the unique 
indecomposable direct summand of $K\cG\!\otimes_{\fA_{p,K}}\![j]$ such that $K_\lambda\! \mid\! \Top_{\cG_K}(M_K)$.

As observed in the preceding remarks, we have
\[ \tau_{\cG_K}(K\cG\!\otimes_{\fA_{p,K}}\![j]) \oplus ({\rm proj.}) \cong K\cG\!\otimes_{\fA_{p,K}}\![j],\]
so that $\tau_{\cG_K}(M_K)$ is also a direct summand of $K\cG\!\otimes_{\fA_{p,K}}\![j]$. We put $\ell := \min\{n \ge 1 \ ; \ \tau_\cG^n(M) \cong M\}$. Since
$\tau_\cG^\ell|_\Theta$ is an automorphism of  $\Theta$ preserving quasi-lengths, we conclude that $\tau_\cG^\ell|_\Theta \cong \id_\Theta$. Consequently, $\Theta \cong
\ZZ[A_\infty]/\langle \tau^\ell\rangle$.

Suppose that the $\cG$-module $k_\lambda$ belongs to $\Theta$. Then $k_\lambda$ has complexity $\le 1$, and the formula $\cx_\cG(M\!\otimes_k\!k_\lambda) \le \cx_\cG(k_\lambda)$
implies that every $\cG$-module has complexity $\le 1$. By general theory (cf.\ \cite[(6.8)]{Wa}),
\[ \cG = \cG^0 \rtimes \cG_{\rm red}\]
is the semidirect product of an infinitesimal normal subgroup $\cG^0$ and a reduced group $\cG_{\rm red}$. Thus, the complexities of the trivial modules for the subgroups $\cG_{\rm red}$
and $\cG^0$ are also bounded by $1$.

Suppose that $\cx_{\cG_{\rm red}}(k)=1$. Then $\cG_{\rm red}$ is not linearly reductive and there exists a closed subgroup $\cP \subseteq \cG_{\rm red}$ such that $\cP(k) \cong
\ZZ/(p)$. If $\cx_{\cG^0}(k)=1$, then the proof of \cite[(3.1)]{FV1} ensures the existence of a closed subgroup $\cU \subseteq \cG^0$ such that $\cU \cong \GG_{a(1)}$ is isomorphic
to the first Frobenius kernel of the additive group and with $\cP$ acting trivially on $\cU$. Consequently,
\[ 2 = \cx_{\cU\times \cP}(k) \le \cx_\cG(k),\]
a contradiction. We conclude that $\cx_{\cG^0}(k) = 0$. Hence the normal subgroup $\cG^0 \unlhd \cG$ is linearly reductive, and \cite[(1.1)]{Fa2}, provides a normal subgroup $\cN
\unlhd \cG_{\rm red}$ such that

(a) \ $p \nmid {\rm ord}(\cN(k))$, and

(b) \ the principal block $\cB_0(\cG)$ is  isomorphic to $\cB_0(\cG_{\rm red}/\cN)$.

\noindent
In particular, the trivial module of the finite group $\cG(k)/\cN(k)$ has complexity $1$ and is thus periodic (cf.\ \cite[(5.10.4)]{Be2}). Owing to \cite[(XII.11.6)]{CE}, this implies that the
Sylow-$p$-subgroups of $\cG(k)/\cN(k)$ are either cyclic, or generalized quaternion. In the former case, Higman's Theorem \cite[Thm.~4]{Hi} ensures that $\cB_0(\cG)\cong
\cB_0(\cG_{\rm red}/\cN)$ has finite representation type, and \cite[(3.1)]{FV1} yields a contradiction. Alternatively, $p=2$, which contradicts our current assumption.

Consequently, $\cx_{\cG_{\rm red}}(k) = 0$, and the group $\cG_{\rm red}$ is linearly reductive. A consecutive application of \cite[(2.7)]{FV1} and \cite[(3.1)]{FV1} now shows that
$\cG$ is representation-finite, a contradiction.

Thus, $k_\lambda \not \in \Theta$, and the defining property of almost split sequences implies that the function $X \mapsto \dim_k\Hom_\cG(\tau_\cG^n(X),k_\lambda)$ is additive on 
$\Theta$ for every $n \ge 0$. Since $\Theta \cong \ZZ[A_\infty]/\langle \tau^\ell \rangle$, the map
\[ d_\lambda : \Theta \lra \NN \ \ ; \ \ X \mapsto \sum_{n=0}^{\ell-1}\dim_k\Hom_\cG(\tau_\cG^n(X),k_\lambda)\]
is a $\tau_\cG$-invariant, additive function. By Lemma \ref{SF4}, there exists $r \in \NN$ with
\[ d_\lambda(X) = r\ql(X) \ \ \ \ \forall \ X \in \Theta.\]
Owing to \cite[(2.5),(3.6)]{Ka}, the modules $\tau_{\cG_K}^n(M_K) \ \ 0\le n \le \ell\!-\!1$ are pairwise non-isomorphic summands of $K\cG\!\otimes_{\fA_{p,K}}\![j]$. Consequently,
$d_\lambda(M) = 1$, so that $\ql(M) = 1$ and $M$ is quasi-simple.

If $\Soc_\cG(M)$ contains a one-dimensional module, then observing the fact that $k\cG$ is a Frobenius algebra, we conclude that the top of the injective hull $E(M)$ of $M$ contains a 
one-dimensional module. Consequently, $\Top_\cG(\Omega^{-1}_\cG(M))$ has a one-dimensional constituent. The isomorphism
\[ (\ast) \ \ \ \ \ \Omega^{-1}_{\cG_K}(K\cG\!\otimes_{\fA_{p,K}}\![i]) \oplus ({\rm proj.}) \cong K\cG\!\otimes_{\fA_{p,K}}\![p\!-\!i]\]
ensures that $\Omega^{-1}_\cG(M)$ is relatively $\alpha_K$-projective. As a result, the module $\Omega^{-1}_\cG(M)$ is quasi-simple. Since $\Omega^{-1}_\cG$ induces an 
automorphism of the quiver $\Gamma_s(\cG)$, the module $M$ enjoys the same property.

(3) Now let $N \in \Theta$ be an arbitrary relatively $\alpha_K$-projective $\cG$-module. Since $\Theta \cong \ZZ[A_\infty]/\langle\tau^m\rangle$ and $M$ is quasi-simple, there exist a 
surjection $\tau_\cG^\ell(N)\lra M$ as well as an injection $M \lra \tau_\cG^s(M)$ for some $\ell,s \in \{0,\ldots,m\!-\!1\}$. As $\tau_\cG^\ell(N)$ and $\tau_\cG^s(N)$ are also 
relatively $\alpha_K$-projective, part (2) ensures that $\tau_\cG^\ell(N)$ or $\tau_\cG^s(N)$ are quasi-simple. Thus, $N$ is also quasi-simple.

Suppose that $k_\lambda \subseteq \Top_\cG(M)$. Let $j \in \{1,\ldots,p\!-\!1\}$. If $n_j$ is the multiplicity of $M_K$ in $K\cG\!\otimes_{\fA_{p,K}}\![j]$, then the one-dimensional 
module $K_\lambda$ occurs in $\Top_{\cG_K}(K\cG\!\otimes_{\fA_{p,K}}\![j])$ with multiplicity $m_j \ge n_j$. Lemma \ref{QRP2} now shows that $m_j$ is bounded by $1$. 

If $k_\lambda \subseteq \Soc_\cG(M)$, then $n_j(\Omega^{-1}_\cG(M),\alpha_K) \in \{0,1\}$. In view of ($\ast$), we
have $n_j(M,\alpha_K) = n_{p-j}(\Omega^{-1}_\cG(M),\alpha_K)$, so that $n_j(M,\alpha_k) \in \{0,1\}$. The assertion thus follows from Proposition \ref{QRP1}. \end{proof}

\bigskip

\begin{Remarks} (1) The proof shows that Proposition \ref{QRP3} also holds for $p=2$ as long as the Sylow-$2$-subgroup of $\cG(k)$ is not generalized quaternion.

(2) Let $\alpha_k \in \Pt(\cG)$ be a $\pi$-point that is defined over $k$, and denote by $P(k)$ the projective cover of the trivial $\cG$-module $k$. The canonical map
\[ \mu : k\cG\!\otimes_{\fA_{p,k}}\!k \lra k \ \ ; \ \ a\otimes 1 \mapsto \varepsilon(a)\]
affords an $\fA_{p,k}$-linear splitting $\gamma : k \lra k\cG\!\otimes_{\fA_{p,k}}\!k \ ; \ s \mapsto 1 \otimes s$. If
\[ k\cG\!\otimes_{\fA_{p,k}}\!k = P(k)\oplus X,\]
then Lemma \ref{QRP2} implies that $k$ is not a top composition factor of $X$, whence $\mu(X) = (0)$. Consequently, the trivial $\fA_{p,k}$-module $k$ is a direct summand of the
projective module $\alpha_k^\ast(P(k))$, a contradiction. As a result, $k\cG\!\otimes_{\fA_{p,k}}\!k$ possesses a non-projective, indecomposable constituent with $k$ being a composition
factor of its top. The arguments of Lemma \ref{TG1} below show that this applies to every one-dimensional $\cG$-module. \end{Remarks}

\subsection{Special types of $\pi$-points}
In this section we shall provide another two criteria ensuring that certain relatively projective modules are quasi-simple. We fix a $\pi$-point $\alpha_K \in \Pt(\cG)$ such that $[\alpha_K] \in 
\Pi(\cG)$ is closed.

\bigskip

\begin{Prop} \label{STP1} Let $M$ be a non-projective indecomposable relatively $\alpha_K$-projective module which is contained in an infinite component $\Theta \subseteq 
\Gamma_s(\cG)$. If either

{\rm (a)} \ $K\cG\alpha_K(t)$ is an ideal of $K\cG$, or

{\rm (b)} \ there exists an abelian unipotent normal subgroup $\cU \subseteq \cG_K$ of complexity $\cx_\cU(K)=1$ such \indent \indent \indent that $\im \alpha_K \subseteq K\cU$,

\noindent
then $M$ is quasi-simple. \end{Prop}

\begin{proof} Assuming (a), we have $\alpha_K(t)K\cG \subseteq K\cG\alpha_K(t)$, whence $\alpha_K(t)^iK\cG \subseteq  K\cG\alpha_K(t)^i$ for all $1 \le i \le p\!-\!1$. Thus, $I_i := 
K\cG\alpha_K(t)^i$ is an ideal of $K\cG$. Since $M$ is relatively $\alpha_K$-projective, there exists $i \in \{1,\ldots,p\!-\!1\}$ such that $M_K \!\mid\! K\cG\!\otimes_{\fA_{p,K}}\![i]$. 
Since the homomorphism $\alpha_K$ factors through an abelian unipotent subgroup, it is also right flat and the natural map $K\cG \lra K\cG\!\otimes_{\fA_{p,K}}\![i] \ ; \ a \mapsto 
a\otimes1$ induces an isomorphism $K\cG/I_i \cong K\cG\!\otimes_{\fA_{p,K}}\![i]$. Thus, $M_K$ is isomorphic to a principal indecomposable $(K\cG/I_i)$-module and hence has a 
simple top. Consequently, $M$ enjoys the same property (cf.\ \cite[(3.3)]{Ka}).

Assuming (b), we obtain that $M_K$ is a direct summand of $K\cG\!\otimes_{K\cU}\!N$ for some indecomposable $K\cU$-module $N$. Since $\cx_\cU(K)=1$, it follows that $K\cU \cong 
K[X]/(X^{p^n})$ as well as $N \cong K[X]/(X^i)$ for some $i \in \{1,\ldots,p^n\!-\!1\}$. Thus, $K\cG\!\otimes_{K\cU}\!N \cong K\cG /K\cG x^i$, with $x$ being the canonical 
generator of $K\cU$. As $\cU$ is normal, the left ideal $K\cG K\cU^\dagger = K\cG x$ is an ideal of $K\cG$, and the arguments employed above guarantee that $\Top_\cG(M)$ is simple.

Recall that the indecomposable $\cG$-module $\Omega^{-1}_\cG(M)$ is also relatively $\alpha_K$-projective. Consequently, $\Omega^{-1}_\cG(M)$ has a simple top, implying that $M$ 
has a simple socle. 

Since $\tau_\cG^\ell(M)$ is relatively $\alpha_K$-projective for all $\ell \in \ZZ$  and $\Theta$ is infinite, we may apply \cite[(1.2)]{Er0} to conclude that $M$ is quasi-simple. \end{proof}

\bigskip

\begin{Cor} \label{STP2} Suppose that $\im \alpha_K$ is contained in the center $\fZ(K\cG)$ of $K\cG$. Let $M$ be a non-projective indecomposable relatively $\alpha_K$-projective 
module which is contained in an infinite component $\Theta \subseteq \Gamma_s(\cG)$. Then there exist $j \in \{1,\ldots,p\!-\!1\}$ and $m,n \in \NN$ with $m\ge 2n$ such that
\[ \alpha_{K,i}(X) = \left\{ \begin{array}{cl} (m\!-\!2n)\ql(X)+2n & \text{for} \ i=j\\ n(\ql(X)\!-\!1) & \text{for} \ i=j\!-\!1,j\!+\!1\\ 0 & \text{for} \ i \neq j,j\!+\!1,j\!-\!1 
\end{array} \right.\]
for every $X \in \Theta$ and $i \in \{1,\ldots,p\!-\!1\}$.\end{Cor}

\begin{proof} By assumption, $K\cG\alpha_K(t)$ is an ideal of $K\cG$. Let $j \in \{1,\ldots,p\!-\!1\}$ be such that $M_K\!\mid \!K\cG\!\otimes_{\fA_{p,K}}\![j]$. Since $\alpha_K(t)$ is 
central, we obtain $\alpha_K^\ast(K\cG\!\otimes_{\fA_{p,K}}\![j]) \cong (\frac{\dim_kk\cG}{p})[j]$. Hence there exists $m\in \NN$ with
\[ \alpha^\ast_K(M_K) \cong m[j].\]
This readily implies $n_i(M,\alpha_K) = n\delta_{ij}$ $1\le i \le p\!-\!1$ for some $n \in \NN$. The asserted formula now follows by applying Propositions \ref{STP1} and \ref{QRP1} 
consecutively. Since $\alpha_{K,j}(X) \ge 0$ for all $X \in \Theta$, we also conclude that $m \ge 2n$.  \end{proof}

\subsection{Trigonalizable Group Schemes}
{\it Throughout this section we consider a trigonalizable finite group scheme $\cG$}. By definition, all simple $\cG$-modules are one-dimensional. The Theorem of Lie-Kolchin ensures that the
Frobenius kernels of the smooth connected solvable algebraic groups belong to this class.

By general theory (cf.\ \cite[(IV,\S2,3.5]{DG}), $\cG = \cU \rtimes \ccD$ is a semidirect product of a unipotent normal subgroup $\cU$ and diagonalizable factor $\ccD$. Thus, the 
coordinate ring $k[\ccD] = kX(\ccD)$ is the group algebra of the finite group $X(\ccD)$ of characters of $\ccD$. If $K\!:\!k$ is a field extension, then
\[ K[\ccD_K] \cong k[\ccD]\otimes_kK \cong KX(\ccD),\]
so that $\ccD_K$ is diagonalizable with character group $X(\ccD_K) = X(\ccD)$.

Since $\cU$ is unipotent, the canonical restriction map $X(\cG) \lra X(\ccD)$ is an isomorphism. Given $\lambda \in X(\cG)$, we let $k_\lambda$ be the one-dimensional $\cG$-module 
defined by $\lambda$. Thus, $\{k_\lambda \ ; \ \lambda \in X(\cG)\}$ is a complete set of representatives for the isomorphism classes of the simple $\cG$-modules. By the above, these 
observations also apply to the group $\cG_K$, given by an extension field $K$ of $k$. We will henceforth identify $X(\cG_K)$ with $X(\cG)$.

\bigskip

\begin{Lem} \label{TG1} Let $\alpha_K \in \Pt(\cG)$ be a $\pi$-point, $j \in \{1,\ldots,p\}$.

{\rm (1)} \ $\Top_{\cG_K}(K\cG\!\otimes_{\fA_{p,K}}\![j]) \cong \bigoplus_{\lambda \in X(\cG)}K_\lambda$.

{\rm (2)} \ We have $(K\cG\!\otimes_{\fA_{p,K}}\![j])\!\otimes_K\!K_\lambda \cong K\cG\!\otimes_{\fA_{p,K}}\![j]$ for all $\lambda \in X(\cG)$.

{\rm (3)} \ We have $\tau_{\cG_K}(K\cG\!\otimes_{\fA_{p,K}}\![j]) \cong K\cG\!\otimes_{\fA_{p,K}}\![j]$ for $1\le j \le p\!-\!1$. \end{Lem}

\begin{proof} (1) This follows directly from Lemma \ref{QRP2}.

(2) Since $\cG/\cU$ is diagonalizable, every $\pi$-point $\alpha_K$ of $\cG$ factors through $\cU$. Let $V[j] := K\cU\!\otimes_{\fA_{p,K}}\![j]$, so that
\[ K\cG\!\otimes_{\fA_{p,K}}\![j] \cong K\cG\!\otimes_{K\cU}\!V[j].\]
The tensor identity \cite[(I.3.6)]{Ja} then yields
\[ (K\cG\!\otimes_{K\cU}\!V[j])\!\otimes_k\!K_\lambda \cong K\cG\!\otimes_{K\cU}\!(V[j]\!\otimes_K\!(K_\lambda|_\cU)) \cong K\cG\!\otimes_{K\cU}\!V[j]\]
for every character $\lambda \in X(\cG)$.

(3) Let $j \in \{1,\ldots,p\!-\!1\}$. Consider the exact sequence
\[(0) \lra K\cG\!\otimes_{\fA_{p,K}}\![p\!-\!j] \lra K\cG \stackrel{\pi}{\lra} K\cG\!\otimes_{\fA_{p,K}}\![j] \lra (0).\]
According to (1), we have $\Top_{\cG_K}(K\cG) = \Top_{\cG_K}(K\cG\!\otimes_{\fA_{p,K}}\![j])$, so that $(K\cG,\pi)$ is a
projective cover of $K\cG\!\otimes_{\fA_{p,K}}\![j]$. Consequently, $\Omega_{\cG_K}( K\cG\!\otimes_{\fA_{p,K}}\![j]) \cong K\cG\!\otimes_{\fA_{p,K}}\![p-j]$, and
$\Omega_{\cG_K}^2(K\cG\!\otimes_{\fA_{p,K}}\![j]) \cong K\cG\!\otimes_{\fA_{p,K}}\![j]$. Since $\tau_{\cG_K}$ is the composite of $\Omega^2_{\cG_K}$ with the functor
$M \lra M\!\otimes_K\!K_\zeta$, defined by the modular function $\zeta \in X(\cG)$, part (2) yields the desired isomorphism. \end{proof}

\bigskip
\noindent
Given a $\cG$-module $M$, we let
\[ {\rm Stab}_{X(\cG)}(M) := \{ \lambda \in X(\cG) \ ; \ M\!\otimes_k\!K_\lambda \cong M\}\]
be the {\it stabilizer} of $M$. Let $K$ be an extension field of $k$. By the above observations, we have $M_K\!\otimes_K\!K_\lambda \cong (M\!\otimes_k\!k_\lambda)_K$ for every
$\lambda \in X(\cG)$. It now follows from \cite[(2.5)]{Ka} that
\[ {\rm Stab}_{X(\cG)}(M) = {\rm Stab}_{X(\cG)}(M_K).\]
Recall that a $\cG_K$-module $N$ is said {\it to be defined over k} if there exists a $\cG$-module $M$ such that $N \cong M_K$.

\bigskip

\begin{Lem} \label{TG2} Let $M \in \modd \cG$ be a non-projective indecomposable $\cG$-module, $\alpha_K \in \Pt(\cG)$ be a $\pi$-point. If $M$ is relatively $\alpha_K$-projective,
then there exists a unique $j \in \{1,\ldots,p\!-\!1\}$ and a subset $X_M \subseteq X(\cG)$ of cardinality $[X(\cG)\!:\!{\rm Stab}_{X(\cG)}(M)]$ such that
\[ K\cG\!\otimes_{\fA_{p,K}}\![j] \cong \bigoplus_{\lambda \in X_M} M_K\!\otimes_K\!K_\lambda.\]
In particular, $\dim_k\Top_\cG(M) = {\rm ord}({\rm Stab}_{X(\cG)}(M))$, and the module $K\cG\!\otimes_{\fA_{p,K}}\![j]$ is defined over $k$. \end{Lem}

\begin{proof} By virtue of Lemma \ref{AFP2}, the $\cG_k$-module $M_K$ is a direct summand of $K\cG\!\otimes_{\fA_{p,K}}\![j]$ for some $j \in \{1,\ldots,p\!-\!1\}$. We decompose
\[ K\cG\!\otimes_{\fA_{p,K}}\![j] \cong M_1\oplus \cdots \oplus M_n\]
into its indecomposable constituents, where $M_1 \cong M_K$. In view of (1) of Lemma \ref{TG1}, the corresponding decomposition of its top yields a partition
\[ X(\cG) = \bigsqcup_{i=1}^n X_i,\]
where $\Top_{\cG_K}(M_i) \cong \bigoplus_{\lambda \in X_i}K_\lambda$. Note that the isoclass of $M_i$ is uniquely determined by $X_i$.

Let $\lambda \in X(\cG)$. Owing to Lemma \ref{TG1}, the $K\cG$-module $M_1\!\otimes_K\!K_\lambda$ is an indecomposable direct summand of $K\cG\!\otimes_{\fA_{p,K}}\![j]$. 
Hence there exists exactly one $i(\lambda) \in \{1,\ldots,n\}$ such that $M_1\!\otimes_K\!K_\lambda \cong M_{i(\lambda)}$. It readily follows that $X_{i(\lambda)} = \lambda\dact X_1$.
Let $X_M \subseteq X(\cG)$ be a complete set of left coset representatives for ${\rm Stab}_{X(\cG)}(M)$. Since ${\rm Stab}_{X(\cG)}(M)\dact X_1 = X_1$, we obtain
\[ X(\cG) = \bigcup_{\lambda \in X(\cG)} \lambda\dact X_1 = \bigcup_{\lambda \in X_M} \lambda\dact X_1 = \bigsqcup_{\lambda \in X_M} \lambda\dact X_1.\]
Thus, $N := \bigoplus_{\lambda \in X_M} M_1\!\otimes_K\!K_\lambda$ is a direct summand of $K\cG\!\otimes_{\fA_{p,K}}\![j]$ with $\Top_{\cG_K}(N) \cong \bigoplus_{\lambda \in
X(\cG)}K_\lambda \cong \Top_{\cG_K}(K\cG\!\otimes_{\fA_{p,K}}\![j])$, whence $K\cG\!\otimes_{\fA_{p,K}}\![j] \cong \bigoplus_{\lambda \in X_M} M_K\!\otimes_K\!K_\lambda$.
By definition, we have $|X_M|= [X(\cG)\!:\! {\rm Stab}_{X(\cG)}(M)]$. Since
\[ \frac{\dim_kk\cG}{p}j = [X(\cG)\!:\!{\rm Stab}_{X(\cG)}(M)]\dim_kM,\]
the number $j$ is completely determined by $M$.

In view of \cite[(3.5)]{Ka}, we have $\Top_\cG(M)_K \cong \Top_{\cG_K}(M_K)$, whence
\[ \dim_k\Top_\cG(M) = \dim_K \Top_{\cG_K}(M_K) = |X_1| = {\rm ord}({\rm Stab}_{X(\cG)}(M_K)) = {\rm ord}({\rm Stab}_{X(\cG)}(M)),\]
as desired. \end{proof}

\bigskip

\begin{Thm} \label{TG3} Let $M$ be a non-projective, indecomposable, relatively $\alpha_K$-projective $\cG$-module. If $\cG$ is not of finite representation type, then the following 
statements hold:

{\rm (1)} \ The component $\Theta \subseteq \Gamma_s(\cG)$ containing $M$ is a tube of rank $\ell$, where $\ell\! \mid\! {\rm ord}(X(\cG))$.

{\rm (2)} \ The module $M$ is quasi-simple.

{\rm (3)}  \ There exists $j \in \{1,\ldots,p\!-\!1\}$ such that
\[ \alpha_{K,i}(X) = (\alpha_{K,i}(M)\!-\!a_{ij})\ql(X)+a_{ij} \ \ \ \ \ \ \ 1 \le i \le p\!-\!1\]
for every $X \in \Theta$.\end{Thm}

\begin{proof} (1) Thanks to Lemma \ref{TG1}(3) and \cite[(3.6)]{Ka}, the module $\tau_\cG(M)_K \cong \tau_{\cG_K}(M_K)$ is a direct summand of $K\cG\!\otimes_{\fA_{p,K}}\![j]$.
Hence Lemma \ref{TG2} provides a character $\lambda_0 \in X(\cG)$ such that $ \tau_{\cG_K}(M_K) \cong M_K\!\otimes_K\!K_{\lambda_0}$. In view of \cite[(2.5)]{Ka}, we thus obtain
\[\tau_\cG(M) \cong M\!\otimes_k\!k_{\lambda_0}.\]
Hence there exists a divisor $\ell$ of ${\rm ord}(X(\cG))$ such that $\tau_\cG^\ell(M) \cong M$ and $\tau^{\ell-1}_\cG(M) \not \cong M$. As observed in the proof of (\ref{QRP3}), this
implies that $\Theta \cong \ZZ[A_\infty]/\langle \tau^\ell \rangle$.

(2) Suppose there exists $\gamma \in X(\cG)$ with $k_\gamma \in \Theta$. In view of the above, the functor
\[ t_{\lambda_0} : \modd \cG \lra \modd \cG \ \ ; \ \ X \mapsto X\!\otimes_k\!k_{\lambda_0}\]
is an auto-equivalence of $\modd \cG$ that commutes with $\tau_\cG$. Moreover, $t_{-\lambda_0}\circ \tau_\cG$ sends $\Theta$ to $\Theta$ and fixes $M$. As $t_{-\lambda_0}\circ
\tau_\cG$ preserves the quasi-length of a module, it follows that $t_{-\lambda_0}\circ \tau_\cG|_\Theta = \id_\Theta$, whence
\[ \tau_\cG = t_{\lambda_0}.\]
This implies $\tau_\cG(k_\gamma) \cong k_{\lambda_0+\gamma}$.

Now let $\omega \in X(\cG)$ be an arbitrary character. Then
\[ \tau_\cG(k_\omega) \cong \tau_\cG(k_\gamma\!\otimes_k\!k_{\omega-\gamma}) \cong \tau_\cG(k_\gamma)\!\otimes_k\!k_{\omega-\gamma} \cong
k_{\gamma+\lambda_0}\!\otimes_k\!k_{\omega-\gamma} \cong k_{\omega+\lambda_0}.\]
As a result, $\tau_\cG$ sends simple modules to simple modules and \cite[(IV.2.10)]{ARS} implies that $k\cG$ is a Nakayama algebra. In particular, $k\cG$ is representation-finite, a 
contradiction.

The arguments of (\ref{QRP3}(2)) now show that $M$ is quasi-simple.

(3) By Lemma \ref{TG2}, there exists a unique $j \in \{1,\ldots,p\!-\!1\}$ such that $M_K$ is a direct summand of $K\cG\!\otimes_{\fA_{p,K}}\![j]$. Since $M_K$ occurs in
$K\cG\!\otimes_{\fA_{p,K}}\![j]$ with multiplicity $1$, it follows that the multiplicities $n_i$ of $M_K$ in $N_i := K\cG\!\otimes_{\fA_{p,K}}\![i]$ are given by
\[ n_i = \delta_{i j}.\]
In view of (2), Proposition \ref{QRP1} implies
\[ \alpha_{K,i}(X) = (\alpha_{K,i}(M)\!-\!\sum_{\ell=1}^{p-1}a_{i\ell}n_\ell)\ql(X) +\sum_{\ell=1}^{p-1}a_{i\ell}n_\ell = (\alpha_{K,i}(M)\!-\! a_{ij})\ql(X)+a_{ij}\]
for every $X \in \Theta$ and $i \in \{1,\ldots,p\!-\!1\}$. \end{proof}

\bigskip

\begin{Remark} If $\cG$ is an infinitesimal supersolvable group of infinite representation type, then the principal block $\cB_0(\cG)$ is isomorphic to the algebra of measures of a trigonalizable group $\cG'$ (cf.\ \cite[(2.3),(2.4)]{FV1}). Hence the foregoing result also holds for indecomposable relatively projective modules that belong to $\cB_0(\cG)$. \end{Remark}

\bigskip

\begin{Cor} \label{TG4} Let $\cU$ be a unipotent finite group scheme of infinite representation type, $\alpha_K : \fA_{p,K} \lra K\cU$ be a $\pi$-point. Suppose that $M$ is an
indecomposable, relatively $\alpha_K$-projective $\cU$-module. Then the following statements hold:

{\rm (1)} \ The module $\Top_\cU(M)$ is simple and $M_K \cong K\cU\!\otimes_{\fA_{p,K}}\![j]$ for some $j \in \{1,\ldots, p\!-\!1\}$.

{\rm (2)} \ The component $\Theta_M \subseteq \Gamma_s(\cU)$ containing $M$ is isomorphic to $\ZZ[A_\infty]/\langle \tau\rangle$, and $M$ is quasi-simple.

{\rm (3)} \ There exists $j \in \{1,\ldots,p\!-\!1\}$ such that
\[ \alpha_{K,i}(X) = (\alpha_{K,i}(M)\!-\!a_{ij})\ql(X)+a_{ij} \ \ \ \ \ \ \ \ 1 \le i \le p\]
for every $X \in \Theta_M$. \end{Cor}

\begin{proof}  Since $X(\cU) = \{1\}$, our assertions are direct consequences of Lemma \ref{TG2}, Theorem \ref{TG3}, and the remark succeeding Proposition \ref{QRP1}, respectively.
\end{proof}

\bigskip

\begin{Remarks} (1) Corollary \ref{TG4} shows in particular that, among the functions $\alpha_{K,i} : \Theta \lra \NN_0$, at least $p\!-\!3$ are additive.

(2) According to \cite[(1.3)]{FRV}, a unipotent group $\cU$ whose algebra of measures has finite representation type either corresponds to the group $\ZZ/(p)$, or it is ``V-uniserial". A 
classification of the latter groups is given in \cite[(1.2)]{FRV}. The algebra $k\cU$ is a truncated polynomial ring $k[X]/(X^{p^n})$. \end{Remarks}

\bigskip
\noindent
In the following examples, we shall be concerned with enveloping algebras of restricted Lie algebras. By definition, a {\it restricted Lie algebra} is a pair $(\fg,[p])$ consisting of a Lie algebra
$\fg$ and a map $\fg \lra \fg \ ; \ x \mapsto x^{[p]}$ that enjoys the formal properties of an associative $p$-th power. We refer the reader to \cite[Chap.II]{SF} for further details. For our
present purposes, it suffices to know that, given a Lie algebra $\fg$ with basis $(x_i)_{i\in I}$, any map $x_i \mapsto x_i^{[p]}$ with $\ad x_i^{[p]} = (\ad x_i)^p$ uniquely extends to a 
$p$-map on $\fg$.

In view of \cite[(II,\S7,no.4)]{DG}, restricted Lie algebras correspond to infinitesimal group of height $\le 1$ in that
\[ k\cG \cong U_0(\fg),\]
where $U_0(\fg)$ is the {\it restricted enveloping algebra} of the restricted Lie algebra $\fg := \Lie(\cG)$. By definition,
\[ U_0(\fg) = U(\fg)/(\{x^p-x^{[p]}\ ; \ x \in \fg))\]
is a finite-dimensional Hopf algebra quotient of the ordinary universal enveloping algebra $U(\fg)$. In particular, unipotent infinitesimal group schemes of height $\le 1$ correspond to
unipotent restricted Lie algebras, that is, to Lie algebras with a nilpotent $p$-map.

\bigskip

\begin{Examples} (1) Let $\cU$ be a unipotent group scheme such that $\dim \Pi(\cU) \ge 1$. Suppose that $\alpha_K : \fA_{p,K} \lra K\cU$ is a $\pi$-point that factors through the
center $\fZ(K\cU)$ of $K\cU$. Let $\Theta \subseteq \Gamma_s(\cU)$ be a component that is not $\alpha_K$-split. Then $\Theta \cong \ZZ[A_\infty]/\langle \tau \rangle$, and the 
quasi-simple module $M \in \Theta$ satisfies $M_K \cong K\cU\!\otimes_{\fA_{p,K}}\![j]$ for some $j \in \{1,\ldots,p\!-\!1\}$, so that there exists $r \in \NN$ with $\alpha_{K,i}(M) = 
\delta_{ij}(\frac{\dim_kk\cU}{p}) = \delta_{ij}p^r$ for $1 \le i \le p$, cf.\ \cite[(6.7),(14.4)]{Wa}. Corollary \ref{STP2} thus yields
\[ \alpha_{K,i}(X) = \left\{ \begin{array}{cl} (p^r\!-\!2)\ql(X)+2 & \text{for} \ i=j\\ \ql(X)-1 & \text{for} \ i=j\!-\!1,j\!+\!1\\ 0 & \text{for} \ i \neq j,j\!+\!1,j\!-\!1 \end{array} 
\right.\]
for every $X \in \Theta$ and $1\le i\le p\!-\!1$.

(2) We consider the three-dimensional {\it Heisenberg algebra} $\fh = kx\oplus ky\oplus kz$, with product and $p$-map given by
\[ [x,y] = z \ ,  \ [x,z] = 0 = [y,z] \ \text{and} \ x^{[p]} = y^{[p]} = z^{[p]} = 0,\]
respectively. The $U_0(\fh)$-module $M_x := U_0(\fh)\!\otimes_{U_0(kx)}\!k$ is indecomposable and relatively projective with respect to the $\pi$-point $\alpha_k : \fA_{p,k} \lra 
U_0(\fh)$ that sends $t$ to $x$. The Cartan-Weyl identities (cf.\ \cite[(I.1.3)]{SF}) yield
\[ xy^n = y^nx +ny^{n-1}z \ \ \ \ \ \ \ \forall \ n\ge 0,\]
so that $x.(y^nz^m\otimes 1) = ny^{n-1}z^{m+1}\otimes 1$. Consequently, for every $\ell \in \{0,\ldots,2p\!-\!2\}$, the $k$-space $M_x(\ell) := \bigoplus_{n+m 
=\ell}k(y^nz^m\otimes 1)$ is a $U_0(kx)$-submodule of $M$. We thus obtain
\[ \alpha_k^\ast(M_x) = \bigoplus_{\ell=0}^{p-1} \alpha_k^\ast(M_x(\ell)) \oplus \bigoplus_{\ell=p}^{2p-2} \alpha_k^\ast(M_x(\ell)) \cong \bigoplus_{\ell=0}^{p-1}[\ell\!+\!1]
\oplus \bigoplus_{\ell=p}^{2p-2}[2p\!-\!\ell\!-\!1] \cong \bigoplus_{\ell=1}^{p-1}2[\ell] \oplus [p].\]
Now let $\Theta \subseteq \Gamma_s(\fh)$ be the component containing $M_x$. Given $X \in \Theta$, Corollary \ref{TG4} implies
\[ \alpha_{k,i}(X) = \left\{ \begin{array}{cl} 2 & \text{for} \ i=1\\ (3\!-\!\delta_{p,2})\ql(X)-1 & \text{for} \ i=2 \\ (2\!-\!\delta_{i,p})\ql(X) & \text{for} \ 3 \le i \le p. \end{array} 
\right.\]
Since $z \in \fZ(U_0(\fh))$, the choice $\alpha_k(t) :=z$ and $M_z := U_0(\fh)\!\otimes_{U_0(kz)}\!k$ leads to the formulae of (1) with $r=2$.

(3) According to Lemma \ref{AFP2}, the $\pi$-points $\alpha_K$ for which a given $\cG$-module is relatively $\alpha_K$-projective all belong to one equivalence class. The following example
shows that the converse does not hold. Let $\fu := kx\oplus ky$ be the two-dimensional abelian restricted Lie algebra with trivial $p$-mapping. Then $M_x :=
U_0(\fu)\!\otimes_{U_0(kx)}\!k$ is an indecomposable $U_0(\fu)$-module of dimension $p$. Let $\alpha_k : \fA_{p,k} \lra U_0(\fu)$ be the $\pi$-point, given by $\alpha_k(t): = x$. Then
$M$ is relatively $\alpha_k$-projective and $\alpha_k^\ast(M_x) = p[1]$. According to \cite[(2.2)]{FPe1}, the map $\beta_k : \fA_{p,k} \lra U_0(\fu)$ defined via $\beta_k(t) :=
x+y^{p-1}$ is a $\pi$-point such that $\beta_k \sim \alpha_k$. Since
\[ \beta_k(t)^i\dact(y^j\otimes 1) = \left\{ \begin{array}{cl} y^j\otimes 1 & \text{for} \ i=0 \\ y^{p-1}\otimes 1 & \text{for} \ (i,j) = (1,0) \\ 0 & \text{else},\end{array} \right.\]
we see that $\beta_k(M_x) = (p\!-\!2)[1]\oplus [2]$. In view of (1), the module $M_x$ is not relatively $\beta_k$-projective, whenever $p\ge 3$. \end{Examples}

\bigskip

\begin{Examples} For $p \ge 3$ we consider the restricted Lie algebra $\fg := kt\oplus kx \oplus ky$, whose bracket and $p$-map are given by
\[ [t,x] = x \ , \ [t,y]=2y \ , \ [x,y] = 0 \ \ \text{and} \ \ t^{[p]} = t \ , \ x^{[p]} = 0 = y^{[p]},\]
respectively. Every simple $U_0(\fg)$-module is annihilated by the $p$-ideal $kx\oplus ky$ and thus one-dimensional. Hence $\fg$ is trigonalizable.

(1) Suppose that $\alpha_k : \fA_{p,k} \lra U_0(\fg)$ is given by $\alpha_k(t) := ky$. We consider the $U_0(\fg)$-module
\[ V := U_0(\fg)\!\otimes_{U_0(ky)}\!k.\]
Direct computation shows that $y$ annihilates $V$, so that $V \cong U_0(\fg/ky)$ is a $U_0(\fg/ky)$-module. As such it has $p$-indecomposable constituents, each having a simple top.

(2) Suppose that $\alpha_k(t) = x+y=:v$ and consider
\[ V := U_0(\fg)\!\otimes_{U_0(kv)}\!k.\]
Consider the $\pi$-point $\beta_k$, defined by $\beta_k(t) :=x$. Since $\beta_k^\ast(V)$ is projective, the dimension of every constituent $M\mid V$ is a multiple of $p$. Lemma 
\ref{TG2} now shows that $\dim_kM = p$ for every proper summand $M$. Using techniques from \cite{Fa4} one can then show that $V$ is indecomposable. \end{Examples}

\bigskip

\begin{Cor} \label{TG5} Let $\cG$ be a trigonalizable finite group scheme of infinite representation type, $M$ be a non-projective indecomposable $\cG$-module. If $\alpha_K : \fA_{p,K} 
\lra K\cG$ is a $\pi$-point such that $\alpha_{K,j}(M) \le 1$ for $1\le j \le p\!-\!1$, then $M$ is not relatively $\alpha_K$-projective.\end{Cor}

\begin{proof} Suppose that $M$ is relatively $\alpha_K$-projective and let $\Theta$ be the stable AR-component containing $M$. Theorem \ref{TG3} provides $j \in \{1,\ldots,p\!-\!1\}$ 
such that
\[ \alpha_{K,j}(X) = (\alpha_{K,j}(M)\!-\!2)\ql(X)+2 \le 2-\ql(X)\]
for all $X \in \Theta$. Since $\alpha_{K,j}(X) \ge 0$, this cannot happen. \end{proof}

\bigskip

\begin{Remark}  Let $\cG = \SL(2)_1$ be the first Frobenius kernel of $\SL(2)$. In Section \ref{S:Ex} below, we shall see that, for every non-projective baby Verma module $Z(\lambda)$, 
the $\pi$-points $\alpha_K$ giving rise to elements of $\Pi(\SL(2)_1)_{Z(\lambda)}$ all satisfy $\alpha_{K,j}(M) \le 1$ for $1\le j \le p\!-\!1$. On the other hand, one can show that 
$Z(\lambda)$ is relatively $\alpha_K$-projective for each of these $\pi$-points.\end{Remark}

\subsection{Modules with cyclic vertices}
For finite groups, the Mackey decomposition theorem along with the theory of vertices of indecomposable modules provides better control over relatively projective modules. Let $G$ be a finite 
group, $M$ be an indecomposable $G$-module. Recall that a $p$-subgroup $D \subseteq G$ is called a {\it vertex} of $M$ if $M\!\mid\!(kG\!\otimes_{kD}\!V)$ for some $D$-module $V$ 
and $D$ is minimal subject to this property. All vertices of $M$ are conjugate, and we write $D={\rm vx}(M)$ (cf.\ \cite[(3.10)]{Be1}). 

\bigskip

\begin{Thm} \label{MCV1} Suppose that $p\ge 3$. Let $G$ be a finite group, $\alpha_K \in \Pt(G)$ be a $\pi$-point that factors through a cyclic $p$-subgroup $C \subseteq G$. If $M$ is a
non-projective  indecomposable relatively $\alpha_K$-projective $G$-module belonging to a block of infinite representation type, then

{\rm (1)} \ $M$ is quasi-simple with component $\Theta \cong \ZZ[A_\infty]/\langle \tau^q \rangle$.

{\rm (2)} \ If $\beta_L \in \Pt(G)$ is a $\pi$-point factoring through $C$, then $M$ is relatively $\beta_L$-projective.

{\rm (3)} \ There exist $j \in \{1,\ldots,p\!-\!1\}$ and $m,n \in \NN$ with $m \ge 2n$ such that
\[ \beta_{L,i}(X) = \left\{ \begin{array}{cl} (m\!-\!2n)\ql(X)+2n & \text{for} \ i=j\\ n(\ql(X)\!-\!1) & \text{for} \ i=j\!-\!1,j\!+\!1\\ 0 & \text{for} \ i \neq j,j\!+\!1,j\!-\!1 \end{array} 
\right.\]
for every $X \in \Theta$ and every $\beta_L \in \Pt(G)$ that factors through $C$. 

{\rm (4)} \ ${\rm vx}(M) \cong \ZZ/(p)$. \end{Thm}

\begin{proof} (1) By assumption, there exists a $KC$-module $V$ such that the indecomposable $KG$-module $M_K$ is a direct summand of
\[ KG\!\otimes_{\fA_{p,K}}\!\alpha^\ast_K(M_K) \cong KG\!\otimes_{KC}\!V.\]
As a result, the vertex ${\rm vx}(M_K)$ of $M_K$ is contained in $C$ and hence is cyclic. If the block of $KG$ containing the indecomposable $KG$-module $M_K$ has finite representation 
type, then \cite[(2.5)]{Ka} implies that the block containing $M$ enjoys the same property, a contradiction. Thanks to \cite[Theorem]{Er0}, the module $M_K$ is quasi-simple. According to 
\cite[(3.8)]{Ka}, $\fE_M\!\otimes_k\!K$ is the almost split sequence terminating in $M_K$. Hence $M$ also has exactly one predecessor in $\Theta$. As a result, the $G$-module $M$ is also 
quasi-simple.

(2) Let $X \in \Theta$ be arbitrary. Since $C$ is cyclic, we have $\dim \Pi(C) = 0$, so that every $\pi$-point of $C$ is generic. According to \cite[(4.2)]{FPS}, we have
\[ \Jt(X|_C,\alpha_K) = \Jt(X|_C,\beta_L),\]
whence $\alpha_{K,i}(X) = \beta_{L,i}(X)$ for $1\le i\le p\!-\!1$. 

By (1), every relatively $\alpha_K$-projective module belonging to $\Theta$ is quasi-simple. Hence Proposition \ref{QRP1} provides $(n_1,\ldots,n_{p-1}) \in \NN_0^{p-1}\setminus \{0\}$ 
such that
\[ \alpha_{K,i}(X) = (\alpha_{K,i}(M)\!-\!\sum_{j=1}^{p-1}a_{ij}n_j)\ql(X)+\sum_{j=1}^{p-1}a_{ij}n_j \ \ \ \ \ \ \ \ 1 \le i \le p\!-\!1\]
for every $X \in \Theta$. Hence this formula also holds for the functions $\beta_{L,i} : \Theta\lra \NN_0$, so that $\ell(\beta_{L,i}) \le 2$ with equality holding for at least one $i \in 
\{1,\ldots,p\!-\!1\}$. In view of (1) and Corollary \ref{AFP5} the module $M$ is relatively $\beta_L$-projective. 

(3) The subgroup $C \subseteq G$ contains an element $g$ of order $p$, and we define $\alpha(g)_k : \fA_{p,k} \lra kG$ via $\alpha(g)_k(t) := g\!-\!1$. By (2), there exists $j \in 
\{1,\ldots,p\!-\!1\}$ such that the $G$-module $M$ is a direct summand of $kG\!\otimes_{kC_p}\![j]$, where we set $C_p := \langle g \rangle$. The Mackey decomposition theorem
yields
\[ \alpha(g)^\ast_k(kG\!\otimes_{kC_p}\![j]) \cong (kG\!\otimes_{kC_p}\![j])|_{kC_p} \cong \ell[j]\oplus s[p],\]
so that $\alpha(g)^\ast_k(M) \cong m[j]\oplus t[p]$ for some $m \ge 1$. This implies in particular, that the number $j$ defined above is uniquely determined, whence $n_i(M,\alpha(g)_k) 
= \delta_{i,j}n$ for some $n \in \NN$. In view of Proposition \ref{QRP1}, the asserted formula holds for the $\pi$-point $\alpha(g)$. As noted earlier, it thus holds for any $\pi$-point 
$\beta_L$ factoring through $C$.

Finally, since $\alpha(g)^\ast_{k,j}(X) \ge 0$ for all $X \in \Theta$, we have $m\ge 2n$. 
 
(4) Since $M$ is a direct summand of $kG\!\otimes_{kC_p}\![j]$, the vertex ${\rm vx}(M)$ is either trivial, or isomorphic to $\ZZ/(p)$. In the former case, $M$ is projective, a contradiction.
\end{proof}
 
\bigskip
\noindent
We can endow the truncated polynomial ring $\fA_{p,k}$ with the structure of a Hopf algebra such that $\fA_{p,k} \cong k\ZZ/(p)$. Then we have:

\bigskip

\begin{Cor} \label{MCV2} Let $\alpha_K \in \Pt(G)$ be a $\pi$-point, which is a homomorphism $K\ZZ/(p) \lra KG$ of Hopf algebras. If $M$ is a non-projective indecomposable relatively 
$\alpha_K$-projective module with infinite AR-component $\Theta$, then $\Theta \cong \ZZ[A_\infty]/\langle\tau^q\rangle$ and there exist $j \in \{1,\ldots,p\!-\!1\}$ and $m,n \in \NN$ 
with $m\ge 2n$ such that
\[ \alpha_{K,i}(X) = \left\{ \begin{array}{cl} (m\!-\!2n)\ql(X)+2n & \text{for} \ i=j\\ n(\ql(X)\!-\!1) & \text{for} \ i=j\!-\!1,j\!+\!1\\ 0 & \text{for} \ i \neq j,j\!+\!1,j\!-\!1 \end{array}
 \right.\]
for every $X \in \Theta$  \hfill $\square$ \end{Cor}

\bigskip
\noindent
\begin{Remark} Proposition \ref{Ex1} below shows that the formulae of Corollary \ref{MCV2} do not necessarily hold for infinitesimal group schemes $\cG$ and $\pi$-points defined by Hopf
 algebra homomorphisms $k\GG_{a(1)} \lra k\cG$.  \end{Remark}

\bigskip

\section{Constantly Supported Modules}\label{S:CSM}
Retaining the general assumptions of Section \ref{S:AFP}, we consider those $\cG$-modules that have constant Jordan type on their $\Pi$-supports. In case a $\cG$-module $M$ has full
$\Pi$-support $\Pi(\cG)_M = \Pi(\cG)$, this amounts to $M$ being of constant Jordan type. We are thus led to the following:

\bigskip

\begin{Definition} A $\cG$-module $M$ is referred to as {\it constantly supported} if

(a) \ $\Pi(\cG)_M \ne \Pi(\cG)$, and

(b) \ $|\Jt(M)| = 2$. \end{Definition}

\bigskip
\noindent
By virtue of Corollary \ref{SJ1}, locally split stable Auslander-Reiten components either contain no constantly supported modules or they consist entirely of such modules.

Our fundamental examples are given by direct summands of the Carlson modules $L_\zeta$. Let $(P_n,d_n)_{n\ge 0}$ be a minimal projective resolution of the trivial $\cG$-module $k$.
Then
\[ \Hom_\cG(\Omega^n_\cG(k),k) \lra \HH^n(\cG,k) \ \ ; \ \ \hat{\zeta} \mapsto [\hat{\zeta}\circ d_n]\]
is an isomorphism. If $\zeta = [\hat{\zeta}\circ d_n] \ne 0$, then the {\it Carlson module}
\[ L_\zeta := \ker \hat{\zeta} \subseteq \Omega^n_\cG(k)\]
does not depend on the choice of the representing cocycle. In view of \cite[(5.9.4)]{Be2}, we postulate that the zero element corresponds to the modules $L_0^n := \Omega^n_\cG(k)\oplus
\Omega_\cG(k)$. Being submodules of $\Omega^{n}_\cG(k)\oplus \Omega_\cG(k)$, $L_\zeta$ and $L_0^n$ are projective-free, that is, they do not have non-zero projective summands.

Given $\zeta \in \HH^n(\cG,k)\setminus\{0\}$, we note that $L_\zeta = (0)$ if and only if $\Omega_\cG^n(k) \cong k$. Since $\Omega_\cG^n(k)$ has constant Jordan type $[p\!-\!1]$ 
for $n$ odd, we see that, for $p\ge 3$, this degenerate case can only occur if $n$ is even. In view of
\[ \dim \Pi(\cG) = \dim \cV_\cG(k)-1 = \cx_\cG(k)-1,\]
it follows that $L_\zeta \ne (0)$ whenever $\dim \Pi(\cG) \ge 1$.

\bigskip

\begin{Remark} If $\cG$ is an infinitesimal group scheme and $L_\zeta = (0)$ for some $\zeta \in \HH^n(\cG,k)\setminus\{0\}$, then \cite[(2.1)]{FV1} implies that $\cG$ is
supersolvable. In view of \cite[(2.4)]{FV1}, all simple modules of the principal block $\cB_0(\cG) \subseteq k\cG$ are one-dimensional. Owing to \cite[(2.7)]{FV1}, this block is a Nakayama
algebra, so that \cite[(IV.2.10)]{ARS} implies that $\dim_k\Omega^{2m}_\cG(k) = 1$ for all $m \ge 0$. Thus, $\cG$ affords no non-zero Carlson modules of even degree. \end{Remark}

\bigskip
\noindent
The cohomology ring
\[ \HH^\ast(\cG,k) := \bigoplus_{m\ge 0} \HH^m(\cG,k)\]
is graded commutative, so that $\HH^\bullet(\cG,k) := \bigoplus_{m\ge 0}\HH^{2m}(\cG,k)$ is a commutative subalgebra. The importance of Carlson modules resides in their support 
varieties often being hyperplanes of $\cV_\cG(k)$: The variety
\[\cV_\cG(L_\zeta) = Z(\zeta),\]
is the zero locus of the homogeneous element $\zeta \in \HH^\bullet(\cG,k)$ (see \cite[(5.9.1)]{Be2},\cite[(4.11)]{FPe1}). In particular, $L_\zeta \ne (0)$ for every non-zero nilpotent 
element $\zeta$ of even degree. 

Recall that a $\cG$-module $M$ is referred to as {\it endo-trivial} if
\[ \End_k(M) \cong k \oplus ({\rm proj}.).\]
The following subsidiary result elaborates on \cite[(4.1)]{Ca1}.

\bigskip

\begin{Lemma} \label{CSM1} Let $\zeta \in \HH^{2n}(\cG,k)\setminus \{0\}$ and write $L_\zeta \cong M_1\oplus \cdots \oplus M_r$, with $M_i$ indecomposable.

{\rm (1)} \ If $\zeta$ is not nilpotent and $M\!\mid\!L_\zeta$ is a non-zero summand, then $M$ is constantly supported and
\[ \Jt(M) = \{ (\frac{\dim_kM}{p})[p], [1]\oplus [p\!-\!1] \oplus n_M[p]\}.\]

{\rm (2)} \ If $\zeta$ is not nilpotent, then
\[ \Pi(\cG)_{L_\zeta} = \bigsqcup_{i=1}^r \Pi(\cG)_{M_i}\]
is the decomposition of $\Pi(\cG)_{L_\zeta}$ into its connected components.

{\rm (3)} \ If $\zeta$ is nilpotent, then $L_\zeta$ has constant Jordan type $\Jt(L_\zeta) = \{[1]\oplus [p\!-\!1] \oplus n_\zeta[p]\}$. Moreover, $L_\zeta$ is either indecomposable or a direct sum of two endo-trivial modules. \end{Lemma}

\begin{proof} We begin by proving a number of auxiliary statements.

\medskip
(a) {\it If $[\alpha_K] \in \Pi(\cG)_{L_\zeta}$, then $\Jt(L_\zeta,\alpha_K) = [1] \oplus [p\!-\!1] \oplus n_\zeta[p]$ for some $n_\zeta \in \NN_0$}.

\smallskip
\noindent
We consider the short exact sequence
\[ (0) \lra L_\zeta \lra \Omega_\cG^{2n}(k) \stackrel{\hat{\zeta}}{\lra} k \lra (0).\]
A $\pi$-point $\alpha_K \in \Pt(\cG)$ induces an exact sequence
\[ (0) \lra \alpha^\ast_K((L_\zeta)_K) \lra K \oplus ({\rm proj.}) \stackrel{(f,g)}\lra K \lra (0).\]
Since $[\alpha_K] \in \Pi(\cG)_{L_\zeta}$, the module $\alpha_K^\ast((L_\zeta)_K)$ is not projective and we have $f = 0$. Hence
\[ \alpha^\ast_K((L_\zeta)_K) \cong K \oplus \Omega_{\fA_{p,K}}(K) \oplus ({\rm proj.}),\]
so that
\[ \Jt(L_\zeta,\alpha_K) = [1] \oplus [p\!-\!1] \oplus n_\zeta[p]\]
for some $n_\zeta \in \NN_0$. \hfill $\Diamond$

\medskip
(b) {\it Let $M\!\mid\!L_\zeta$ be a non-zero summand such that $\dim_k M \equiv 0 \ \modd(p)$. Then}
\[\Jt(M,\alpha_K) = [1]\oplus [p\!-\!1]\oplus n_M[p] \  \text{for all} \ [\alpha_K] \in \Pi(\cG)_M.\]

\smallskip
\noindent
If $\alpha_K \in \Pt(\cG)$ is a $\pi$-point with $[\alpha_K] \in \Pi(\cG)_M \subseteq \Pi(\cG)_{L_\zeta}$ (cf.\ \cite[(3.3)]{FPe2}), then (a) implies
\[ \alpha_K^\ast(M_K)\!\mid \! ([1]\oplus [p\!-\!1]\oplus n_\zeta[p]).\]
Accordingly, the projective-free part $X$ of $\alpha_K^\ast(M_K)$ is a direct summand of $[1] \oplus [p\!-\!1]$. Since $[\alpha_K] \in \Pi(\cG)_M$, the $\fA_{p,K}$-module $X$ is 
non-zero. Our assumption on $M$ yields $\dim_KX \equiv 0 \ \modd (p)$, so that $X \cong [1] \oplus [p\!-\!1]$. Consequently, $M$ has the asserted Jordan type with respect to $\alpha_K$.
\hfill $\Diamond$

\medskip
(c) {\it Let $i \ne j \in \{1,\ldots,r\}$ such that $\dim_kM_i, \dim_kM_j \equiv 0 \ \modd(p)$. Then $\Pi(\cG)_{M_i}\cap \Pi(\cG)_{M_j} = \emptyset$.}

\smallskip
\noindent
If $[\alpha_K] \in \Pi(\cG)_{M_i}\cap \Pi(\cG)_{M_j}$, then (b) yields
\[ \alpha_K^\ast((M_\ell)_K)  \cong [1] \oplus [p\!-\!1] \oplus m_\ell[p]\]
for $\ell \in \{i,j\}$, which contradicts $\alpha^\ast_K((M_i)_K)\oplus \alpha^\ast_K((M_j)_K)$ being a direct summand of $\alpha^\ast_K((L_\zeta)_K)$. \hfill $\Diamond$

\medskip

(1)  Since the element $\zeta \in \HH^\bullet(\cG,k)$ is not nilpotent, \cite[(3.3)]{FPe2} implies $\Pi(\cG)_M \subseteq \Pi(\cG)_{L_\zeta} \ne \Pi(\cG)$, so that $\dim_kM \equiv 0 \ 
\modd (p)$. The result now follows from (b).

(2)  Arguing as in (1), we see that $\dim_k M_i \equiv 0 \ \modd(p)$ for every $i \in \{1,\ldots,r\}$. As a result, \cite[(3.3)]{FPe2} in conjunction with (c) implies
\[  \Pi(\cG)_{L_\zeta} = \bigsqcup_{i=1}^r \Pi(\cG)_{M_i},\]
while \cite[(3.4)]{CFP} (see also \cite{Ca1}) ensures the connectedness of each subspace $\Pi(\cG)_{M_i}$.

(3) Since $\zeta \ne 0$ is nilpotent, we have $\cV_\cG(L_\zeta) = \cV_\cG(k) \ne \{0\}$. Thanks to \cite[(3.4),(3.6)]{FPe2}, this implies $\Pi(\cG)_{L_\zeta} = \Pi(\cG) \ne \emptyset$. 
By (a), the module $L_\zeta$ has constant Jordan type $\Jt(L_\zeta) = \{[1]\oplus [p\!-\!1] \oplus n_\zeta[p]\}$. Let $\alpha_K$ be a $\pi$-point. Each of the modules 
$\alpha_K^\ast((M_i)_K)$ is a direct summand of $[1]\oplus [p\!-\!1]\oplus n_\zeta[p]$, whence $\dim_kM_i \equiv 1,-1,0 \ \modd(p)$.

Suppose there exists $i \in \{1,\ldots, r\}$ such that $\dim_k M_i \equiv \ell \ \modd(p)$ for $\ell \in \{1,p\!-\!1\}$. Then $\Pi(\cG)_{M_i} = \Pi(\cG)$, and there exists $n_i \in \NN$ 
such that $\alpha^\ast_K((M_i)_K) = [\ell] \oplus n_i[p]$ for every $\pi$-point $\alpha_K \in \Pt(\cG)$. There also exists $j \in \{1,\ldots, r\}$ such that $\alpha^\ast_K((M_j)_K =
[p\!-\!\ell] \oplus n_j[p] \ \ \forall \ \alpha_K \in \Pt(\cG)$. Consequently, all other summands of $L_\zeta$ are projective, so that $L_\zeta$ being projective-free yields $r=2$. We may now
apply \cite[(5.6)]{CFP} to see that each summand is endo-trivial.

Alternatively, $\dim_kM_i \equiv 0 \ \modd(p)$ for every $i \in \{1,\ldots,r\}$. According to (c), this implies
\[ \Pi(\cG) = \Pi(\cG)_{L_\zeta} = \bigsqcup_{i=1}^r \Pi(\cG)_{M_i},\]
which, in view of $\Pi(\cG)$ being connected (cf.\ \cite[(3.4)]{CFP}), yields $r=1$. \end{proof}

\bigskip

\begin{Remarks} (1) The above result shows in particular that, for a non-nilpotent element $\zeta \in \HH^{2n}(\cG,k)$, the indecomposable summands of $L_\zeta$ occur with 
multiplicity one. By the same token, the $\cG$-module $M_i\!\otimes_k\!M_j$ is projective, whenever $i \ne j$ (cf.\ \cite[(3.2)]{FPe2}). 

(2) We shall see in Theorem \ref{CNED1} that the second alternative in Lemma \ref{CSM1}(3) does not occur. \end{Remarks}

\bigskip

\begin{Theorem} \label{CSM2} Suppose that $\zeta \in \HH^{2n}(\cG,k)$ is a non-nilpotent element. Let $M$ be an indecomposable summand of $L_\zeta$, $\Theta \subseteq 
\Gamma_s(\cG)$ be the component containing $M$.

{\rm (1)} \ If $\Theta$ is locally split, then every $N \in \Theta$ is constantly supported and
\[ \Jt(N) = \{ (\frac{\dim_kN}{p})[p], f_\Theta(N)[1]\oplus f_\Theta(N)[p\!-\!1] \oplus m_N[p]\}.\]

{\rm (2)} \ If $\bar{T}_\Theta \cong A_\infty$ and $\Theta$ is locally split, then $M$ is quasi-simple. \end{Theorem}

\begin{proof} (1) Let $\alpha_K$ be a $\pi$-point with $[\alpha_K] \in \Pi(\cG)_\Theta$. Lemma \ref{CSM1} implies $d_i^\Theta(\alpha_K) = \delta_{i,1}+\delta_{i,p-1}$ for
$1\le i \le p\!-\!1$, so that the assertion follows from Theorem \ref{FI1}.

(2) This follows directly from (1), Lemma \ref{CSM1} and Corollary \ref{FI2}. \end{proof}

\bigskip

\begin{Corollary} \label{CSM3} Suppose that $\dim \Pi(\cG) \ge 2$. For every non-nilpotent homogeneous element $\zeta \in \HH^\bullet(\cG,k)$ there exists a component $\Theta_\zeta 
\subseteq \Gamma_s(\cG)$ such that
\[ \Jt(M) = \{ (\frac{\dim_kM}{p})[p],\, f_{\Theta_\zeta}(M)[1]\oplus f_{\Theta_\zeta}(M)[p\!-\!1]\oplus n_M[p]\}\]
for every $M \in \Theta_\zeta$. \end{Corollary}

\begin{proof} Let $\zeta \in \HH^{2n}(\cG,k)$ be a non-nilpotent element. According to \cite[(3.4),(3.6)]{FPe2}, we have
\[ \dim \Pi(\cG)_{L_\zeta} = \dim \cV_\cG(L_\zeta)-1 \ge \dim \cV_\cG(k)-2 = \dim \Pi(\cG)-1 \ge 1.\]
Lemma \ref{CSM1} provides an indecomposable summand $M_\zeta|L_\zeta$ such that $\dim \Pi(\cG)_{M_\zeta} \ge 1$
and
\[ \Jt(M_\zeta) = \{ (\frac{\dim_kM_\zeta}{p})[p],\, [1]\oplus [p\!-\!1]\oplus n_{M_\zeta}[p]\}.\]
Thus, letting $\Theta_\zeta \subseteq \Gamma_s(\cG)$ be the stable AR-component containing $M_\zeta$, our assertion is a consequence of Theorem \ref{CSM2}. \end{proof}

\bigskip
\noindent
We conclude this section with an application concerning trigonalizable group schemes.

\bigskip

\begin{Corollary} \label{CSM4} Suppose that $p\ge 3$. Let $\cG$ be a trigonalizable finite group scheme of infinite representation type, $\zeta \in \HH^{2n}(\cG,k)$ be a non-nilpotent
element. Then every component $\Theta \subseteq \Gamma_s(\cG)$ containing an indecomposable direct summand of $L_\zeta$ is locally split.\end{Corollary}

\begin{proof} Let $X\!\mid\!L_\zeta$ be an indecomposable direct summand such that $X \in \Theta$. If $\alpha_K : \fA_{p,K} \lra K\cG$ is a $\pi$-point such that $\Theta$ is not 
$\alpha_K$-split, then Theorem \ref{TG3} implies that $\Theta \cong \ZZ[A_\infty]/\langle\tau^m \rangle$ for some $m \ge 1$. By the same token, every relatively $\alpha_K$-projective 
vertex of $\Theta$ is quasi-simple. Given such a module $M$, there is $j \in \{1,\ldots,p\!-\!1\}$ with
\[ \alpha_{K,i}(X) = (\alpha_{K,i}(M)\!-\!a_{ij})\ql(X) + a_{ij} \ \ \ \ \ 1 \le i \le p\!-\!1.\]
In view of Lemma \ref{CSM1} and Corollary \ref{TG5}, the module $X$ is not relatively $\alpha_K$-projective, so that $\ql(X)\ge 2$. Specializing $i=j$, we obtain
\[ \delta_{j,1}+\delta_{j,p-1} -2 = \ql(X)(\alpha_{K,j}(M)\!-\!2).\]
Since $p\ge 3$, the left-hand side belongs to $\{-2,-1\}$, whence $\ql(X)=2$, as well as $j \not\in \{1,p\!-\!1\}$. Consequently,
\[ 1 = (\alpha_{K,i}(M)\!-\!a_{ij})2 + a_{ij} \ \ \ \text{for} \ i \in \{1,p\!-\!1\},\]
so that $a_{1,j} \ne 0 \ne a_{p\!-\!1,j}$. Observing $j \ne \{1,p\!-\!1\}$, we obtain $2=j=p\!-\!2$, a contradiction. \end{proof}

\bigskip

\begin{Remark} Let $\cG = \SL(2)_1$ be the first Frobenius kernel of $\SL(2)$. If $\zeta \in \HH^{2n}(\SL(2)_1,k)$ is a non-zero element, then $\zeta$ is not nilpotent and each
indecomposable summand of $L_\zeta$ belongs to a homogeneous tube. We shall see in Section \ref{S:Ex} that none of these components is locally split and that each quasi-simple
$\SL(2)_1$-module has dimension $p$. On the other hand, the examples succeeding Theorem \ref{CNN5} show that certain Carlson modules have indecomposable summands of dimension
$2p$. Thus, Theorem \ref{CSM2} may fail for components of tree class $A_\infty$ that are not locally split. \end{Remark}

\bigskip

\section{Indecomposable Carlson Modules}\label{S:CM}
In view of the foregoing results, the question when a Carlson module $L_\zeta$ is indecomposable arises. In this section, we provide indecomposability criteria for Carlson modules. Depending
on properties of $\zeta$, these are established by means of support varieties or Jordan types.

\subsection{Carlson modules of non-nilpotent elements}
Here we study the case, where $\zeta \in \HH^\bullet(\cG,k)\setminus\{0\}$ is a non-nilpotent homogeneous element of even degree.

\bigskip

\begin{Thm} \label{CNN1} Suppose that the variety $\cV_\cG(k)\subseteq \A^m$ is equidimensional of dimension $n \ge \frac{m+3}{2}$. If $\zeta \in \HH^\bullet(\cG,k)$ is a 
non-nilpotent homogeneous element of positive degree, then $L_\zeta$ is indecomposable.\end{Thm}

\begin{proof} In view of Lemma \ref{CSM1}, it suffices to show that ${\rm Proj}(Z(\zeta)) = {\rm Proj}(\cV_\cG(L_\zeta)) \cong \Pi(\cG)_{L_\zeta}$ is connected.

Let $X,Y \subseteq Z(\zeta)$ be closed, conical subsets such that
\[ Z(\zeta) = X \cup Y \ \ \text{and} \ \ X\cap Y = \{0\}.\]
We denote by ${\rm Irr}(\zeta)$ the set of irreducible components of $Z(\zeta)$. By general theory (cf.\ \cite[(I.\S8)]{Mu}), every $Z \in {\rm Irr}(\zeta)$ has dimension $\ge n-1$, and the
equality
\[ Z = (Z\cap X)\cup (Z\cap Y)\]
implies that ${\rm Irr}(\zeta) = I(X) \cup I(Y)$, where the subsets $I(X),I(Y) \subseteq {\rm Irr}(\zeta)$ contain those irreducible components lying inside $X$ and $Y$, respectively. The 
assumption $I(X),I(Y) \ne \emptyset$ provides $Z,Z' \in {\rm Irr}(\zeta)$ with
\[ Z \cap Z' = \{0\}.\]
The affine dimension theorem \cite[(I.7.1)]{Har} then implies
\[ 0 = \dim Z\cap Z' \ge 2n-2-m \ge m+3-2-m =1,\]
a contradiction. Hence we may assume that $I(Y) = \emptyset$, whence ${\rm Irr}(\zeta) = I(X)$ and $X = Z(\zeta)$. Consequently, the variety ${\rm Proj}(Z(\zeta))$ is connected.
\end{proof}

\bigskip

\begin{Remark} Let $n_\cG$ be the minimum of the dimensions of the irreducible components of $\cV_\cG(k)$. Then the conclusion of Theorem \ref{CNN1} also holds if $n_\cG \ge 
\frac{m+3}{2}$. For a finite group $G$, the invariant $n_G$ coincides with the saturation rank $\srk(G)$ of $G$, which we consider in Section 6.4. \end{Remark} 

\bigskip

\begin{Cor} \label{CNN2} Let $\cU$ be an abelian unipotent group scheme with $\dim \cV_\cU(k) \ge 3$. Then the Carlson module $L_\zeta$ associated to a homogeneous non-nilpotent 
element $\zeta \in \HH^\bullet(\cU,k)$ is indecomposable.\end{Cor}

\begin{proof} Let $r := \dim \cV_\cU(k)$ be the dimension of the support variety of $\cU$. In view of \cite[(14.4)]{Wa}, the ring $\HH^\ast(\cU,k) \cong \Ext_{k\cU}^\ast(k,k)$ is
isomorphic to the cohomology ring of an abelian $p$-group. Thanks to \cite[(3.5.5),(3.5.6)]{Be1}, the support variety $\cV_\cU(k) \cong \A^r$ is irreducible. Since $r \ge \frac{r+3}{2}$ for
$r \ge 3$, it follows from Theorem \ref{CNN1} that the module $L_\zeta$ is indecomposable. \end{proof}

\bigskip

\begin{Remark} The foregoing results fail for group schemes whose support varieties have dimension $\le 2$, see our discussion below concerning restricted Lie algebras. \end{Remark}

\bigskip
\noindent
We now turn to infinitesimal groups and their rank varieties. Let $\cG$ be an infinitesimal $k$-group. Given $r \in \NN$, the authors introduce in \cite{SFB1} the scheme of infinitesimal
$1$-parameter subgroups of $\cG$ of height $\le r$, whose variety of $k$-rational points is
\[ V_r(\cG) := \Hom(\GG_{a(r)},\cG).\]
Here $\GG_{a(r)} := {\rm Spec}_k(k[X]/(X^{p^r}))$ is the r-th Frobenius kernel of the additive group $\GG_a$. We set $x := X + (X^{p^r})$, so that $\{x^i \ ; \ 0 \le i \le p^r-1\}$ is
a basis of $k[X]/(X^{p^r})$. If $\{\delta_0,\ldots,\delta_{p^r-1}\} \subseteq k\GG_{a(r)}$ denotes the dual basis, then, setting $u_i := \delta_{p^i}$, we obtain a canonical set of
generators $\{u_0,\ldots,u_{r-1}\}$ of the algebra $k\GG_{a(r)}$. We consider $A_r := k[u_{r-1}] \subseteq k\GG_{a(r)}$,  the $p$-dimensional subalgebra generated by $u_{r-1}$. Given a 
$\cG$-module $M$, the authors define in \cite[\S6]{SFB2} the {\it r-th rank variety of $M$} via
\[ V_r(\cG)_M := \{ \alpha \in V_r(\cG) \ ; \ \alpha^\ast(M)|_{A_r} \ \text{is not projective}\}.\]
Suppose that $\cG$ has height $\le r$, that is, $\cG$ coincides with its r-th Frobenius kernel $\cG_r$. Thanks to \cite[(5.2),(6.8)]{SFB2}, there exists a morphism $\Psi : V_r(\cG) \lra
\cV_\cG(k)$ which is a homeomorphism such that
\[ \Psi(V_r(\cG)_M) = \cV_\cG(M)\]
for every $M \in \modd \cG$. We thus obtain the following analogue of Theorem \ref{CNN1} for rank varieties:

\bigskip

\begin{Cor} \label{CNN3} Let $\cG$ be an infinitesimal group of height $\le r$. Suppose that $V_r(\cG)\subseteq \A^m$ is equidimensional of dimension $n \ge \frac{m+3}{2}$. If
$\zeta \in \HH^\bullet(\cG,k)$ is a non-nilpotent homogeneous element of positive degree, then $L_\zeta$ is indecomposable.\end{Cor}

\begin{proof} We denote by $\psi : \HH^\bullet(\cG,k)/\sqrt{(0)} \lra k[V_r(\cG)]$ the comorphism of $\Psi$. Owing to \cite[(1.14)]{SFB1}, the map $\psi$ is a homomorphism of
$k$-algebras which multiplies degrees by $\frac{p^r}{2}$. Let $\eta \in k[V_r(\cG)]$ be the image of the residue class $\bar{\zeta}$ of $\zeta$ under $\psi$. Then we have
\[ Z(\eta) = \Psi^{-1}(Z(\bar{\zeta})) = \Psi^{-1}(\cV_\cG(L_\zeta)) = V_r(\cG)_{L_\zeta}.\]
As before, it suffices to show that $\Proj(Z(\eta))$ is connected, and the arguments of the proof of (\ref{CNN1}) yield our assertion. \end{proof}

\bigskip
\noindent
In the special case, where our group $\cG$ has height $\le 1$, the foregoing result can be formulated in the language of restricted Lie algebras. Given a restricted Lie algebra $(\fg,[p])$, we let
\[ \cV_\fg := \{x \in \fg \ ; \ x^{[p]} = 0\}\]
be the {\it nullcone} of $\fg$. If $M$ is a $U_0(\fg)$-module, then
\[ \cV_\fg (M) := \{ x \in \cV_\fg \ ; M|_{U_0(kx)} \ \text{is not projective}\} \cup \{0\}\]
is the {\it rank variety} of $M$, see \cite{FP2,FP3}. Thanks to \cite[(1.6)]{SFB1}, we obtain:

\bigskip

\begin{Cor} \label{CNN4} Let $(\fg,[p])$ be a restricted Lie algebra. Suppose that $\cV_\fg \subseteq \A^m$ is equidimensional of dimension $n \ge \frac{m+3}{2}$. If $\zeta \in 
\HH^\bullet(U_0(\fg),k)$ is a non-nilpotent homogeneous element of positive degree, then $L_\zeta$ is indecomposable. \hfill $\square$ \end{Cor}

\bigskip
\noindent
In many cases of interest, the nullcone $\cV_\fg$ is known to be irreducible. Let $\fg := \Lie(G)$ be the Lie algebra of a smooth reductive group $G$. If $p$ is good for $G$, then
\cite[(6.3.1)]{NPV} ensures the irreducibility of $\cV_\fg$. Moreover, there is a formula expressing $\dim \cV_\fg$ in terms of the root system of $G$. For the case where $G$ is simple and
simply connected explicit formulae may be found in \cite{CLNP}. Recall that the {\it rank} $\rk(G)$ of $G$ is the dimension of any maximal torus $T\subseteq G$. By way of example, we
provide the following result:

\bigskip

\begin{Thm} \label{CNN5} Suppose that $\fg := \Lie(G)$ is the Lie algebra of a semi-simple, simply connected algebraic group $G \not \cong \SL(2)\times \SL(2)$ of rank $\rk(G) \ge 2$, 
whose Coxeter number $h$ satisfies $p \ge h$. Let $\zeta \in \HH^\bullet(U_0(\fg),k)$ be a non-nilpotent homogeneous element of positive degree. Then the following statements hold:

{\rm (1)} \ The Carlson module $L_\zeta$ is indecomposable.

{\rm (2)} \ $L_\zeta \in \ZZ[A_\infty]$ is quasi-simple. \end{Thm}

\begin{proof} In virtue of our present assumption, \cite[(6.3.1)]{NPV} implies that $\cV_\fg$ is irreducible of dimension
\[ \dim \cV_\fg = \dim_k\fg - \rk(G).\]
The latter number is $\ge \frac{\dim_k\fg +3}{2}$ in case $\rk(G) \ge 3$.  For groups of rank $2$, a case-by-case analysis yields the same conclusion, unless $G$ is of type $A_1\times A_1$.
Since $\cV_\fg \subseteq \fg$, Corollary \ref{CNN4} ensures the indecomposability of $L_\zeta$.

Observing $\dim \cV_\fg = \dim_k\fg -\rk(G) \ge 3 \rk(G) -\rk(G) \ge 4$, we obtain
\[ \dim \cV_\fg(L_\zeta) \ge 3.\]
Thanks to \cite[(2.2)]{Fa1}, the module $L_\zeta$ thus belongs to a stable AR-component of type $\ZZ[A_\infty]$, and Theorem \ref{CSM2} shows that $L_\zeta$ is actually quasi-simple. 
\end{proof}

\bigskip

\begin{Remarks} (1) If $\fg$ is as in (\ref{CNN5}) with $p>h$, then $\HH^\ast(U_0(\fg),k) = \HH^\bullet(U_0(\fg),k)$ is a reduced $k$-algebra (see \cite{AJ} or \cite{FP1}), so that the 
above result actually describes all Carlson modules of the Lie algebra $\fg$.

(2) The results of \cite{CLNP} give rise to refinements of the above Theorem. For instance, if $G = \SL(n)(k)$ and $p \ge 3$, \cite[(3.1)]{CLNP} implies that
\[ \dim \cV_{\fsl(n)} \ge \frac{n^2+2}{2},\]
unless $n=4$ and $p=3$. Hence Theorem \ref{CNN5} holds for $\SL(n)(k)$ for $p\ge 3$ and $(n,p) \ne (4,3)$. \end{Remarks}

\bigskip
\noindent
Let $(\fg,[p])$ be a restricted Lie algebra. The connection between rank varieties and support varieties is conveyed by the {\it Hochschild map}
\[ \Phi : S(\fg^\ast) \lra \HH^\bullet(U_0(\fg),k)\]
between the algebra $S(\fg^\ast)$ of polynomial functions on $\fg$ (whose variables are given degree $2$) and the even cohomology ring $\HH^\bullet(U_0(\fg),k)$, cf. 
\cite[p.~571]{Ho}. Given $f \in S(\fg^\ast)$, Friedlander and Parshall \cite[p.~560]{FP3} have shown that
\[ \cV_\fg(L_{\Phi(f)}) = Z(f)\cap \cV_\fg,\]
where $Z(f) \subseteq \fg$ denotes the set of zeros of the polynomial function $f$.

If $\fg := \Lie(G)$ is the Lie algebra of an algebraic group $G$, then $G\times k^\times$ acts on $\fg^\ast$ via
\[ (g,\alpha)\dact \eta := \alpha \, \eta \circ \Ad(g^{-1}) \ \ \ \ \ \ \forall \ (g,\alpha) \in G\times k^\times, \, \eta \in \fg^\ast.\]
Here $\Ad$ denotes the adjoint representation of $G$ on $\fg$. The $G$-invariance of $\cV_\fg$ readily yields
\[ Z((g,\alpha)\dact \eta)\cap \cV_\fg = \Ad(g)(Z(\eta)\cap \cV_\fg).\]

\bigskip

\begin{Examples}
(1) For $p \ge 3$ we consider the Lie algebra $\fsl(2)$ with its standard basis $\{e,h,f\}$, so that
\[ \cV_{\fsl(2)} = \{ah+be+cf \in \fsl(2) \ ; \ a^2 +bc = 0\}\]
is a two-dimensional irreducible variety. Thus, the pair $(\cV_\fg, \, \fg)$ violates the condition $\dim \cV_\fg \ge \frac{\dim_k\fg+3}{2}$.

Let $\eta \in \fg^\ast\setminus \{0\}$. By the proof of \cite[(2.2)]{FS}, the group $\SL(2)(k)\times k^\times$ acts on $\fsl(2)^\ast \setminus \{0\}$ with $2$ orbits, whose
representatives $\eta_1, \eta_2$ have kernels $kh\oplus ke$ and $ke \oplus kf$, respectively. Thus, $\cV_{\fsl(2)}(L_{\Phi(\eta_1)}) =ke$ while $\cV_{\fsl(2)}(L_{\Phi(\eta_2)}) =ke\cup 
kf$. In view of Lemma \ref{CSM1} and the observations above, the $3p$-dimensional Carlson modules $L_{\Phi(\eta)}$ of degree $2$ are either indecomposable, or decomposable with two 
constituents (of dimensions $p$ and $2p$, respectively). Moreover, since the Chevalley-Eilenberg cohomology groups $\HH^i(\fsl(2),k)$ vanish for $i=1,2$, the exact sequence of 
\cite[p.~575]{Ho} implies that the Hochschild map induces an isomorphism $\fsl(2)^\ast \cong \HH^2(U_0(\fsl(2)),k)$. Hence all Carlson modules of degree $2$ are of the form indicated 
above.

(2) Let $\fu := kx\oplus ky$ be the two-dimensional abelian restricted Lie algebra with trivial $p$-map. Let $f : \fu \lra k$ be the polynomial map given by $f(ax+by) = ab$. Then 
$\Proj(\cV_\fu(L_{\Phi(f)}))$ consists of two points and $L_{\Phi(f)}$ is decomposable.\end{Examples}

\bigskip

\subsection{Subsidiary Results}
In this subsection we collect a few results that will be applied in our investigation of Carlson modules corresponding to homogeneous nilpotent elements of the cohomology ring
$\HH^\ast(\cG,k)$.

\bigskip

\begin{Lem} \label{SR1} Suppose that $\zeta \in \HH^n(\cG,k)\setminus\{0\}$ has positive degree. If $L_\zeta \ne (0)$, then there exists an exact sequence
\[ (0) \lra \Omega^{1-n}_\cG(k) \lra \Omega^{-n}_\cG(L_\zeta) \lra k \lra (0).\]
\end{Lem}

\begin{proof} According to \cite[(5.9.4)]{Be2} (which only holds for $L_\zeta \ne (0)$), the element $\zeta \in \HH^n(\cG,k) \cong \Ext^1_\cG(\Omega^{n-1}_\cG(k),k)$ corresponds
to the exact sequence
\[ (0) \lra k \lra \Omega^{-1}_\cG(L_\zeta) \lra \Omega^{n-1}_\cG(k) \lra (0).\]
Application of $\Omega_\cG^{1-n}$ thus yields an exact sequence
\[ (0) \lra \Omega^{1-n}_\cG(k) \stackrel{\binom{f}{g}}{\lra} \Omega_\cG^{-n}(L_\zeta)\oplus P \stackrel{(\varphi,\psi)}{\lra} k \lra (0),\]
with a projective $\cG$-module $P$. There results a commutative diagram
\[ \begin{CD} \Omega_\cG^{1-n}(k) @>g>> P\\
@V-f VV  @V\psi VV @. \\
\Omega^{-n}_\cG(L_\zeta) @>\varphi>> \, k,
\end{CD}\]
which, by virtue of \cite[(I.5.7)]{ARS}, is both, a pull-back and a push-out diagram. If $\varphi = 0$, then
\[ \Omega_\cG^{1-n}(k) = \ker(0,\psi) = \Omega^{-n}_\cG(L_\zeta) \oplus \ker \psi.\]
Since $\Omega_\cG^{1-n}(k)$ is indecomposable, we obtain $\ker \psi = (0)$ or $\Omega^{-n}_\cG(L_\zeta) = (0)$. Since $L_\zeta \ne (0)$ is not injective, we may rule out the latter
case, whence $\ker\psi = (0)$ and $P\cong k$. This, however, implies $\Omega^n_\cG(k) = (0)$, a contradiction. Consequently, $\varphi$ is surjective and \cite[(I.5.6)]{ARS} ensures the
surjectivity of $g$. Hence $P$ is a direct summand of $\Omega^{1-n}_\cG(k)$, so that $P = (0)$. \end{proof}

\bigskip
\noindent
Recall that a $\cG$-module is endo-trivial if and only if it has constant stable Jordan type $[1]$ or $[p\!-\!1]$, see \cite[(5.6)]{CFP}.

\bigskip

\begin{Lem} \label{SR2} Suppose that $\zeta \in \HH^n(\cG,k)\setminus\{0\}$ is an element that $L_\zeta = X\oplus Y$ is the direct sum of two indecomposable endo-trivial modules. If
$\StJt(X) = \left\{ \begin{array}{cl} {[1]} & n \ \text{even} \\ {[p\!-\!1]} & n \ \text{odd} \end{array} \right.$ and $\StJt(Y) = [p\!-\!1]$, then the following statements hold:

{\rm (1)} Let $M := \Omega^{-n}_\cG(X)$ and $N := \Omega^{-n}_\cG(Y)$. There exists a commutative diagram
\[\begin{CD} @. @. (0) @. (0)  @. \\
@. @. @VVV @VVV   \\
@. @. V @= V @. \\
@. @. @VVV @VVV \\
(0) @ >>> W @>>> \Omega^{1-n}_\cG(k) @>g>> N @>>> (0)\\
@. @| @V-fVV @V\psi VV\\
(0) @ >>> W@>>> M @>\varphi >> k @>>> (0)\\
@. @. @VVV @VVV  \\
@. @. (0) @. (0) @.
\end{CD} \]
with exact rows and colums.

{\rm (2)} \ Let $\alpha_K \in \Pt(\cG)$ be a $\pi$-point. Then $\alpha_K^\ast(\varphi_K)$ is split surjective or $\alpha^\ast_K(\psi_K)$ is split surjective. \end{Lem}

\begin{proof} We shall only consider the case, where $n$ is even. The arguments for odd $n$ are analogous.

(1) Since $n$ is even, we have $\Omega^{-n}_\cG(L_\zeta) \cong M \oplus N$, with both constituents being indecomposable endo-trivial modules of constant stable Jordan types $[1]$ and
$[p\!-\!1]$, respectively. Lemma \ref{SR1} furnishes an exact sequence
\[ (0) \lra \Omega^{1-n}_\cG(k) \stackrel{\binom{f}{g}}{\lra} M\oplus N \stackrel{(\varphi,\psi)}{\lra} k \lra (0).\]
Setting $V := \ker f$ and $W := \ker g$, our claim is a consequence of \cite[(I.5.6)]{ARS}, once we know that $\varphi,\psi \ne 0$.

If $\varphi = 0$, then
\[ \Omega^{1-n}_\cG(k) = \ker(0,\psi) = M \oplus \ker \psi.\]
As $\Omega^{1-n}_\cG(k)$ is indecomposable, this implies $\ker \psi = (0)$, so $N \hookrightarrow k$, a contradiction.

If $\psi = 0$, then
\[ \Omega^{1-n}_\cG(k) = \ker(\varphi,0) = N\oplus \ker \varphi.\]
Hence $\Omega^{1-n}_\cG(k) \cong N$ and $M \cong k$, so that $\Omega^{-n}_\cG(L_\zeta) \cong \Omega^{1-n}_\cG(k)\oplus k$. As a result, $L_\zeta \cong \Omega_\cG(k) \oplus 
\Omega^n_\cG(k)$, which contradicts $L_\zeta \subseteq \Omega^n_\cG(k)$.

(2) We apply $\alpha^\ast_K$ to the exact sequence
\[ (0) \lra \Omega^{1-n}_{\cG_K}(K) \lra \Omega^{-n}_{\cG_K}((L_{\zeta})_K) \stackrel{(\varphi_K,\psi_K)}{\lra}K \lra (0)\]
and obtain an exact sequence
\[ (\ast) \ \ \ \ \ \ \ \ \ \ \ \ (0) \lra [p\!-\!1] \oplus Q \stackrel{\binom{c}{d}}{\lra} ([1]\oplus [p\!-\!1]) \oplus P \stackrel{(a,b)}\lra [1] \lra (0),\]
with projective $\fA_{p,K}$-modules $P,Q$ satisfying $\dim_KP = \dim_KQ$. Since $\fA_{p,K}$ is local, this implies $P \cong Q$. As a result, the middle term of ($\ast$) is isomorphic to the direct sum of the extreme terms, so that ($\ast$) is split exact. There thus exist $\gamma : K \lra \alpha^\ast_K(M_K)$ and $\eta :  K \lra \alpha^\ast_K(N_K)$ such that
\[ \id_K = (\alpha^\ast_K(\varphi_K),\alpha^\ast_K(\psi_K))\circ \binom{\gamma}{\eta} = \alpha^\ast_K(\varphi_K)\circ \gamma + \alpha^\ast_K(\psi_K)\circ \eta.\]
This readily implies our claim. \end{proof}

\bigskip

\subsection{Carlson modules of nilpotent elements of even degree}
In contrast to the results of Section 6.1, the structure of the Carlson modules associated to nilpotent elements of even degree does not depend on the internal structure of the underlying
group scheme $\cG$.

\bigskip

\begin{Thm} \label{CNED1} Suppose that $p \ge 3$. Let $\zeta \in \HH^{2n}(\cG,k)\setminus\{0\}$ be nilpotent. Then $L_\zeta$ is indecomposable. \end{Thm}

\begin{proof} Recall from Section \ref{S:CSM} that $L_\zeta \ne (0)$. Assume that $L_\zeta$ is decomposable. According to Lemma \ref{CSM1}, the Carlson module $L_\zeta = X\oplus Y$ 
is the direct sum of two endo-trivial modules of constant stable Jordan types $[1]$ and $[p\!-\!1]$, respectively.  We consider the associated commutative diagram of Lemma \ref{SR2} and 
note that the stable Jordan types of the modules $M$ and $N$ coincide with those of $X$ and $Y$, respectively.

Let $\alpha_K \in \Pt(\cG)$ be a $\pi$-point. If $\alpha_K^\ast(\psi_K)$ is split surjective, then
\[ \alpha^\ast_K(N_K) \cong \alpha_K^\ast(V_K) \oplus [1],\]
whence $[p\!-\!1]\oplus ({\rm proj.}) \cong \alpha_K^\ast(V_K)\oplus [1]$. Since $p \ge 3$, the trivial module $[1]$ is not a direct summand of $([p\!-\!1]\oplus ({\rm proj.}))$, and we 
have reached a contradiction. It now follows from Lemma \ref{SR2} that $\alpha_K^\ast(\varphi_K)$ is split surjective. Thus,
\[[1]\oplus ({\rm proj.}) \cong \alpha_K^\ast(W_K)\oplus [1],\]
so that $\alpha_K^\ast(W_K)$ is projective.

As a result, $\Pi(\cG)_W = \emptyset$, and \cite[(5.3)]{FPe2} ensures that $W$ is a projective $\cG$-module. The upper row of our diagram thus splits and
\[ \Omega_\cG^{1-2n}(k) \cong W \oplus N.\]
As the Heller shift is projective-free, we see that $W = (0)$, whence $N \cong \Omega^{1-2n}_\cG(k)$ and $M \cong k$. Consequently, $\Omega^{-2n}_\cG(L_\zeta) \cong M\oplus N
\cong k \oplus \Omega^{1-2n}_\cG(k)$, whence $L_\zeta \cong \Omega^{2n}_\cG(k)\oplus \Omega_\cG(k)$, a contradiction. \end{proof}

\bigskip

\begin{Cor} \label{CNED2} Suppose that $p \ge 3$ and let $\Theta_\zeta \subseteq \Gamma_s(\cG)$ be the stable Auslander-Reiten component containing an indecomposable Carlson 
module $L_\zeta$ of the nilpotent element $\zeta \in \HH^{2n}(\cG,k)\setminus\{0\}$. Then the following statements hold:

{\rm (1)} \ If $\dim \Pi(\cG) \ge 2$, then $\Theta_\zeta \cong \ZZ[A_\infty]$, $L_\zeta$ is quasi-simple and every $N \in \Theta_\zeta$ has constant Jordan type $\Jt(N) = \{\ql(N)[1]
\oplus \ql(N)[p\!-\!1] \oplus n_N[p]\}$.

{\rm (2)} \ If $\cG$ is infinitesimal, or trigonalizable and representation-infinite, then $\Theta_\zeta$ is locally split. \end{Cor}

\begin{proof} (1) According to Lemma \ref{CSM1}, the Carlson module $L_\zeta$ has constant Jordan type $\Jt(L_\zeta) = \{[1]\oplus[p\!-\!1]\oplus n_\zeta[p]\}$. Thus, 
$\Pi(\cG)_{\Theta_\zeta} = \Pi(\cG)$, and \cite[(3.3)]{Fa3} implies $\Theta_\zeta \cong \ZZ[A_\infty]$. Our assertion now follows from Corollary \ref{FI2}.

(2) If $\Theta_\zeta$ is not locally split, then $\cx_\cG(k) = \dim \Pi(\cG)+1 = 1$. If $\cG$ is infinitesimal,  then a consecutive application of \cite[(2.1)]{FV1} and \cite[(2.4)]{FV1}
shows that all simple $\cG$-modules belonging to the principal block $\cB_0(\cG)$ of $k\cG$ are one-dimensional. By virtue of \cite[(2.7)]{FV1}, $\cB_0(\cG)$ is a Nakayama algebra, and
\cite[(IV.2.10)]{ARS} now yields $\dim_k\Omega_\cG^{2n}(k) = 1$. Thus, $L_\zeta = (0)$, a contradiction. 

In case $\cG$ is trigonalizable and of infinite representation type, we consider a $\Pi$-point $\alpha_K$ such that $\Theta_\zeta$ is not $\alpha_K$-split. A consecutive application of Corollary
\ref{TG5} and Lemma \ref{CSM1} shows that the $\cG$-module $L_\zeta$ is not $\alpha_K$-projective. Theorem \ref{TG3} now implies $\ql(L_\zeta) \ge 2$, and the arguments of
Corollary \ref{CSM4} yield a contradiction. \end{proof}

\bigskip

\subsection{Carlson modules of odd degree} In order to address the case, where $L_\zeta$ is defined by an element $\zeta \in \HH^\ast(\cG,k)$ of odd degree, we recall that the map
$\alpha_K : \fA_{p,K} \lra K\cG$ induces a homomorphism $\alpha_K^\ast : \HH^\ast(\cG_K,K) \lra \HH^\ast(\fA_{p,K},K)$. The latter algebra is known: For $p>2$ we have
$\HH^\ast(\fA_{p,K},K) \cong K[X,Y]/(Y^2)$, with $\deg(X)=2$ and $\deg(Y)=1$.

Given a field extension $K\!:\!k$, we shall identify $\HH^\ast(\cG,k)$ with the subalgebra $\HH^\ast(\cG,k)\!\otimes\! 1 \subseteq \HH^\ast(\cG,k)\!\otimes_k\!K \cong 
\HH^\ast(\cG_K,K)$. Thus, we can consider $\alpha_K^\ast(\zeta) \in \HH^\ast(\fA_{p,K},K)$ for every $\zeta \in \HH^\ast(\cG,k)$ and $\alpha_K \in \Pt(\cG)$.

{\it Throughout this section, we assume that $p \ge 3$}.

\bigskip

\begin{Lem} \label{COD1} Let $n$ be odd and $\zeta \in \HH^n(\cG,k) \setminus \{0\}$.

{\rm (1)} \ We have $\Jt(L_\zeta,\alpha_K) = \left\{ \begin{array}{ll} 2[p\!-\!1] \oplus m_\zeta [p] & {\rm if} \ \alpha^\ast_K(\zeta) = 0 \\  {[p\!-\!2]} \oplus n_\zeta [p] & {\rm if}\ 
\alpha^\ast_K(\zeta) \ne 0. \end{array} \right.$

{\rm (2)} \  If $|\Jt(L_\zeta)| = 2$, then $L_\zeta$ is indecomposable.

{\rm (3)} \ If $L_\zeta$ has constant Jordan type, then $L_\zeta$ is either indecomposable or the direct sum of two endo-trivial modules of constant stable Jordan type $[p\!-\!1]$.\end{Lem}

\begin{proof} (1) Recall that $\Pi(\cG)_{L_\zeta} = \Pi(\cG)$ and let $\alpha_K \in \Pt(\cG)$ be a $\pi$-point. Since $n$ is odd, there results an exact sequence
\[ (0) \lra \alpha_K^\ast((L_\zeta)_K) \stackrel{\binom{\varphi}{-\gamma}}{\lra} [p\!-\!1]\oplus m[p] \stackrel{(f,g)}{\lra} K \lra (0),\]
where $\alpha^\ast_K(\zeta)\in \HH^n(\fA_{p,K},K)$ corresponds to $f$. If $f=0$, then $g \ne 0$ and
\[ \alpha_K^\ast((L_\zeta)_K) \cong [p\!-\!1] \oplus \ker g \cong 2[p\!-\!1] \oplus m'[p].\]
Alternatively, \cite[(I.5.6),(I.5.7)]{ARS} provides a commutative diagram
\[ \begin{CD} (0) @>>> \ker \gamma @>>>\alpha_K^\ast((L_\zeta)_K) @>\gamma >> m[p] @>>> (0) \\
@. @| @V\varphi VV  @VgVV @. \\
(0) @>>> \ker f @>>> [p\!-\!1] @>f>> K @>>> (0)
\end{CD}\]
with exact rows. Consequently,
\[ \alpha_K^\ast((L_\zeta)_K) \cong \ker \gamma \oplus m[p]\cong  \ker f \oplus  m[p] \cong [p\!-\!2] \oplus m[p],\]
as desired.

(2) Suppose that $|\Jt(L_\zeta)| = 2$, so that
\[ \Jt(L_\zeta) = \{ 2[p\!-\!1] \oplus m_\zeta [p], [p\!-\!2] \oplus n_\zeta [p]\}. \]
Let $M$ be an direct summand of $L_\zeta$, so that $\dim_kM \equiv -1,-2,0 \ \modd(p)$.

If $\dim_kM \equiv 0 \ \modd(p)$, then, for any $\pi$-point $\alpha_K$, $p$ divides the dimension of the projective-free part $X$ of $\alpha_K^\ast(M)$. As $X\!\mid\!2[p\!-\!1]$ or 
$X\!\mid\![p\!-\!2]$, it follows that  $X=(0)$. Consequently, $\alpha_K^\ast(M_K)$ is projective and, thanks to \cite[(5.3)]{FPe2}, $M$ is also projective, so that $M = (0)$.

If $\dim_kM \equiv -1 \ \modd(p)$, then $\alpha_K^\ast(M_K) \cong [p\!-\!1]\oplus n_M[p]$ for any $\alpha_K \in \Pt(\cG)$. Since there exists a $\pi$-point $\alpha_K$ with 
$\alpha_K^\ast((L_\zeta)_K) \cong [p\!-\!2]\oplus n'_\zeta[p]$, we have reached a contradiction.

Finally, if $\dim_kM \equiv -2 \ \modd(p)$, then the dimension of a direct complement $N$ of $M$ is divisible by $p$. By the first step, $N = (0)$, so that $M = L_\zeta$. As a result, 
$L_\zeta$ is indecomposable.

(3) Owing to \cite[(3.7)]{CFP}, every direct summand of $L_\zeta$ has constant Jordan type. Thus, if $L_\zeta$ is decomposable, then $\Jt(L_\zeta) = \{2[p\!-\!1]\oplus n_\zeta[p]\}$ and
each indecomposable summand $M\ne L_\zeta$ of $L_\zeta$ has dimension $\dim_kM \equiv -1 \ \modd(p)$. Consequently, such a summand has Jordan type $\Jt(M) = \{[p\!-\!1]\oplus
n_M[p]\}$, and there are exactly two such summands. By virtue of \cite[(5.6)]{CFP}, each of these summands is endo-trivial. \end{proof}

\bigskip

\begin{Remarks} (1) Let $n$ be odd and $\zeta \in \HH^n(\cG,k)\setminus\{0\}$. If $p=2$, then $\HH^\ast(\fA_{p,K},K) \cong k[X]$, so that $\alpha^\ast_K(\zeta) = 0$ for every
$\alpha_K \in \Pt(\cG)$. Consequently, we have
\[ \Jt(L_\zeta) = \{2[1]\oplus m_\zeta [2]\}\]
in that case.

(2) If $n=1$, then $L_\zeta \subseteq \Omega_\cG(k)$ is a submodule of the projective cover of $k$ and hence is either zero or indecomposable.

(3) Let $\cU$ be an abelian unipotent group scheme, $K \supseteq k$ be an extension field of $k$. General theory \cite[\S3.5]{Be1}, \cite[(I.4.27)]{Ja} provides an isomorphism
\[ \HH^\ast(\cU_K,K) \cong S(\HH^2(\cU_K,K))\otimes_k\Lambda(\HH^1(\cU_K,K))\]
of graded-commutative $K$-algebras, where $S(-)$ and $\Lambda(-)$ denote the symmetric algebra and the exterior algebra of a space, whose elements are homogeneous of the given
cohomological degree. Let $\alpha_K \in \Pt(\cU)$ be a $\pi$-point. Writing $\HH^2(\fA_{p,K},K) = KX$ and $\HH^1(\fA_{p,K},K) = KY$, we see that the graded
homomorphism $\alpha_K^\ast : \HH^\ast(\cU_K,K) \lra \HH^\ast(\fA_{p,K},K)$ annihilates $\Lambda^j(\HH^1(\cU_K,K))$ for $j\ge 2$. Consequently,
\[ \alpha_K^\ast(\HH^{2n+1}(\cU_K,K)) = \alpha_K^\ast(S^n(\HH^2(\cU_K,K))\alpha_K^\ast(\HH^1(\cU_K,K))\]
for every $n\ge 0$. \end{Remarks}

\bigskip

\begin{Examples} (1) Let $\cU = \GG_{a(1)}^3$ be the product of three copies of the first Frobenius kernel of the additive group. Then we have $\dim_k \HH^1(\cU,k) = 3$. Let $\alpha_i
: \fA_{p,k} \lra \cU$ be the $\pi$-points given by the three canonical embeddings $\GG_{a(1)} \hookrightarrow \cU$. Recall that
\[ \HH^\ast(\cU,k) = \HH^\ast(\GG_{a(1)},k)^{\otimes3}.\]
Let $\zeta_1 = \zeta\otimes 1 \otimes 1 \in \HH^1(\cU,k)$ be given by $\zeta \in \HH^1(\GG_{a(1)},k)\setminus \{0\}$. Then we have $\alpha^\ast_j(\zeta_1) = \delta_{j,1}\zeta$.
Accordingly, $|\Jt(L_{\zeta_1})| = 2$ and $L_{\zeta_1}$ is indecomposable.

(2) Let $\cU$ be as in (1), and consider $\zeta_2 \in \Lambda^3(\HH^1(\cU,k)) \setminus \{0\}$. The foregoing remarks imply
\[ \alpha_K^\ast(\zeta_2) = 0 \ \ \ \ \ \ \forall \ \alpha_K \in \Pt(\cU),\]
so that $L_{\zeta_2}$ has constant Jordan type $2[p\!-\!1]\oplus n_\zeta[p]$. We shall see in (\ref{COD3}) below, that $L_{\zeta_2}$ is indecomposable. \end{Examples}

\bigskip
\noindent
Let $\zeta \in \HH^n(\cG,k)\setminus\{0\}$ be an element of odd degree such that $L_\zeta$ is decomposable. Thanks to Lemma \ref{COD1}
\[ L_\zeta \cong X \oplus Y\]
is the direct sum of two endo-trivial modules of constant stable Jordan type $[p\!-\!1]$. We may therefore consider the commutative diagram of Lemma \ref{SR2}.

\bigskip

\begin{Lem} \label{COD2} Let $\zeta \in \HH^n(\cG,k)\setminus\{0\}$ be an element of odd degree $n\ne 1$ such that $L_\zeta$ is decomposable. Set $U := \ker (\psi\! \circ\! g)$, then 
the following statements hold:

{\rm (1)} \ $\Pi(\cG)_V \cap \Pi(\cG)_W = \emptyset$.

{\rm (2)} \ There is an exact sequence $(0) \lra W \lra U \stackrel{g}{\lra} V \lra (0)$.

{\rm (3)} \ There is an exact sequence  $(0) \lra V \lra U \stackrel{f}{\lra} W \lra (0)$.

{\rm (4)} \ We have $\Pi(\cG)_U = \Pi(\cG)_V \sqcup \Pi(\cG)_W$.

{\rm (5)} \ The map $\psi \!\circ \! g : \Omega^{1-n}_\cG(k) \lra k$ corresponds to a non-split extension
\[ (0) \lra k \lra \Omega^{-1}_\cG(U) \lra \Omega^{-n}_\cG(k) \lra (0).\]
\end{Lem}

\begin{proof} (1) Let $\alpha_K \in \Pt(\cG)$ be a $\pi$-point such that $[\alpha_K] \in \Pi(\cG)_V\cap \Pi(\cG)_W$. Note that $M = \Omega^{-1}_\cG(X)$ and $N=
\Omega^{-1}_\cG(Y)$ are $\cG$-modules of constant stable Jordan type $[1]$. By part (2) of Lemma \ref{SR2}, one of the maps $\alpha^\ast(\varphi_K)$ or $\alpha_K(\psi_K)$ is split 
surjective. By symmetry, we may assume without loss of generality, that the map $\alpha_K^\ast(\varphi_K)$ is split surjective. Consequently, the sequence
\[ (0) \lra \alpha^\ast_K(W_K) \lra \alpha^\ast_K(M_K) \lra \alpha_K^\ast(K) \lra (0)\]
splits, so that $[1]\oplus ({\rm proj.}) \cong [1] \oplus \alpha_K^\ast(W_K)$. Hence $\alpha^\ast_K(W_K)$ is projective, a contradiction.

(2),(3) We have $U = \ker (\psi\!\circ\! g) = g^{-1}(\ker\psi) = g^{-1}(V)$ as well as $U = \ker(\varphi\! \circ\! f) = f^{-1}(\ker \varphi) = f^{-1}(W)$.

(4) Owing to (2), the inclusion $\Pi(\cG)_U \subseteq \Pi(\cG)_V \cup \Pi(\cG)_W$ holds. Thanks to \cite[(3.2),(5.5)]{FPe2} and (1), the module $V\!\otimes_k\!W$ is projective, so
tensoring the exact sequence (2) with $V$ and $W$ implies
\[ U\!\otimes_k\!V \cong (V\!\otimes_k\!V) \oplus ({\rm proj.}) \ \ \text{and} \ \ U\!\otimes_k\!W \cong (W\!\otimes_k\!W) \oplus ({\rm proj.}),\]
respectively. Consequently,
\[ \Pi(\cG)_S = \Pi(\cG)_{S\otimes_kS} = \Pi(\cG)_U\cap \Pi(\cG)_S \subseteq \Pi(\cG)_U,\]
for $S = V,W$. This gives the reverse inclusion $\Pi(\cG)_U \supseteq \Pi(\cG)_V \cup \Pi(\cG)_W$, and our result now follows from (1).

(5) By general theory, the extension associated to $\psi\!\circ\! g$ is the lower row of the following diagram, whose left-hand square is a push-out:
\[ \begin{CD} (0) @>>> \Omega^{1-n}_\cG(k) @ >>> P @>>> \Omega^{-n}_\cG(k) @>>> (0) \\
                        @.      @ VV\psi\circ gV @VV\lambda V @| @. \\
			(0) @>>> k @ >>> E @>>> \Omega^{-n}_\cG(k) @>>> \,(0).
\end{CD} \]
The Snake Lemma yields $U \cong \ker \lambda$ as well as $\im \lambda = E$. If $U = (0)$, then (2) and (3) yield $V=(0)=W$, so that Lemma \ref{SR2} gives $N \cong k \cong M$.
Consequently, $L_\zeta \cong \Omega^n_\cG(k)\oplus \Omega^n_\cG(k)$, a contradiction. Since $\Omega^{1-n}_\cG(k)$ is indecomposable and $\dim_kU = \dim_k 
\Omega^{1-n}_\cG(k)-1$, it now follows that $\Soc_\cG(U) = \Soc_\cG( \Omega^{1-n}_\cG(k)) = \Soc_\cG(P)$. As a result, $P$ is an injective hull of $U$, whence $E \cong
\Omega^{-1}_\cG(U)$. \end{proof}

\bigskip
\noindent
Let $\cG$ be a finite group scheme. We say that $\cG$ is {\it linearly reductive} if the algebra $k\cG$ is semi-simple. Following Voigt \cite[(I.2.37)]{Vo}, we let $\cG_{\rm lr} \unlhd \cG$ be
the unique largest linearly reductive normal subgroup of $\cG$.

We consider the set ${\rm Max}_{\rm au}(\cG)$ of maximal abelian unipotent subgroups of $\cG$, as well as its subsets
\[  {\rm Max}_{\rm au}(\cG)_\ell := \{ \cU \in {\rm Max}_{\rm au}(\cG)\ ; \ \cx_\cU(k) \ge \ell\}\]
for every $\ell \ge 1$.

If $\cH \subseteq \cG$ is a subgroup, then the canonical inclusion $\iota : \cH \hookrightarrow \cG$ induces a continuous map $\iota_{\ast,\cH} : \Pi(\cH) \lra \Pi(\cG)$. Setting 
$\Pi(\cG)_\ell := \bigcup_{\cU \in {\rm Max}_{\rm au}(\cG)_\ell} \iota_{\ast,\cU}(\Pi(\cU))$, we define the {\it saturation rank} of $\cG$ via
\[ \srk(\cG) := \max \{\ell \ge 1\ ; \ \Pi(\cG) = \Pi(\cG)_\ell\}.\]
Note that $\srk(\cG) \le \cx_\cG(k)$. If $G$ is a finite group, then Quillen's dimension theorem \cite{AE} implies that $\srk(G)\!-\!1$ is the minimum of the dimensions of the irreducible 
components of $\Pi(G)$. This does not hold for infinitesimal group schemes, as the example of the group $\SL(2)_1$ shows.

The following Lemma is a direct consequence of the proof of \cite[(6.3)]{CFP}:

\bigskip

\begin{Lem} \label{COD} Suppose that $n<0$ and let 
\[ \fE : \ \ (0) \lra k \lra E \lra \Omega^{n-1}_\cG(k) \lra (0)\]
be a short exact sequence of $\cG$-modules. If $\alpha_K \in \Pt(\cG)$ factors through an abelian unipotent subgroup $\cU$ of complexity $\ge 2$, then the sequence 
$\alpha^\ast_K(\fE\!\otimes_k\!K)$ splits. \hfill $\square$\end{Lem} 

\bigskip

\begin{Thm} \label{COD3} Suppose that $\srk(\cG/\cG_{\rm lr}) \ge 2$. If $\zeta \in \HH^n(\cG,k)\setminus\{0\}$ has odd degree, then $L_\zeta$ is indecomposable.\end{Thm}

\begin{proof} According to \cite[(1.1)]{Fa2}, the canonical map $\pi : k\cG \lra k(\cG/\cG_{\rm lr})$ induces an isomorphism $\pi : \cB_0(\cG) \stackrel{\sim}{\lra} \cB_0(\cG/\cG_{\rm 
lr})$ between the corresponding principal blocks. Thus, $\pi^\ast : \modd \cG/\cG_{\rm lr} \lra \modd \cG$ induces an equivalence $\modd \cB_0(\cG/\cG_{\rm lr}) \stackrel{\sim}{\lra}
\modd \cB_0(\cG)$. The isomorphism $\pi^\ast : \HH^\ast(\cG/\cG_{\rm lr},k) \lra \HH^\ast(\cG,k)$ yields
\[ \pi^\ast(L_\zeta) \cong L_{\pi^\ast(\zeta)} \ \ \ \ \ \ \forall \ \zeta \in \HH^\ast(\cG/\cG_{\rm lr},k).\]
Since Carlson modules belong to the principal block, it suffices to verify our assertion under the assumption that $\srk(\cG) \ge 2$.

Assume that $L_\zeta$ is decomposable, so that $n \ge 3$. Part (5) of Lemma \ref{COD2} provides an exact sequence
\[(0) \lra k \lra \Omega^{-1}_{\cG}(U) \lra \Omega^{-n}_{\cG}(k) \lra (0).\]
Let $x$ be a point of $\Pi(\cG)$. Since $\srk(\cG) \ge 2$, there exists a $\pi$-point $\alpha_K$ representing $x$ which factors through an abelian unipotent subgroup $\cU \subseteq \cG$ 
with $\cx_\cU(k) \ge 2$. Since $1-n \le -2$, Lemma \ref{COD} guarantees that the sequence
\[ (0) \lra \alpha^\ast_K(K) \lra \alpha^\ast_K(\Omega^{-1}_{\cG_K}(U)_K) \lra \alpha^\ast_K(\Omega^{-n}_{\cG_K}(K)) \lra (0)\]
splits. Consequently,
\[ \Omega^{-1}_{\fA_{p,K}}(\alpha^\ast_K(U_K)) \oplus ({\rm proj.}) \cong [1] \oplus [p\!-\!1] \oplus ({\rm proj.}),\]
so that $\Omega^{-1}_{\fA_{p,K}}(\alpha^\ast_K(U_K)) \ne (0)$. Hence the $\fA_{p,K}$-module $\alpha^\ast_K(U_K)$ is not projective, and we conclude that $x=[\alpha_K] \in 
\Pi(\cG)_U$. Thus, $\Pi(\cG)_U = \Pi(\cG)$, and Lemma \ref{COD2}(4) implies
\[ \Pi(\cG) = \Pi(\cG)_V \sqcup \Pi(\cG)_W.\]
In view of \cite[(3.4)]{CFP} (see also \cite{Ca1}) and \cite[(5.5)]{FPe2}), it follows that one of the modules $V$ or $W$ is projective. As both of these spaces are submodules of
$\Omega^{1-n}_\cG(k)$, we obtain $V = (0)$ or $W=(0)$. In either case, Lemma \ref{SR2} implies that $\Omega^{-n}_\cG(L_\zeta) \cong \Omega^{1-n}_\cG(k)\oplus k$, whence 
$L_\zeta \cong \Omega_\cG(k)\oplus \Omega_\cG^n(k)$, which contradicts $L_\zeta \subseteq \Omega^n_\cG(k)$. \end{proof}

\bigskip

\begin{Examples} (1) Let $\fh := kx\oplus ky\oplus kz$ be the Heisenberg algebra with $p$-map defined by
\[ x^{[p]} = 0 = y^{[p]} \ \ ; \ \ z^{[p]} = z.\]
If $\cG$ is the infinitesimal group corresponding to $\fh$, then $\cG_{\rm lr}$ corresponds to the $p$-ideal $kz$. Hence $\srk(\cG/\cG_{\rm lr}) = 2$, while $\srk(\cG) = 1$.

(2) Now consider the $p$-map on $\fh$ that is defined via
\[ x^{[p]} = z = y^{[p]} \ \ ; \ \ z^{[p]} = 0.\]
Since $\cV_\fh = k(x\!-\!y)\oplus kz$ is an abelian unipotent $p$-subalgebra of complexity $2$, general theory (cf.\ \cite[p.68f]{Fa3}) implies that every $p$-point of $U_0(\fh)$ is 
equivalent to one factoring through $U_0(\cV_\fh)$. Accordingly, $\srk(\fh)=2$, and every Carlson module $L_\zeta$ of odd degree is indecomposable.\end{Examples}

\bigskip

\begin{Cor} \label{COD4} Suppose that $\zeta \in \HH^n(\cG,k)\setminus \{0\}$ has odd degree. If there exists an abelian unipotent subgroup $\cU \subseteq \cG$ such that 
$\res_\cU(\zeta) \ne 0$, then the Carlson module $L_\zeta$ is indecomposable. \end{Cor}

\begin{proof} Let $\eta \in \HH^n(\cU,k)\setminus\{0\}$ be such that $\eta = \res_\cU(\zeta)$. General theory (cf.\ \cite[p.190]{Be2}) provides a projective $\cU$-module $P$ such that
\[ L_\zeta|_\cU \cong L_\eta \oplus P.\]
If $L_\zeta$ is decomposable, then Lemma \ref{COD1} shows that $L_\zeta \cong M\oplus N$, with $M$ and $N$ having constant stable Jordan type $[p\!-\!1]$. Thus, $M|_\cU$ and 
$N|_\cU$ are non-projective $\cU$-modules, whose projective-free parts are summands of $L_\eta$. Consequently, $L_\eta$ is decomposable, and Theorem \ref{COD3} implies that $\dim 
\Pi(\cU) = 0$. 

Thanks to \cite[(14.4)]{Wa}, there exist $r_1,\ldots, r_n \in \NN$ such that $k\cU \cong k[X_1,\ldots, X_n]/(X_1^{p^{r_1}},\ldots , X_n^{p^{r_n}})$. Consequently, $n = \dim 
\cV_\cU(k) = \dim \Pi(\cU)+1 = 1$, so that $k\cU$ is a Nakayama algebra. As noted earlier, the Auslander-Reiten translation $\tau_\cU$ of the local algebra $k\cU$ coincides with the 
square of the Heller translate $\Omega_\cU$. Thus, \cite[(IV.2.10)]{ARS} yields $\Omega^n_\cU(k) \cong \Omega_\cU(k)$, implying that $\Soc_\cU(\Omega^n_\cU(k)) \cong k$ is 
simple. Hence $L_\eta \subseteq \Omega^n_\cU(k)$ is indecomposable, a contradiction.  \end{proof}

\bigskip
\noindent
We finally address the case of finite groups:

\bigskip

\begin{Cor} \label{COD5} Let $G$ be a finite group. If $\zeta \in \HH^n(G,k)\setminus\{0\}$ has odd degree, then $L_\zeta$ is indecomposable. \end{Cor}

\begin{proof} In view of Theorem \ref{COD3} and its proof, we may assume that $\srk(G)=1$.

\medskip

($\ast$) {\it We have $\cx_G(k)=1$}.

\smallskip
\noindent
Thanks to \cite[(4.2)]{FPe2}, we have
\[ \Pi(G) = \bigcup_E \iota_{\ast,E}(\Pi(E)),\]
where $E$ runs through the maximal $p$-elementary abelian subgroups of $G$. Since $\srk(G)=1$, there exists a maximal $p$-elementary abelian subgroup $E_0 \subseteq G$ such that
$\cx_{E_0}(k)=1$. Consequently, $E_0 \cong \ZZ/(p)$.

Let $P \subseteq G$ be a Sylow $p$-subgroup of $G$ containing $E_0$ and let $C(P)_p := \{ x \in C(P) \ ; \ x^p=1\}$ be the subgroup of those elements of the center $C(P)$ of $P$, 
whose order is a divisor of $p$. Then $C(P)_p$ is a non-trivial normal subgroup of $P$ and, given a maximal $p$-elementary abelian subgroup $F \subseteq P$, the group $C(P)_pF$ is 
$p$-elementary abelian, whence
\[ \{e\} \subsetneq C(P)_p \subseteq F.\]
In particular, $C(P)_p = E_0$, so that $F=E_0$.

Now let $E$ be a maximal $p$-elementary abelian subgroup of $G$. Sylow's theorem provides an element $g \in G$ such that $gEg^{-1} \subseteq P$. By the above, we conclude that 
$gEg^{-1} = E_0$, whence $\rk(E) = 1$. Quillen's dimension theorem (cf.\ \cite{AE}) now implies $\cx_G(k)=1$, as desired. \hfill $\diamond$

\medskip
\noindent
Let $P \subseteq G$ be a Sylow $p$-subgroup. Owing to \cite[(4.2.2)]{Ev}, the canonical restriction map
\[ {\rm res} : \HH^\ast(G,k) \lra \HH^\ast(P,k)\]
is injective, so that $\eta := {\rm res}(\zeta) \ne 0$. By ($\ast$), we have $\cx_G(k)=1$, and \cite[(XII.11.6)]{CE} shows that $P$ is cyclic. The assertion now follows from Corollary 
\ref{COD4}. \end{proof}

\bigskip

\begin{Remark} Suppose that $\cG$ is a finite group scheme. Let $\zeta \in \HH^n(\cG,k)\setminus\{0\}$ be an element of arbitrary positive degree such $L_\zeta$ is quasi-simple. Then 
$L_\zeta$ has exactly one successor in $\Gamma_s(\cG)$ and the middle term of the almost split sequence originating in $L_\zeta$ is either indecomposable or it possesses a non-zero 
projective summand. In the latter case, \cite[(V.5.5)]{ARS} provides a principal indecomposable $\cG$-module $P$ such that $L_\zeta \cong \Rad(P)$. Consequently, 
$\Omega^{-1}_\cG(L_\zeta)$ is simple. On the other hand, there is a short exact sequence
\[ (0) \lra k \lra \Omega^{-1}_\cG(L_\zeta) \lra \Omega_\cG^{n-1}(k) \lra (0),\]
see \cite[(5.9.4)]{Be2}. Thus, the left-hand arrow is an isomorphism, so that $\Omega_\cG^{n-1}(k) = (0)$, a contradiction. As a result, the almost split sequence originating in $L_\zeta$
has an indecomposable middle term. \end{Remark}

\bigskip

\section{Endo-trivial Modules}\label{S:ET}
Endo-trivial modules play an important r\^ole in the modular representation theory of finite groups, where they occur in connection with the study of sources of simple modules. The reader 
may consult \cite{CT} for the classification of endo-trivial modules over $p$-groups.

Thanks to \cite[(5.6)]{CFP}, a $\cG$-module $M$ is endo-trivial if and only if there exists $i \in \{1,p\!-\!1\}$ such that $M$ has constant stable Jordan type $\StJt(M) = \{[i]\}$. 
Consequently, the $\tau_\cG$-orbit of an indecomposable endo-trivial module consists entirely of endo-trivial modules.

The following immediate consequence of recent work by Dave Benson \cite{Be3} characterizes endo-trivial modules as being precisely the modules of constant Jordan type with one 
non-projective block:

\bigskip

\begin{Proposition} \label{ET1} Let $\cG$ be a finite group scheme possessing an abelian unipotent subgroup $\cU \subseteq \cG$ of complexity $\cx_\cU(k) \ge 2$. If $M \in \modd \cG$ 
has constant Jordan type $[i]\oplus n[p]$, then $i \in \{1,p\!-\!1\}$. In particular, $M$ is endo-trivial. \end{Proposition}

\begin{proof} Since every $\pi$-point of $\cU$ is also a $\pi$-point of $\cG$, it readily follows that the $\cU$-module $M|_\cU$ has constant Jordan type $[i]\oplus n[p]$.  We may 
therefore assume that $\cG = \cU$. As noted in \cite[(1.6)]{Fa3}, this implies that $k\cG = U_0(\fu)$ is the restricted enveloping algebra of an abelian $p$-unipotent restricted Lie algebra 
$\fu$. By assumption, the nullcone
\[\cV_\fu := \{ x \in \fu \ ; \ x^{[p]} = 0\}\]
is a $p$-subalgebra of dimension $\ge 2$, and $M|_{\cV_\fu}$ is a $\cV_\fu$-module of constant Jordan type $[i]\oplus n[p]$. Writing $r := \dim \cV_\fu$, we have an isomorphism
$U_0(\cV_\fu) \cong k(\ZZ/(p))^r$ of $k$-algebras, and \cite[(1.1)]{Be3} yields $i \in \{1,p\!-\!1\}$. In view of \cite[(5.6)]{CFP}, this implies that $M$ is endo-trivial. \end{proof}

\bigskip

\begin{Remarks} (1) If $G$ is a finite group, then Quillen's dimension theorem implies that the condition of the Proposition is equivalent to $\cx_G(k)\ge 2$.

(2) The infinitesimal group $\SL(2)_1$ has complexity $2$, but also affords indecomposable modules of stable constant Jordan types $[2],\ldots,[p\!-\!2]$ (see Section \ref{S:Ex}).
\end{Remarks}

\bigskip

\begin{Theorem} \label{ET2} Let $\Theta \subseteq \Gamma_s(\cG)$ be a component containing an indecomposable endo-trivial module $M_0$. If $\dim \Pi(\cG)  \ge 1$, then the 
following statements hold:

{\rm (1)} \ A $\cG$-module $M$ in $\Theta$ is endo-trivial if and only if $f_\Theta(M) = 1$.

{\rm (2)} \ If $\bar{T}_\Theta \cong \tilde{A}_{12},\, A_\infty^\infty$, then every $\cG$-module $M \in \Theta$ is endo-trivial.

{\rm (3)} \ If $\dim \Pi(\cG) \ge 2$, then $\Theta \cong \ZZ[A_\infty]$.

{\rm (4)} \ If $\Theta \cong \ZZ[A_\infty]$, then $\ql(M_0) = 1$ and $M \in \Theta$ is endo-trivial if and only if $M \cong \tau^n_\cG(M_0)$ for some $n \in \ZZ$.\end{Theorem}

\begin{proof} (1) Thanks to \cite[(5.6)]{CFP}, the module $M_0$ has constant Jordan type, so that
\[ \dim \Pi(\cG)_\Theta = \dim \Pi(\cG) \ge 1.\]
By the same token, there exist $i \in \{1,p\!-\!1\}$ and $m_0 \in \NN_0$ such that the module $M_0$ has constant Jordan type
\[ \Jt(M_0) = [i] \oplus m_0[p] .\]
Since $\dim \Pi(\cG)_\Theta \ge 1$, Lemma \ref{AFP2} implies that the component $\Theta$ is locally split. Consequently, Theorem \ref{FI1} yields $d_j^\Theta \equiv 0$ for $1 \le j \ne i
\le p\!-\!1$ and $d_i^\Theta \equiv 1$, so that
\[ \Jt(M,\alpha_K) = f_\Theta(M)[i] \oplus \alpha_{K,p}(M)[p] \ \ \ \ \ \ \forall \ M \in \Theta, \, \alpha_K \in \Pt(\cG).\]
Applying \cite[(5.6)]{CFP} again, we conclude that $M \in \Theta$ is endo-trivial exactly when $f_\Theta(M) = 1$.

(2) In this case, we have $f_\Theta \equiv 1$, so that (1) yields the assertion.

(3) In view of $\Pi(\cG) = \Pi(\cG)_\Theta$, the assertion follows directly from \cite[(3.3)]{Fa3}.

(4) If $\Theta \cong \ZZ[A_\infty]$, then $f_\Theta(M) = \ql(M)$ for every $M \in \Theta$. Hence the endo-trivial modules are the quasi-simple modules. As these form the $\tau_\cG$-orbit 
of $M_0$, our assertion follows. \end{proof}

\bigskip

\begin{Remark} Part (1) of the foregoing result implies that components of tree class $D_\infty, \tilde{D}_n, \tilde{E}_6, \tilde{E}_7$ and $\tilde{E}_8$ containing an endo-trivial module 
have $2, 4, 3,2$, and $1$ $\tau_\cG$-orbits of endo-trivial modules, respectively.\end{Remark}

\bigskip
\noindent
We turn to Carlson's construction \cite[(4.5),(9.3),(9.4)]{Ca2} of endo-trivial modules, which we shall study from the vantage point of  $\Pi$-supports and Jordan types. Let $\zeta \in 
\HH^{2n}(\cG,k)\setminus \{0\}$ for some $n \ge 1$. Given a proper summand $M|L_\zeta$, we write
\[ L_\zeta = M \oplus M'\]
and consider the push-out diagram
\begin{equation} \label{E:diag} \begin{CD} @. (0) @. (0) @.  @. \\
@. @VVV @VVV  @. \\
@. M' @= M'@. @.\\
@. @VVV @VVV  @.\\
(0) @ >>> L_\zeta @>>> \Omega^{2n}_\cG(k) @>>> k @>>> (0)\\
@. @VVV @VVV @| @.\\
(0) @ >>> M @>>> N_M @>>> k @>>> (0)\\
@. @VVV @VVV  @.\\
@. (0) @. \ (0). @.  @.
\end{CD} \end{equation}
The following application of Lemma \ref{CSM1} shows that decomposable Carlson modules often give rise to indecomposable endo-trivial modules that are not syzygies of the trivial module.

\bigskip

\begin{Theorem} \label{ET3} Suppose that $\zeta \in \HH^\bullet(\cG,k)$ is a non-nilpotent homogeneous element of positive degree. Let $M|L_\zeta$
be a proper summand. Then the following statements hold:

{\rm (1)} \ The $\cG$-module $N_M$ is endo-trivial, indecomposable, and of constant stable Jordan type $[1]$.

{\rm (2)} \ If $p \ge 3$ and $\dim \Pi(\cG) \ge 2$, then $N_M \not \cong \Omega^m_\cG(k)$ for all $m \in \ZZ$.

{\rm (3)} \ If $p \ge 3$ and $\dim \Pi(\cG) \ge 2$, then the module $N_M \in \ZZ[A_\infty]$ is quasi-simple and does not belong to the stable AR-components containing $k$ or 
$\Omega_\cG(k)$. \end{Theorem}

\begin{proof} (1) Let $\alpha_K \in \Pt(\cG)$ be a $\pi$-point. Application of the exact functor $\alpha_K^\ast$ to Diagram \ref{E:diag} yields
\[ \begin{CD} @. (0) @. (0) @.  @. \\
@. @VVV @VVV  @. \\
@. \alpha_K^\ast(M'_K) @= \alpha_K^\ast(M'_K)@. @.\\
@. @VVV @VVV  @.\\
(0) @ >>> \alpha_K^\ast((L_\zeta)_K) @>>> K\oplus ({\rm proj.})  @>>> K @>>> (0)\\
@. @VVV @VVV @| @.\\
(0) @ >>> \alpha_K^\ast(M_K) @>>> \alpha_K^\ast((N_M)_K) @>>> K @>>> (0)\\
@. @VVV @VVV  @.\\
@. (0) @. \ (0), @.  @.
\end{CD} \]
where we have used $\alpha_K^\ast(\Omega^{2n}_\cG(k)_K) \cong \Omega^{2n}_{\fA_{p,K}}(K)\oplus ({\rm proj.}) \cong K \oplus ({\rm proj.})$. Lemma \ref{CSM1} ensures that at 
least one of the $\fA_{p,K}$-modules $\alpha_K^\ast(M_K)$ and $\alpha_K^\ast(M'_K)$ is projective. If $\alpha_K^\ast(M_K)$ is projective, then the lower row is split exact and
\[ \alpha_K^\ast((N_M)_K) \cong K\oplus ({\rm proj.}).\]
Alternatively, the right-hand column splits, so that
\[ \alpha_K^\ast((N_M)_K) \oplus ({\rm proj.}) \cong K\oplus ({\rm proj.}).\]
We thus obtain $\Jt(N_M,\alpha_K) = [1]\oplus n[p]$ in either case, and \cite[(5.6)]{CFP} shows that $N_M$ is endo-trivial.

Any indecomposable summand of $N_M$ is either endo-trivial or projective, with exactly one summand being endo-trivial. The right-hand column of Diagram \ref{E:diag} shows that the 
projective summands of $N_M$ are also summands of $\Omega^{2n}_\cG(k)$ and thus are equal to zero. Consequently, the module $N_M$ is indecomposable.

(2) Since $\Jt(N_M) = [1]\oplus n[p]$ and $p \ge 3$, it follows that $N_M \not \cong \Omega_\cG^m(k)$ whenever $m \in \ZZ$ is odd.

Assume that $N_M \cong \Omega_\cG^{2m}(k)$ for some $m \in \ZZ$. As $M$ and $M'$ are both non-zero, we have $m \ne n,0$, and the right-hand exact column of Diagram \ref{E:diag}
provides an exact sequence
\[ (0) \lra M' \lra \Omega^{2n}_\cG(k) \lra \Omega^{2m}_\cG(k) \lra (0).\]
Lemma \ref{CSM1} in conjunction with \cite[(3.2)]{FPe2} implies that $\emptyset = \Pi(\cG)_{M}\cap \Pi(\cG)_{M'} = \Pi(\cG)_{M\otimes_KM'}$, so that $M\!\otimes_k\!M'$ is
projective, cf.\ \cite[(5.3)]{FPe2}. Tensoring the above sequence with $M$ thus yields an isomorphism
\[ \Omega^{2n}_\cG(M) \oplus ({\rm proj.}) \cong \Omega_\cG^{2m}(M) \oplus ({\rm proj.}).\]
As $M$ is projective-free, we obtain $\Omega_\cG^{2(n-m)}(M) \cong M$. Accordingly, the $\cG$-module $M$ is periodic, and
\[\dim \Pi(\cG)_M = \dim \cV_\cG(M)-1 = \cx_\cG(M)-1 = 0.\]
On the other hand, the lower exact row of Diagram \ref{E:diag} now reads as
\[ (0) \lra M \lra \Omega^{2m}_\cG(k) \lra k \lra (0),\]
and tensoring with $M'$ yields
\[ \Omega^{2m}_\cG(M')\oplus ({\rm proj.}) \cong M' \oplus ({\rm proj.}).\]
Since $M'$ has no projective summands, we conclude that $M'$ is periodic with $\dim \Pi(\cG)_{M'} = 0$. It follows that
\[\dim \Pi(\cG) \le \dim \Pi(\cG)_{L_\zeta}+1 = \max\{\dim \Pi(\cG)_M,\dim\Pi(\cG)_{M'}\} +1 \le 1,\]
a contradiction.

(3) Since $N_M$ is endo-trivial, we have $\Pi(\cG)_{N_M} = \Pi(\cG)$, and Theorem \ref{ET2} ensures that $N_M \in \ZZ[A_\infty]$ is quasi-simple. By the same token, the assumption that 
$N_M$ belongs to the component containing $k$ or $\Omega_\cG(k)$ implies that
\[ N_M \cong \Omega_\cG^m\circ \nu^n_\cG(k)\]
for some $m,n \in \ZZ$. The arguments of (2) now show that this cannot happen. \end{proof}

\bigskip

\begin{Remarks} (1) Note that the module $N_M$ belongs to the principal block $\cB_0(\cG)$ of $k\cG$. There are of course group schemes having endo-trivial modules that do not belong to 
the principal block: Any one-dimensional module is endo-trivial, so the simple modules of trigonalizable group schemes belong to this class. In view of \cite[(2.4)]{FV1}, the algebras of 
measures of such groups often have more than one block.

(2) Consider the abelian restricted Lie algebra $\fu := kx\oplus ky$ with trivial $p$-map. If $f \in S(\fu^\ast)$ is a homogeneous polynomial function which is not a power of an irreducible
polynomial function, then $\Proj(\cV_\fu(L_\Phi(f))) = \Proj(Z(f))$ is not connected, so that $L_{\Phi(f)}$ is decomposable. By Dade's Theorem \cite{Da1,Da2} and \cite[(5.6)]{CFP}, the
endo-trivial modules of $U_0(\fu) \cong k(\ZZ/(p))^2$ are Heller shifts of the trivial module. Hence (2) of Theorem \ref{ET3} may fail for groups $\cG$, whose $\Pi$-support $\Pi(\cG)$
has dimension $\le 1$. \end{Remarks}

\bigskip

\section{Examples}\label{S:Ex}
In their recent article \cite{CF}, the authors endow the category $\cC(\cG)$ of modules of constant Jordan type with an exact structure and study realizability problems via $K_0(\cC(\cG))$.
In this section we explicitly compute the Jordan types of the indecomposable modules for a few group schemes of tame representation type and in particular classify their indecomposable
modules of constant Jordan type. We refer the reader to \cite{Er} for the definition of tameness. For our purposes, it suffices to know that only representation-finite and tame algebras admit
(in principle) a classification of their indecomposables. As noted earlier, the stable Auslander-Reiten components belonging to blocks of finite representation type are not locally split. Hence the
next more complicated class of tame algebras serves as testing ground for the invariants of AR-components defined via $\Pi$-points.

{\it Throughout, we assume that $\Char(k)=p \ge 3$}.

\subsection{The groups {$\bf \SL(2)_1T_r$}} Consider the infinitesimal group scheme $\cG = \SL(2)_1$, that is, the first Frobenius kernel of $\SL(2)$. General theory provides an 
isomorphism $k\SL(2)_1 \cong U_0(\fsl(2))$ between the algebra of measures on $\SL(2)_1$ and the restricted enveloping algebra $U_0(\fsl(2))$ of the restricted Lie algebra $\fsl(2)$, see 
\cite[(II,\S7,no.4)]{DG}. In \cite{Pr} Premet explicitly determined the indecomposable $U_0(\fsl(2))$-modules. We shall use the interpretation of his result within the framework of 
Auslander-Reiten theory (cf.\ \cite[\S4]{Fa4}).

Recall that the algebra $k\SL(2)_1$ has $\frac{p-1}{2}$ non-simple blocks $\cB_1, \ldots, \cB_{\frac{p-1}{2}}$, with each $\cB_i$ having two simple modules, of dimensions $i$ and 
$p\!-\!i$, respectively. The stable Auslander-Reiten quiver $\Gamma_s(\cB_i)$ is the full subquiver of $\Gamma_s(\cG)$, whose vertices belong to the block $\cB_i$.

Part (2) of the following result shows in particular that the stable AR-components $\Theta \subseteq \Gamma_s(\SL(2)_1)$ with $\dim \Pi(\SL(2)_1)_\Theta = 0$ are not locally split.

\bigskip

\begin{Prop} \label{Ex1} Let $\Theta \subseteq \Gamma_s(\SL(2)_1)$ be a component.

{\rm (1)} \ If $\dim \Pi(\cG)_\Theta =1$, then $\Theta \cong \ZZ[\tilde{A}_{12}]$ and every module belonging to $\Theta$ has constant Jordan type. Moreover, there exists $s_\Theta \in 
\{1,\ldots, p\!-\!1\}$ such that
\[ \Jt(M) = \{[s_\Theta] \oplus (\frac{\dim_kM\!-\!s_\Theta}{p})[p]\}\]
for every $M \in \Theta$.

{\rm (2)} \ If $\dim \Pi(\cG)_\Theta =0$, then $\Theta \cong \ZZ[A_\infty]/\langle \tau \rangle$ and every module belonging to $\Theta$ is constantly supported. Moreover, there exists 
$i_\Theta \in \{1,\ldots,\frac{p-1}{2}\}$ such that $\Theta \subseteq \Gamma_s(\cB_{i_\Theta})$ and
\[ \Jt(M) = \{ \ql(M)[p], [i_\Theta]\oplus [p\!-\!i_\Theta] \oplus (\ql(M)\!-\!1)[p]\}\]
for every $M \in \Theta$.

{\rm (3)} \ An indecomposable $\SL(2)_1$-module $M$ is endo-trivial if and only if $M \cong \Omega_\cG^n(S)$ for $n \in \ZZ$ and $S$ simple of dimension $1$ or $p\!-\!1$.\end{Prop}

\begin{proof} By work of Drozd \cite{Dr}, Fischer \cite{Fi} and Rudakov \cite{Ru}, each of the $\frac{p-1}{2}$ non-simple blocks of $k\SL(2)_1$ is Morita equivalent to the trivial extension
${\rm Kr}\ltimes {\rm Kr}^\ast$ of the path algebra ${\rm Kr} = k[\bullet \rightrightarrows \bullet]$ of the Kronecker quiver. As a result, the Auslander-Reiten quiver $\Gamma_s(\cB)$ of 
each such block $\cB \subseteq k\SL(2)_1$ has two components of type $\ZZ[\tilde{A}_{12}]$, and infinitely many components of type $\ZZ[A_\infty]/\langle \tau \rangle$, see 
\cite[(V.3.2),(I.5.5),(I.5.6)]{Ha}.

(1) Since $\dim \Pi(\SL(2)_1)_\Theta = 1$, the above in conjunction with \cite[(3.3)]{Fa3} implies $\Theta \cong \ZZ[\tilde{A}_{12}]$. It follows from \cite[(2.4)]{We} that such a 
component contains a simple $\SL(2)_1$-module $S$. As $S$ is the restriction of an $\SL(2)$-module, \cite[(2.5)]{CFP} implies that $S$ has constant Jordan type.

Let $\{e,h,f\}$ be the standard basis of $\fsl(2)$. Recall that the simple $U_0(\fsl(2))$-modules are cyclic $f$-spaces, cf.\ \cite[p.208]{SF}. Hence $S$ has constant Jordan type $\Jt(S) =
\{[\dim_kS]\}$.

We put $s_\Theta := \dim_kS$. Corollary \ref{FI2} now yields $d_i^\Theta(\alpha_K) = \delta_{i,s(\Theta)}$ for every $\pi$-point $\alpha_K$ of $\SL(2)_1$ and $i \in 
\{1,\ldots,p\!-\!1\}$. We therefore obtain $\Jt(M) = \{[s_\Theta] \oplus (\frac{\dim_kM-s_\Theta}{p})[p]\}$ for every $M \in \Theta$.

(2) By the above, we have $\Theta \cong \ZZ[A_\infty]/\langle \tau \rangle$. Thanks to \cite[(4.1.2)]{Fa4}, there exist $g \in \SL(2)(k)$ and $a \in \{0,\ldots,p\!-\!2\}$ such that the 
unique module $M \in \Ad(g)^\ast(\Theta)$ of quasi-length $s$ is the maximal submodule $W(sp\!+\!a)$ of dimension $sp$ of the Weyl module $V(sp\!+\!a)$. Recall that 
$\cV_{\fsl(2)}(W(sp\!+\!a)) = ke$.

Let $\alpha_K : \fA_{p,K} \lra U_0(\fsl(2))_K$ be a $\pi$-point. Then there exists an element $x \in \cV_{\fsl(2)_K}$ such that $\im\alpha_K \subseteq U_0(Kx)$. We write $x = \alpha 
(e\otimes 1) +\beta(h\otimes 1) +\gamma(f\otimes 1)$ with $\beta^2 + \alpha\gamma = 0$. If $\gamma \ne 0$, then $W(sp\!+\!a)_K|_{K[x]}$ is projective, so that
$\Jt(W(sp\!+\!a),\alpha_K) = \{s[p]\}$. Alternatively, $\beta=\gamma=0$ and $\alpha_K^\ast(W(sp\!+\!a)_K) \cong W(sp\!+\!a)_K|_{K[e\otimes 1]}$, whence 
$\Jt(W(sp\!+\!a),\alpha_K)=\{[a\!+\!1] \oplus [p\!-\!a\!-\!1]\oplus (s\!-\!1)[p]\}$. Owing to \cite[(4.1.2)]{Fa4}, the component $\Ad(g)^\ast(\Theta)$ contains the baby Verma 
module with highest weight $a$. Consequently, $\Ad(g)^\ast(\Theta) \subseteq \Gamma_s(\cB_i)$, where $i := \min\{a\!+\!1,p\!-\!a\!-\!1\}$.

Since $\Jt(M) = \Jt(\Ad(g)^\ast(M))$ for any $\SL(2)_1$-module $M$, and the connected group $\SL(2)(k)$ acts trivially on the blocks of $k\SL(2)_1$, our assertion follows.

(3) Suppose that $M \cong \Omega_\cG^n(S)$, where $S$ is a simple $k\SL(2)_1$-module of dimension $s \in \{1,p\!-\!1\}$. According to (1) we have $\Jt(S) = \{[s]\}$ and
\cite[(5.6)]{CFP} implies that $S$ is endo-trivial. Being a Heller shift of an endo-trivial module, the module $M$ is also endo-trivial.

Let $M$ be an indecomposable endo-trivial module. Then $\Pi(\SL(2)_1)_M = \Pi(\SL(2)_1)$ is one-dimensional, so that $M$ belongs to a component $\Theta \cong \ZZ[\tilde{A}_{12}]$.
In view of (1) and \cite[(5.6)]{CFP}, the simple module $S \in \Theta$ has dimension $s_\Theta \in \{1,p\!-\!1\}$ and thus belongs to the principal block of $k\SL(2)_1$. Let $T$ be the 
other simple module of the principal block. The standard AR-sequence involving the projective cover $P(T)$ of $T$ has the form
\[ (0) \lra \Rad(P(T)) \lra P(T) \oplus (\Rad(P(T))/\Soc(P(T))) \lra P(T)/\Soc(P(T)) \lra (0),\]
see \cite[(V.5.5)]{ARS}. Since $\Rad(P(T))/\Soc(P(T)) \cong S\oplus S$ (cf.\ \cite[Thm.3]{Po}), we conclude that $S$ and $\Omega_\cG(T)$ are representatives of the two 
$\tau_\cG$-orbits of $\Theta$ (cf.\ \cite[(IV.3.8.3)]{Er}). As $k\SL(2)_1$ is symmetric, it follows that $\{\Omega_\cG^{2n}(S),\Omega_\cG^{2n+1}(T) \ ; \ n \in \ZZ\}$ is the set of 
vertices of $\Theta$. Consequently, there exists a simple $k\SL(2)_1$-module $S$ of dimension $\dim_kS \in \{1,p\!-\!1\}$ and $n \in \ZZ$ with $M \cong \Omega^n_\cG(S)$. \end{proof}

\bigskip

\begin{Remarks} (1) Part (1) of Proposition \ref{Ex1} shows that Benson's result (\ref{ET1}) may fail for group schemes of complexity $\ge 2$.

(2) Since the the principal block $\cB_0(\SL(2)_1) \subseteq k\SL(2)_1$ is Morita equivalent to the trivial extenion of the Kronecker algebra, it follows that $\HH^n(\SL(2)_1,k) = (0)$, 
whenever $n$ is odd. Let $n$ be even and $\zeta \in \HH^n(\SL(2)_1,k)\setminus\{0\}$ be nilpotent. Thanks to Theorem \ref{CNED1}, the Carlson module $L_\zeta$ is indecomposable of 
constant Jordan type $\Jt(L_\zeta) = \{[1]\oplus[p\!-\!1]\oplus n_\zeta[p]\}$ (cf.\ Lemma \ref{CSM1}). As this contradicts Proposition \ref{Ex1}(1), we have retrieved the well-known fact 
that the algebra $\HH^\ast(\SL(2)_1,k) = \HH^\bullet(\SL(2)_1,k)$ is reduced.

(3) Let $\Theta \subseteq \Gamma_s(\SL(2)_1)$ be a component with zero-dimensional $\Pi$-support. If $\alpha_K$ is a $\pi$-point such that $[\alpha_K] \in \Pi(\SL(2)_1)_\Theta$, then 
Proposition \ref{Ex1} shows that $\ell(\alpha_{K,i}) \le 2$ for $1\le i \le p\!-\!1$. One can show that the quasi-simple module of $\Theta$ is the only relatively $\alpha_K$-projective module
belonging to $\Theta$. \end{Remarks}

\bigskip
\noindent
Let $T \subseteq \SL(2)$ be the standard maximal torus of diagonal matrices. Our next example, $\cG = \SL(2)_1T_r$ ($r\ge 2$), the product of the first Frobenius kernel of $\SL(2)$ and the
r-th Frobenius kernel of $T$, is closely related to the previous one. In fact, $k\SL(2)_1T_r$ is a Galois extension of $k\SL(2)_1$ with Galois group $\ZZ/(p^{r-1})$. According to
\cite[(5.5)]{FV2}, the groups $\SL(2)_1T_r$ are precisely the semi-simple infinitesimal groups whose principal blocks have tame representation type.

By the proof of \cite[(5.5)]{FV2}, the algebra $k\SL(2)_1T_r$ has non-simple blocks $\cB_1,\ldots, \cB_{\frac{p-1}{2}}$ with $\cB_i$ possessing $p^{r-1}$ simple modules of dimensions 
$i$ and $p\!-\!i$, respectively.

\bigskip

\begin{Cor} \label{Ex2} Let $\Theta \subseteq \Gamma_s(\SL(2)_1T_r)$ be a component. Then the following statements hold:

{\rm (1)} \ If $\dim \Pi(\SL(2)_1T_r)_\Theta = 1$, then $\Theta \cong \ZZ[\tilde{A}_{p^{r-1},p^{r-1}}]$, and there exists $s_\Theta \in \{1,\ldots, p\!-\!1\}$ such that
\[\Jt(M) = \{[s_\Theta] \oplus (\frac{\dim_kM\!-\!s_\Theta}{p})[p]\} \]
for every $M \in \Theta$.

{\rm (2)} \ If $\dim \Pi(\SL(2)_1T_r)_\Theta = 0$, then $\Theta \cong \ZZ[A_\infty]/\langle \tau \rangle, \, \ZZ[A_\infty]/\langle \tau^{p^{r-1}} \rangle$, and there exists $i_\Theta \in
\{1,\ldots,\frac{p-1}{2}\}$ such that $\Theta \subseteq \Gamma_s(\cB_{i_\Theta})$ and
\[ \Jt(M) = \{ (\frac{\dim_kM}{p})[p], [i_\Theta]\oplus [p\!-\!i_\Theta] \oplus (\frac{\dim_kM}{p}\!-\!1)[p]\}\]
for every $M \in \Theta$.

{\rm (3)} \ An indecomposable $\SL(2)_1T_r$-module $M$ is endo-trivial if and only if $M \cong \Omega_{\SL(2)_1T_r}^n(S)$ for $n \in \ZZ$ and $S$ simple of dimension $1$ or 
$p\!-\!1$.\end{Cor}

\begin{proof} According to \cite[(5.6)]{FV2}, the stable Auslander-Reiten quiver of each non-simple block $\cB \subseteq k\SL(2)_1T_r$ has two components of type
$\ZZ[\tilde{A}_{p^{r-1},p^{r-1}}]$, four components of type $\ZZ[A_\infty]/\langle \tau^{p^{r-1}} \rangle$, and infinitely many components of type $\ZZ[A_\infty]/\langle \tau
\rangle$.

Let $\alpha_K : \fA_{p,K} \lra K\SL(2)_1T_r$ be a $\pi$-point. Then there exists an abelian, unipotent subgroup $\cU \subseteq (\SL(2)_1T_r)_K$ such that $\im \alpha_K \subseteq
K\cU$. Since the group $(\SL(2)_1T_r)_K/(\SL(2)_1)_K \cong (T_{r-1})_K$ is diagonalizable, we obtain $\cU \subseteq (\SL(2)_1)_K$. Accordingly, we have
\[ (\ast) \ \ \ \ \ \ \ \ \ \ \alpha_K^\ast(M_K) = \alpha_K^\ast(M_K|_{(\SL(2)_1)_K})\]
for every $\SL(2)_1T_r$-module $M$.

(1) If $\dim \Pi(\SL(2)_1T_r)_\Theta = 1$ and $M \in \Theta$, then \cite[(2.1.2)]{Fa4} ensures that $M|_{\SL(2)_1}$ belongs to a component with one-dimensional $\Pi$-support.
Hence Proposition \ref{Ex1} in conjunction with ($\ast$) provides $s \in \{1,\ldots,p\!-\!1\}$ such that
\[\Jt(M) = \{[s] \oplus (\frac{\dim_kM\!-\!s}{p})[p]\}.\]
By the above, $\Theta \cong  \ZZ[\tilde{A}_{p^{r-1},p^{r-1}}]$ has tree class $A_\infty^\infty$, so that our assertion follows from Corollary \ref{FI2}.

(2) Assume that $\dim \Pi(\SL(2)_1T_r)_\Theta = 0$. Let $M_0 \in \Theta$. Then the restriction $M_0|_{\SL(2)_1}$ belongs to a homogeneous tube $\ZZ[A_\infty]/\langle \tau \rangle$, 
so that \cite[(4.1.2)]{Fa4},(\ref{Ex1}) and ($\ast$) imply
\begin{itemize}
\item  \ $\dim_k M_0 = p\ql(M_0|_{\SL(2)_1})$, and
\item \ $\Jt(M_0) = \{ \ql(M_0|_{\SL(2)_1})[p], [i]\oplus [p\!-\!i] \oplus (\ql(M_0|_{\SL(2)_1})\!-\!1)[p]\}$, where $M_0|_{\SL(2)_1}$ belongs to the block $\cB_i \subseteq 
k\SL(2)_1$.
\end{itemize}
\noindent
Since the restrictions of the simple $\SL(2)_1T_r$-modules are simple (cf.\ \cite[(5.1)]{FV2}), we conclude that $M_0$ belongs to $\cB_i \subseteq k\SL(2)_1T_r$. Hence any $M \in \Theta$ 
will yield the same data for $\Jt(M)$.

(3) Let $M$ be an indecomposable, endo-trivial $\SL(2)_1T_r$-module. Owing to \cite[(2.1.2)]{Fa4}, the module $M|_{\SL(2)_1}$ is indecomposable. Moreover, a two-fold application
of \cite[(5.6)]{CFP} in conjunction with ($\ast$) implies that $M|_{\SL(2)_1}$ is endo-trivial. Proposition \ref{Ex1} now provides $n \in \ZZ$ such that 
$\Omega^n_{\SL(2)_1}(M|_{\SL(2)_1})$ is a simple $\SL(2)_1$-module of dimension $1$ or $p\!-\!1$. Since $\Omega^n_{\SL(2)_1T_r}(M)|_{\SL(2)_1}$ is indecomposable, it follows that  $\Omega^n_{\SL(2)_1T_r}(M)|_{\SL(2)_1} \cong \Omega^n_{\SL(2)_1}(M|_{\SL(2)_1})$ is also simple. Consequently, $\Omega^n_{\SL(2)_1T_r}(M)$ is simple.

If $S$ is a simple $\SL(2)_1T_r$-module of dimension $1$ or $p\!-\!1$, then $S|_{\SL(2)_1}$ has the same properties (see \cite[(5.1)]{FV2}), so that (1) and \cite[(5.6)]{CFP} imply that
$S$ is endo-trivial. Consequently, $\Omega_{\SL(2)_1T_r}^n(S)$ is also endo-trivial. \end{proof}

\bigskip

\begin{Remarks} (1) The foregoing results indicate the utility of $\pi$-points. In either case, the components with one-dimensional $\Pi$-supports cannot be distinguished via their support 
varieties, yet their Jordan types nicely separate them (cf.\ Corollary \ref{SJ3}). On the other hand, Jordan types do not reflect the ranks of the tubes occurring in (\ref{Ex2}).

(2) In each of the cases above, the indecomposable modules are either projective, constantly supported or of constant Jordan type. It is therefore possible to compute $\Jt(M)$ for any 
$\cG$-module $M$, whose indecomposable constituents are known. \end{Remarks}

\bigskip
\subsection{A central extension of $\bf \fsl(2)$}
Our next example necessitates a simple technical preparation. Recall that $\fA_{p,K} = K[t]$, where $t^p =0$. Given $j \in \{1,\ldots, p\}$ and $1 \le \ell \le \lfloor\frac{p}{j}\rfloor$, the unique $\ell$-dimensional indecomposable module of the local subalgebra $K[t^j] \subseteq \fA_{p,K}$ will also be denoted $[\ell]$. (Here $\lfloor\,\rfloor$ denotes the floor function.)

\bigskip

\begin{Lem} \label{Ex3} Let $i,j \in \{1,\ldots,p\}$. Then
\[ [i]|_{K[t^j]} = \left\{ \begin{array}{cc} (j\!-\!r)[a]\oplus r[a\!+\!1] & {\rm if} \ i =aj\!+\!r \ {\rm with} \ 0\le r <j\le i\\
                                        i[1] & {\rm otherwise}.
\end{array} \right.\]
\end{Lem}

\begin{proof} Let $\{v_1,\ldots,v_i\}$ be a basis of the cyclic $\fA_{p,K}$-module $[i]$ such that $t.v_\ell = v_{\ell+1}\ (v_{i+1} = 0)$. If $j\ge i$,
then $t^j$ acts trivially on $[i]$, whence $[i]|_{K[t^j]} = i[1]$. Alternatively, we have for $1 \le \ell \le j$
\[ K[t^j]v_\ell = \sum_{q=0}^a Kv_{v_{qj+\ell}} =  \left\{ \begin{array}{cc} \bigoplus_{q=0}^a Kv_{v_{qj+\ell}}& 1 \le \ell \le r \\ \bigoplus_{q=0}^{a-1} Kv_{v_{qj+\ell}} &
r\!+\!1 \le \ell \le j. \end{array} \right.\]
Hence $[i]|_{K[t^j]} = (j\!-\!r)[a]\oplus r[a\!+\!1]$. \end{proof}

\bigskip
\noindent
We consider the $4$-dimensional restricted Lie algebra $\fsl(2)_s := \fsl(2)\oplus kv_0$, whose bracket and $p$-map are given by
\[ [(x,\alpha v_0), (y,\beta v_0)] = ([x,y],0) \ \ \text{and} \ \ (x,\alpha v_0)^{[p]} = (x^{[p]},\psi(x)v_0)\]
respectively, where the $p$-semilinear map $\psi : \fsl(2) \lra k$ satisfies $\psi(e) = 0 = \psi(f)$ and $\psi(h) = 1$. Note that the nullcone
\[ \cV_{\fsl(2)_s} = (\cV_{\fsl(2)}\cap \ker \psi)\times kv_0 = (ke\oplus kv_0)\cup (kf\oplus kv_0)\]
is $2$-dimensional and reducible.

Since $kv_0$ is a unipotent $p$-ideal of $\fsl(2)_s$, the simple $U_0(\fsl(2)_s)$-modules are just the pull-backs of the simple $U_0(\fsl(2))$-modules. Moreover, the $p$-dimensional simple 
$U_0(\fsl(2)_s)$-module belongs to a block which is a Morita equivalent to a truncated polynomial ring $k[X]/(X^n)$ (a Nakayama algebra).

According to \cite[\S7]{FS}, each stable AR-component $\Theta \subseteq \Gamma_s(\fsl(2)_s)$ with $2$-dimensional support is isomorphic to $\ZZ[A_\infty^\infty]$, so that $f_\Theta 
\equiv 1$ (cf.\ (\ref{FI2})).  We shall consider two types of these components. One type contains a simple module, the others are given by certain induced modules.

Recall that $\{e,h,f\}$ is the standard basis of $\fsl(2)$. Let $\lambda : kh\oplus ke \oplus kv_0 \lra k$ be a linear form such that $ \lambda(h) \in \{0,\ldots,p\!-\!2\}$ and $\lambda(e) = 
0 =\lambda(v_0)$. Observing that $\fb_s := kh\oplus ke \oplus kv_0$ is a $p$-subalgebra of $\fsl(2)_s$, we consider the ``Verma module"
\[Z(\lambda) := U_0(\fsl(2)_s)\!\otimes_{U_0(\fb_s)}\!k_\lambda.\]
Owing to \cite[(7.3)]{FS}, we have $ke\oplus kv_0 \subseteq \cV_{\fsl(2)_s}(Z(\lambda))$, while general results on induced modules (see for instance \cite[(4.12)]{FPS}) guarantee the reverse inclusion. Thus, $Z(\lambda)$ is not of constant Jordan type and belongs to a component $\Theta_{(\lambda)}\cong \ZZ[A_\infty^\infty]$.

\bigskip

\begin{Prop} \label{Ex4} The following statements hold:

{\rm (1)} \ Let $\Theta_S \subseteq \Gamma_s(\fsl(2)_s)$ be the component containing the simple module $S$. If $n :=\dim_kS \le p\!-\!1$, then $\Theta_S \cong \ZZ[A_\infty^\infty]$ 
and
\[ \Jt(M) = \{(j\!-\!r)[a]\oplus r[a\!+\!1] \oplus (\frac{\dim_kM\!-\!n}{p})[p] \ ; \  n = aj\!+\!r,\, 0 \le r < j \le n \}\]
for every $M \in \Theta_S$. In particular, $|\Jt(M)| = n$ for every $M \in \Theta_S$.

{\rm (2)} \ Let $\Theta_{(\lambda)} \subseteq \Gamma_s(\fsl(2)_s)$ be the component containing $Z(\lambda)$. Then $\Theta_{(\lambda)} \cong \ZZ[A_\infty^\infty]$ and setting $i :=
\min \{\lambda(h)\!+\!1,p\!-\!\lambda(h)\!-\!1\}$ we have
\begin{eqnarray*}
\StJt(M) & = & \{(j\!-\!r)[a]\oplus r[a\!+\!1] \ ;  \  p = aj\!+\!r, \, 1 \le r < j \le p \} \\
            &\cup & \{  i[1] \oplus (j\!-\!r)[a]\oplus r[a\!+\!1] \ ; \  p\!-\!i = aj\!+\!r, \, 0 \le r < j , \, i \le j \le p\!-\!i\} \\
	    & \cup & \{  (j\!-\!r')[b] \oplus r'[b\!+\!1] \oplus(j\!-\!r)[a] \oplus  r[a\!+\!1]\ ; \ i = bj\! +\!r', \, p\!-\!i = aj\!+\!r,\\
	    & &  \, 0 \le r,r' < j <i\}
\end{eqnarray*}
for every $M \in \Theta_{(\lambda)}$.

{\rm (3)} \ Every $M \in \Theta_k \cup \Theta_{\Omega_{\fsl(2)_s}(k)}$ is endo-trivial. \end{Prop}

\begin{proof} (1) The central, nilpotent element $v_0 \in U_0(\fsl(2)_s)$ acts trivially on every simple $U_0(\fsl(2)_s)$-module $S$. Let $\alpha_K : \fA_{p,K} \lra U_0(\fsl(2)_s)_K$ be a
$\pi$-point. Then there exists a $p$-unipotent subalgebra $\fu \subseteq (\fsl(2)_s)_K$ with $\im \alpha_K \subseteq U_0(\fu)$. Hence we can find $x \in \fsl(2)_K$ with $x^{[p]} \in
K(v_0\!\otimes\!1)$ and $\fu \subseteq Kx\oplus Kx^{[p]}$. If $x^{[p]} \ne 0$, then the condition $\alpha_K(t)^p = 0$ entails $\im \alpha_K \subseteq 
(v_0\!\otimes\!1)U_0(\fsl(2)_s)_K$, so that $\Jt(S,\alpha_K) = n[1]$. Alternatively, $x \in \cV_{\fsl(2)_s}$ and $\fu \subseteq K(e\!\otimes\!1)\oplus K(v_0\!\otimes\!1)$ or $\fu 
\subseteq K(f\!\otimes\!1) \oplus K(v_0\!\otimes\!1)$. Both cases being similar, we write $\alpha_K(t) = \sum_{i,j=0}^{p-1} \gamma_{i,j}(e\!\otimes\!1)^i(v_0\!\otimes\!1)^j$ 
($\gamma_{0,0} = 0$) and note that $\alpha_K(t)w = (\sum_{i=1}^{p-1} \gamma_{i,0}(e\!\otimes\!1)^i).w \ \ \forall \ w \in S_K$. Since $S|_{K[e\otimes1]} = [n]$, Lemma 
\ref{Ex3} in conjunction with \cite[(2.2)]{FPe1} gives
\[ \Jt(S)= \{(j\!-\!r)[a] \oplus r[a\!+\!1] \ ; \ 1\le j \le n, \, n = aj\!+\!r,\, 0 \le r \le j\!-\!1\}.\]
As $\Theta_S \cong \ZZ[A_\infty^\infty]$, Corollary \ref{FI2} yields the first assertion. Moreover, each Jordan type of $S$ is uniquely determined by $j \in \{1,\ldots,n\}$, whence
$|\Jt(S)| = n$. Thanks to Corollary \ref{SJ1}, this implies $|\Jt(M)| = n$ for every $M \in \Theta_S$.

(2)  Let $\alpha_K : \fA_{p,K} \lra U_0(\fsl(2)_s)_K$ be a $\pi$-point, $\fu \subseteq (\fsl(2)_s)_K$ be a $p$-unipotent subalgebra through which $\alpha_K$ factors. Since $v_0$ acts
trivially on $Z(\lambda)$, the arguments of part (1) imply $\Jt(Z(\lambda),\alpha_K) = p[1]$ whenever $\fu^{[p]} \ne (0)$. By the same token, the remaining cases are $\fu \subseteq
K(e\!\otimes\!1)\oplus K(v_0\!\otimes\!1)$ and $\fu \subseteq K(f\!\otimes\!1) \oplus K(v_0\!\otimes\!1)$.

Since $Z(\lambda)|_{K[f\otimes 1]} = [p]$, Lemma \ref{Ex3} implies
\[ \Jt(Z(\lambda),\alpha_K) \in \{(j\!-\!r)[a] \oplus r[a\!+\!1] \ ; \  p = aj\!+\!r, \, 0 \le r < j \le p \}\]
in the latter case.

In view of $Z(\lambda)|_{K[e\otimes 1]} = [i]\oplus [p\!-\!i]$ and $i < p\!-\!i$, the assumption $\im \alpha_K \subseteq K(e\!\otimes\!1)\oplus K(v_0\!\otimes\!1)$ in conjunction with 
Lemma \ref{Ex3} gives
\begin{eqnarray*}
\Jt(Z(\lambda),\alpha_K) & \in & \{ i[1] \oplus (j\!-\!r)[a] \oplus r[a\!+\!1] \ ; \ p\!-\!i = aj\!+\!r, \, 0 \le r < j , \, i \le j \le p\!-\!i\}\\
 &\cup &\{ (j\!-\!r')[b]\oplus r'[b\!+\!1]\oplus (j\!-\!r)[a] \oplus r[a\!+\!1] \ ; \  i = bj\!+\!r', \, p\!-\!i = aj\!+\!r,  \\
 & & 0 \le r,r' < j <i\}.
 \end{eqnarray*}
Owing to \cite[(2.2)]{FPe1}, all types in the given sets actually occur, so that we have in fact determined $\Jt(Z(\lambda))$. Since $\Theta_{(\lambda)} \cong \ZZ[A_\infty^\infty]$, 
Corollary \ref{FI2} gives the result.

(3) Since the trivial module $k$ and its Heller shift $\Omega_{\fsl(2)_s}(k)$ are endo-trivial, the assertion is a direct consequence of Theorem \ref{ET2}(2). \end{proof}

\bigskip

\begin{Remarks} (1) In (\ref{Ex1}) and (\ref{Ex2}), the indecomposable periodic modules are constantly supported. This is in general not the case. Consider the group scheme $\GG_{a(2)}$ 
with algebra of measures $k\GG_{a(2)} = k[u_0,u_1]$, where $k\GG_{a(1)} = k[u_0]$. Then $M := k\GG_{a(2)}\!\otimes_{k[u_0]}\!k$ is indecomposable and periodic. We consider the 
$\pi$-points $\alpha_k, \beta_k : \fA_p \lra k\GG_{a(2)}$, given by $\alpha_k(t) = u_0$ and $\beta_k(t) = u_0+u_1^2$. Then \cite[(2.2)]{FPe1} implies $\alpha_k \sim \beta_k$, while Lemma \ref{Ex3} yields
\[ \Jt(M,\alpha_k) = p[1] \ \ \text{and} \ \ \Jt(M,\beta_k) = [\frac{p\!-\!1}{2}]\oplus [ \frac{p\!+\!1}{2}],\]
so that $\{[p],p[1], [\frac{p-1}{2}]\oplus [ \frac{p+1}{2}]\} \subseteq \Jt(M)$.

Also, since $k[u_0] = k\GG_{a(1)} \subseteq k\GG_{a(2)}$, the module $M|_{\GG_{a(1)}}$ has constant Jordan type. Thus, $\cG$-modules $M$, whose restrictions to a subgroup $\cH
\subseteq \cG$ are of constant Jordan type with $\Pi(\cG)_M = \Pi(\cH)$ are not necessarily constantly supported.

(2) Note that the modules belonging to the components $\Theta_S$ are usually not annihilated by $v_0$. In fact, direct computation shows that $\Omega^2_{\fsl(2)_s}(S)$ is too big to be
annihilated by $v_0$.  Thus, our method of computing $\Jt(S)$ does not carry over to all modules of $\Theta_S$. \end{Remarks}

\bigskip

\bigskip

\begin{center}

\bf Acknowledgement

\end{center}
Parts of this paper were written while the author was visiting the Isaac Newton Institute in Cambridge. He would like to take this opportunity to thank the members of the Institute for their
hospitality and support.

\end{document}